\documentclass{amsart}
\usepackage{amsmath,amstext,amssymb,amsfonts,amscd,graphicx,cases}
\usepackage{mathrsfs,accents,mathtools,amsthm,bbm}
\usepackage{xcolor,enumerate}
\newtheorem{thm}{Theorem}[section]
\newtheorem{cor}[thm]{Corollary}
\newtheorem{lem}[thm]{Lemma}
\newtheorem{prop}[thm]{Proposition}
\newtheorem*{thm*}{Theorem 1.1}
\newtheorem*{lem*}{Lemma 2.1}
\theoremstyle{definition}
\newtheorem{defn}[thm]{Definition}
\theoremstyle{remark}
\newtheorem{rem}[thm]{Remark}
\numberwithin{equation}{section}

\newcommand{\lr}[1]{\langle #1 \rangle}
\newcommand{\mt}[1]{\mathfrak #1 }

\newcommand{\R}{\mathbb{R}}

\newcommand{\mbq}{\mathbbm{q}}
\newcommand{\mbr}{\mathbbm{r}}
\newcommand{\mbp}{\mathbbm{p}}
\newcommand{\mcH}{\mathcal{H}}

\newcommand{\tmbq}{\tilde{\mathbbm{q}}}
\newcommand{\tmbr}{\tilde{\mathbbm{r}}}
\newcommand{\tmbp}{\tilde{\mathbbm{p}}}

\title[$L^p$ estimate for the gain term]
{The $L^p$ estimate for the gain term of the Boltzmann collision operator and its application}

\author{Ling-Bing He}
\address{Department of Mathematical Sciences, TsingHua University, Beijing, 100084, P.R.China}
\email{hlb@tsinghua.edu.cn}

\author{Jin-Cheng Jiang}
\address{Department of Mathematics, National Tsing Hua University, Hsinchu, Taiwan 300040, R.O.C.
\\, and  National Center for Theoretical Sciences, National Taiwan University, Taipei, Taiwan 100046, R.O.C.}

\email{jcjiang@math.nthu.edu.tw}

\author[Kuo]{Hung-Wen Kuo}
\address{Department of Mathematics, National Cheng Kung University, Tainan 701401, Taiwan
, R.O.C.}
\email{hwkuo@mail.ncku.edu.tw}

\author{Meng-Hao Liang}
\address{Department of Mathematics, National Tsing Hua University, Hsinchu, Taiwan 300040, R.O.C.}
\email{a0970352580@gmail.com}

\begin{document}

\begin{abstract}
We prove the  Hardy-Littlewood-Sobolev type $L^p$ estimates for the gain term of the Boltzmann 
collision operator including Maxwellian molecule, hard potential and hard sphere models. 
Combining with the results of  Alonso et al.~\cite{ACG10} for the soft potential and Maxwellian molecule 
models, we provide an unified  form of $L^p$ estimates for all cutoff models which 
are sharp in the sense of scaling. The most striking feature of our new estimates for the hard potential 
and hard sphere models is that they do not increase the moment, the same as Maxwellian molecule 
and soft potential models.  
Based on these novelties, 
we prove the global existence and scattering of the non-negative 
unique mild solution for the Cauchy problem of the Boltzmann equation when 
the positive initial data is small in the weighted $L^3_{x,v}$ space. 
\end{abstract}

\maketitle

\section{Introduction and Main results}

We consider the Boltzmann equation, 
\[
{\partial_t f}+v\cdot\nabla_x f =Q(f,f),
\]
in $(0,\infty)\times\mathbb{R}^3\times\mathbb{R}^3$
where the collision operator $Q(f,f)$ is given by 
\[
Q(f,f)(v)=\int_{\mathbb{R}^3}\int_{\omega\in S^{2}_+}(f'f'_{*}-ff_{*}) B(v-v_{*},\omega)\; d\omega dv_{*},
\]
and $d\omega$ is the solid element in the direction of unit vector $\omega$ in the upper hemisphere $S^2_+$.
We use the abbreviations $f'=f(x,v',t)\;,\;f'_{*}=f(x,v'_{*},t)\;,\;f_{*}=f(x,v_{*},t)$, 
and the relation between the pre-collisional velocities of particles and
after collision is given by
\[
v'=v-[\omega\cdot(v-v_{*})]\omega\;,\;v'_{*}=v_{*}+[\omega\cdot(v-v_{*})]\omega\;,\;\omega\in S^{2}_+.
\]

In this paper, we consider the cutoff models, including Maxwellian molecule, hard potential 
and hard sphere models.
Namely, the collision kernel $B$  being the product of kinetic part and angular part, 
\begin{equation}\label{D:kernel}
B(v-v_*,\omega)=|v-v_*|^{\gamma}\;b(\cos\theta)=|v-v_*|^{\gamma}\;\cos\theta ,
\end{equation}
where 
\[
0\leq \gamma \leq 1\;,\;\cos\theta=\lr{\omega,(v-v_*)/|v-v_*|}\;,0\leq\theta\leq \pi/2.
\]
The angular function $b(\cos\theta)=\cos\theta$ satisfying  the Grad's cutoff assumption 
\begin{equation}\label{D:Grad}
\int_{S^{2}_+} b(\cos\theta)d\omega<\infty
\end{equation}
is given in the explicit form for the simplicity of representation and the result of this work still holds if 
it is replaced by smooth function with same decay when $\theta$ towards $0$ and $\pi/2$. 
The cases $-3<\gamma<0,\;\gamma=0,\;0<\gamma<1$ and $\gamma=1$ are called the sot potential, 
Maxwell molecules, hard potential and hard sphere model respectively. When the cutoff condition~\eqref{D:Grad}
is satisfied, the collision operator $Q$ can be split into the gain term $Q^{+}$ and the loss term $Q^{-}$.    
It is useful to introduce the bilinear gain and loss terms
\begin{equation}\label{gain-loss-def}
\begin{split}
 Q^+(f,g)(v)=\int_{{\R}^3}\int_{S^{2}_+}  f(v')g(v_*') B(v-v_*,\omega)
d\omega dv_*, \\
 Q^{-}(f,g)(v)=\int_{{\R}^3}\int_{S^{2}_+}  f(v)g(v_*) B(v-v_*,\omega)
d\omega dv_*.
\end{split}
\end{equation}

In order to study the Boltzmann equation, it is essential to understand the properties of the collision operator  
and which is clearly non-trivial due to the complex structure of the operator. 
In particular, the estimate for the gain term is the core part of the task. 
Among all, the regularization property of the gain term  found first by 
Lions~\cite{Lio94} was used to prove the existence of renormalized solutions. 
In~\cite{Lio94}, Lions assumed that both the relative velocity and angular function are compact 
supported and the support of the later is away from $0$ and $\pi/2$. It was extended to the full 
collision kernel by several authors~\cite{JC12,JC20,Lu98,Mou04,Wen94} in different forms for various purposes. 
For example the sharp regularizing estimates, for $0\leq \gamma\leq 1$, 
given in~\cite{JC20}  take the form:
\begin{equation}\label{gain-smooth}
 \|Q^+(f,g)\|_{\dot{H}^{\gamma}}\leq C \|f\|_{L^{\mt{p}}}\|g\|_{L^{\mt{q}}},\;
 \|Q^+(f,g)\|_{{H}^{\gamma}}\leq C \|\lr{v}^{\gamma}f\|_{L^{\mt{p}}}\|\lr{v}^{\gamma}g\|_{L^{\mt{q}}},
\end{equation} 
where $\dot{H}^{\gamma}$ and $H^{\gamma}$ are the homogeneous and inhomogeneous Sobolev space 
respectively, $1/{\mt{p}}+1/{\mt{q}}=3/2,\;1\leq \mt{p},\mt{q}\leq 2$ and $\lr{v}=(1+|v|^{2})^{1/2}$
is the Japanese bracket (see also the notations in the end of this section).  

Secondly, it is also useful to regard the gain term as a convolution-like operator. The Hardy-Littlewood-Sobolev 
type estimates was first given by Gustsfsson~\cite{Gus86}, then by several authors~\cite{DKR06,AC10,ACG10,AG11}. 
The recent result of Alonso et al.~\cite{ACG10}  for the soft potential and Maxwellain models, 
i.e. $-3<\gamma\leq 0$, reads $1<\mt{p},\mt{q},\mt{r}<\infty$ and
\begin{equation}\label{sharp-scale-soft}
 \|Q^+(f,g)\|_{L^{\mt{r}}}\leq C \|f\|_{L^{\mt{p}}}\|g\|_{L^{\mt{q}}},
 \;\frac{1}{\mt{p}}+\frac{1}{\mt{q}}=1+\frac{\gamma}{3}+\frac{1}{\mt{r}}.
\end{equation}     
This estimate is employed  by the first and second authors in the recent work~\cite{HJ17,HJ23} where they solve 
the small initial data Cauchy problem for the 3 dimensional  Boltzmann equation for the cases
$-1\leq \gamma<0$. Because the estimate~\eqref{sharp-scale-soft}  does not increase the moment and 
thus gives the desirable nonlinear estimate, 
\begin{equation}\label{desire-v}
\|\lr{v}^{\ell} Q^{+}(f_+,f_+)\|_{L^{\tilde{p}'}} \le  C\|\lr{v}^{\ell} f_+\|^{2}_{L^p},
\end{equation}
for suitable $(\tilde{p}',p)$ which fits the requirement of the Strichatz estimate (see also 
Section~\ref{well-posed} for the details).
The above estimate  allows us to close the iteration by the fixed point argument.

On the other hand, for the hard potential and hard sphere models
i.e., $0<\gamma\leq 1$, the estimate of Alonso et al.~\cite{ACG10} reads
\begin{equation}\label{gain-hard-loss}
\|\lr{v}^{\ell} Q^{+}(f,g)\|_{L^{\mt{r}}}\leq C\|\lr{v}^{\ell+\gamma} f\|_{L^{\mt{p}}}
\|\lr{v}^{\ell+\gamma} g\|_{L^{\mt{q}}},
\end{equation}  
where 
\begin{equation}\label{Convolution-scaling}
\frac{1}{\mt{p}}+\frac{1}{\mt{q}}=1+\frac{1}{\mt{r}},\;\ell\in\mathbb{R},\;1\leq \mt{p},\mt{q},\mt{r}\leq \infty.
\end{equation}
There is an increase of moment of order $\gamma$ in the estimate~\eqref{gain-hard-loss} comparing to the 
estimate~\eqref{sharp-scale-soft} which makes the approach in~\cite{HJ17,HJ23} being not practicable.
 
However the first estimate of~\eqref{gain-smooth} and Sobolev embedding implies that for 
$1\leq\mt{p},\mt{q}\leq 2$, 
\[
\|Q^+(f,g)\|_{L^{\mt{r}}}\leq C \|f\|_{L^{\mt{p}}}\|g\|_{L^{\mt{q}}},\;{\rm if}\;
\frac{1}{\mt{p}}+\frac{1}{\mt{q}}=\frac{3}{2}=1+\frac{\gamma}{3}+\frac{1}{\mt{r}}.
\]
This means~\eqref{sharp-scale-soft} holds for hard potential and hard sphere models at least for 
some of triplets $(\mt{p},\mt{q},\mt{r})$. It is natural to ask if it is necessary to increase the moment when we consider
the $L^{p}$ estimate for the gain term with hard potential and hard sphere models.  

The first part of this work provides an unexpected observation to the gain term for the  hard potential and 
hard sphere models for this question.  We prove the  Hardy-Littlewood-Sobolev type $L^p$ estimates
for the gain term of the Boltzmann collision operator including Maxwellian molecule, hard potential and 
hard sphere models. 
Combining with the results of  Alonso et al.~\cite{ACG10} for the soft potential and Maxwellian molecule 
models, we provide an unified  form of $L^p$ estimates for all cutoff models which 
are sharp in the sense of scaling. The most striking feature of our new estimates for the hard potential 
and hard sphere models is that they do not increase the moment, the same as Maxwellian molecule 
and soft potential models.  
Therefore the damping effect inside the collision operator for the hard potential 
and hard sphere models comes from the loss term only, 
at least in the suitable $L^p$ space. This might be a distinguishing feature of hard potential 
and hard sphere models comparing to the Maxwellian molecule 
and soft potential models.

\begin{thm}\label{T:Gain-p-w-est} 
Let $\ell_{0}\geq 0$, $1<\mt{p}, \mt{q}, \mt{r} <\infty$,\; $0\leq  \gamma\leq 1$ and
\begin{equation}\label{scaling-relation}
\frac{1}{\mt{p}}+\frac{1}{\mt{q}}=1+\frac{1}{\mt{r}}+\frac{\gamma}{3},\;\;
\frac{\gamma}{6}\leq \frac{1}{\mt{r}}\leq 1-\frac{5\gamma}{6},
\end{equation}
and both $\leq$ are replaced by $<$ when $\gamma=0$.
Consider
\[
B(v-v_*,\omega)=|v-v_*|^{\gamma}\;\cos\theta.
\]
Then the bilinear operator $Q^{+}(f,g)$ satisfies  
\begin{equation}\label{W-convolution}
\|\lr{v}^{\ell_{0}}Q^{+}(f,g)\|_{L^{\mt{r}}(\mathbb{R}^3)}\leq C\;\|\lr{v}^{\ell_{0}} f\|_{L^{\mt{p}}(\mathbb{R}^3)}
\|\lr{v}^{\ell_{0}}g\|_{L^{\mt{q}}(\mathbb{R}^3)}.
\end{equation}
If $\ell_{1}>3/m$ and $1<\mt{p}_{m}, \mt{q}_{m}, m, \mt{r}_{m} <\infty$ satisfy
\begin{equation}\label{minus-size-condition}
\frac{1}{\mt{p}_{m}}+\frac{1}{m}<1\;,\;\frac{1}{\mt{q}_{m}}+\frac{1}{m}<1,
\end{equation}
both $\leq$ are replaced by $<$ when $\gamma=0$, and
\begin{equation}\label{scaling-relation-w}
\frac{1}{\mt{p}_{m}}+\frac{1}{\mt{q}_{m}}+\frac{1}{m}=1+\frac{1}{\mt{r}_{m}}+\frac{\gamma}{3}, \;\;
\frac{\gamma}{6}\leq \frac{1}{\mt{r}_{m}}\leq 1-\frac{5\gamma}{6},
\end{equation}
then we have  
\begin{equation}\label{W-convolution-m}
\|\lr{v}^{\ell_{1}} Q^{+}(f,g)\|_{L^{\mt{r}_{m}}(\mathbb{R}^3)}\leq C\;
\|\lr{v}^{\ell_{1}} f\|_{L^{\mt{p}_{m}}(\mathbb{R}^3) }\|\lr{v}^{\ell_{1}} g\|_{L^{\mt{q}_{m}}(\mathbb{R}^3)}.
\end{equation}
\end{thm}
\begin{rem}
The estimate~\eqref{W-convolution-m} satisfying~\eqref{minus-size-condition},~\eqref{scaling-relation-w} 
is a modification of the estimate~\eqref{W-convolution} satisfying~\eqref{scaling-relation}.  
The second estimate~\eqref{W-convolution-m} will be used in solving the small data Cauchy problem 
for the Boltzmann equation later.  
\end{rem}
\begin{rem}
Here we restrict our study of the Boltzmann equation to $(t,x,v)\in \R\times\R^{3}\times\R^{3}$ for its importance 
in physics and to avoid lengthy, complicated presentation and proofs. The analogy results of the 
Theorem~\ref{T:Gain-p-w-est} and the following Theorem~\ref{result1} in other dimensions  will be 
presented in the upcoming work~\cite{Liang24}. 
\end{rem}
\begin{rem}
Our results also indicate that the hard sphere model has its own unique feature for the interval of $1/\mt{r}$ 
in~\eqref{scaling-relation} shrinks to a point unlike the hard potential and Maxwellian models.
This also happens to the interval of  $\varepsilon$ in the statement of Theorem~\ref{result1} below.
\end{rem}
\begin{rem}
The proof of Theorem~\ref{T:Gain-p-w-est} indeed also implies  that for  
$\frac{1}{\mt{p}}+\frac{1}{\mt{q}}=1+\frac{1}{\mt{r}},\;\;
\frac{\gamma}{6}\leq \frac{1}{\mt{r}}\leq 1-\frac{5\gamma}{6}$, we have 
\[
\| \Delta^{\frac{\gamma}{2}}\; Q^{+}(f,g)\|_{L^{\mt{r}}(\mathbb{R}^3)}\leq C\| f\|_{L^{\mt{p}}(\mathbb{R}^3)}\| g\|_{L^{\mt{q}}(\mathbb{R}^3)}.
\]
if we want to gain regularity of order $\gamma$ instead of gain integrability 
(we leave the modification of the proof to the reader). However we do not see very interesting 
application of this estimate so far. 
\end{rem}

Our second result is the application of the above result which concerns the global well-posedness of the 
Cauchy problem for the cutoff Boltzmann equation with small initial data:  
\begin{equation}\label{E:Cauchy}
\left\{
\begin{aligned}
&{\partial_t f}+v\cdot\nabla_x f =Q(f,f),\\
&f(0,x,v)=f_0(x,v).
\end{aligned}
\right.
\end{equation}
The first result in this aspect is due to Illner and Shinbrot~\cite{IS84} who showed the global existence of solutions for 
the Cauchy problem~\eqref{E:Cauchy} for several cutoff models when the initial data has exponential decay 
in spatial variable and has suitable weight in velocity variable. The  Kaniel-Shinbrot iteration method 
in~\cite{IS84}  comes from the earlier work~\cite{KS78}. Small initial data Cauchy problem 
 for different cut-off models was extensively studied during that decade using the same iteration or fixed 
point argument, see~\cite{BPT88, G96} and reference therein for more details. We remark that the assumption 
that the initial data has exponential decay in  spatial variable or in velocity variable is necessary to get these 
results.  

The Strichartz estimates for the kinetic equation in the note of Castella and Perthame~\cite{CP96} provide 
a possible tool to relax the harsh assumption, exponential decay,  posed on the initial data. 
Assumed the  kinetic part of the collision kernel is $L^p$ integrable for some $p$ depending on 
dimension, Ars\'{e}nio~\cite{Ars11}  was able to use Strichartz estimates to prove the existence of 
global weak (non-unique) solution
when the initial data is small in $L^D_{x,v}$ where $D$  is the dimension.  For the authentic 
 cutoff kernel with $-1\leq \gamma<0$, dimension 3 and initial data is small in weighted $L^3$ space, the global well-posed 
problem is solved by the first and second authors~\cite{HJ17,HJ23} by combining Strichartz 
estimates~\cite{CP96,Ovc11} and Kaniel-Shinbrot iteration~\cite{IS84,KS78}. As we mentioned before the 
most important ingredient of all is the estimate~\eqref{desire-v}.    
Note that the result for the two dimensional case with $\gamma=0$ is obtained by  Chen, Denlinger and 
Pavlovi\'{c}~\cite{CDP21} using a different approach.  
With Theorem~\ref{T:Gain-p-w-est} in hand, we are able 
to deal with the small initial data Cauchy problem of the Boltzmann equation for the Maxwellian molecule, 
hard potential and hard sphere models following the approach of~\cite{HJ23}.  

To state the result, let us introduce the mixed Lebesgue norm
\[
\|f(t,x,v)\|_{L^q_tL^r_xL^p_v}
\] where the notation  $L^q_tL^r_xL^p_v$ stands for the space $L^q(\mathbb{R};L^r
(\mathbb{R}^3;L^p(\mathbb{R}^3)))$. It is understood that we use 
$L^q_t(\mathbb{R})=L^q_t([0,\infty))$ for the well-posedness problem which can 
be done by imposing support restriction to the inhomogeneous  Strichartz estimates. 
We use $L^a_{x,v}$ to denote $L_x^a(\mathbb{R}^3;L_v^a(\mathbb{R}^3))$. Please note that 
the notation $L^{p}_{v}$ is used to indicate the variable of integration is $v$ which appears only 
when we use mixed norm. This should not be confused with weighted $L^{p}$ norm used in many 
papers.  

We also define the scattering of the solution with respect to kinetic transport operator  as the following. 
We say that a global solution $f\in C([0,\infty),L^a_{x,v})$ scatters in $L^a_{x,v}$ as 
$t\rightarrow\infty$ if there exits $f_{\infty}\in L^a_{x,v}$ such that 
\begin{equation}
\|f(t)-U(t)f_{\infty}\|_{L^a_{x,v}}\rightarrow 0
\end{equation}
where $U(t)f(x,v)=f(x-vt,v)$ is the solution map of the kinetic transport equation 
\[
\partial_t f+v\cdot\nabla_x f=0. 
\]
See also~\cite{BGGL16} and reference therein for the discussion about the scattering of the solution and 
its relation with famous H-theorem. With these terminologies, we state the second result 
as the following. 

\begin{thm}\label{result1}
Let $0\leq \gamma\leq 1$ and 
\[
B(v-v_*,\omega)=|v-v_*|^{\gamma}\;\cos\theta.
\]  
Let $\ell>2\gamma+\frac{10}{9}$ and $\varepsilon$ satisfy
\[
\max {\big \{ } 0,  \frac{5\gamma}{12}-\frac{7}{18}   \} < \varepsilon \leq \frac{1}{9}-\frac{\gamma}{12}.
\] 
There exists a small number $\eta>0$ such that if the initial data 
$$
f_{0}\in {\mathbb B}^{\ell}_{\eta}=\{f_{0} | f_0\geq 0,\; \| \lr{v}^{\ell} f_{0}\|_{L^{3}_{x,v}}<\eta,\;
\|\lr{v}^{\ell}f_0\|_{L^{15/8}_{x,v}}<\infty\}\subset L^3_{x,v},
$$
then the Cauchy problem~\eqref{E:Cauchy} admits a unique and non-negative mild solution 
\[
\lr{v}^{\ell} f\in C([0,\infty),L^3_{x,v})\cap
L^{\mbq}([0,\infty],L^{\mbr}_xL^{\mbp}_v), 
\]
where the triple 
\begin{equation}
(\frac{1}{\mbq},\frac{1}{\mbr},\frac{1}{\mbp})=(\frac{1}{3}-3\varepsilon,\;\frac{2}{9}+\varepsilon,\;\frac{4}{9}-\varepsilon). \end{equation}
The solution map $\lr{v}^{\ell}f_0\in {\mathbb B}^{\ell}_{\eta}\rightarrow \lr{v}^{\ell} 
f\in L^{\mbq}_tL^{\mbr}_xL^{\mbp}_v$ 
is Lipschitz continuous and the solution $\lr{v}^{\ell} f$ scatters with respect to the kinetic 
transport operator in $L^3_{x,v}$. 
\end{thm}

The paper is organized as the following: In section 2, the $L^p$ estimate for the gain term is reduced 
to proving $L^p$ estimate for a Radon transform $\mathbb{T}$.  The transform  $\mathbb{T}$ is decomposed into 
the sum of Fourier integral operators (FIOs) and degenerate operators in section 3.  The $L^p$ boundedness
of the former is proved in  the section 4 and the later is in section 5, thus the proof of 
Theorem~\ref{T:Gain-p-w-est} is complete. The proof of Theorem~\ref{result1} is given in section 6.

\noindent{\bf Notations}\hfil\\
We set Japanese bracket  $\lr{v}=(1+|v|^2)^{1/2},\;v\in\mathbb{R}^3$, $(a\cdot b)=\sum_{i=1}^3 a_ib_i$  
the scalar product in $ \mathbb{R}^3$ and $\lr{f,g}=\int_{\mathbb{R}^3}f(x)g(x)dx$  the inner product in 
$L^2({\mathbb{R}^3}). $
The differential operator
$D^s,s\in\mathbb{R}$ is expressed through the Fourier transform:
\[
D^s f(x)=(2\pi)^{-3}\int_{\mathbb{R}^3} e^{ix\cdot\xi}|\xi|^s \widehat{f}(\xi)d\xi
\]
where the Fourier transform
$
\widehat{f}(\xi)=\int_{\mathbb{R}^3} e^{-ix\cdot\xi} f(x)dx.
$
The  fractional homogeneous and inhomogeneous Sobolev spaces are denoted by
\[
 \|f\|_{\dot{H}^{\alpha}}=\||\xi|^{\alpha}\widehat{f}(\xi)\|_{L^2}\;,\;
\|f\|_{{H}^{\alpha}}=\|\lr{\xi}^{\alpha}\widehat{f}(\xi)\|_{L^2}
\]

We use the multi-indices notation $\partial_x^{\alpha}=\partial_{x_1}^{\alpha_1}\cdots\partial_{x_d}^{\alpha_d}$.
By $\partial_x$ or by $\nabla_x$, we will denote the gradient.
A function $p(x,\xi)\in C^{\infty}(\mathbb{R}^3\times\mathbb{R}^3)$  satisfying 
\begin{equation}\label{D:symbol-definition}
 |\partial^{\alpha}_x \partial^{\beta}_{\xi}p(x,\xi)|
\leq C_{\alpha,\beta} \lr{\xi}^{m-|\beta|}
\end{equation} 
for any multi-indices $\alpha$ and $\beta$ is called a symbol of order $m$. The class of such function is  denoted by $S^{m}_{1,0}$.
We will see that the symbols $p(x,\xi)$ in this paper always enjoy the better decaying condition, i.e., 
\begin{equation}\label{D:symbol-definition-SG}
 |\partial^{\alpha}_x \partial^{\beta}_{\xi}p(x,\xi)|
\leq C_{\alpha,\beta}\lr{x}^{m_1-|\alpha|}\lr{\xi}^{m_2-|\beta|}
\end{equation} 
for any multi-indices $\alpha$ and $\beta$ and which is called a $SG$ symbol of order $(m_1,m_2)$, used by 
Cordes~\cite{Cord95} and Coriasco~\cite{Cori99}.
We  use  $SG^{m_1,m_2}$ to denote the set of such symbols.  For each $p(x,\xi)\in S_{1,0}^l$, the associate operators 
\[
P(x,D) h(x)=\int_{\mathbb{R}^3} e^{ix\cdot\xi} p(x,\xi) \widehat{h}(\xi)d\xi
\]
is called a pseudodifferential operator of order $l$.  The standard notation $S^{-\infty}_{1,0}=\cap_{\;l} S^{l}_{1,0},l\in\mathbb{Z}$
is also used. If $p(x,\xi)\in S^{-\infty}_{1,0}$, it is called a symbol of the smooth operator.
The operator 
\[
\mathscr{T} f(x)=\int_{\mathbb{R}^3} e^{i\psi(x,\xi)} p(x,\xi) \widehat{f}(\xi)d\xi
\]
with symbol $p(x,\xi)\in S^l_{1,0}$  and the phase function $\psi(x,\xi)$ satisfies  
non-degeneracy condition is called a Fourier integral operator of order $l$.
We say a phase function $\psi(x,\xi)$
satisfies the non-degeneracy condition if there is a constant $c>0$ such that 
\begin{equation}\label{D:nondegeneracy}
|\det  {\Big [} \partial_{x}\partial_{\xi} \psi(x,\xi) {\Big ]} | =|\det  {\Big [} \partial^2\psi(x,\xi)/\partial_{x_i}\partial_{\xi_j}  {\Big ]} |\geq c>0
\end{equation} 
for all $(x,\xi)\in{\rm supp}\;p(x,\xi)$.

\section{$L^p$ estimate for the gain term}

To study the gain term of the Boltzmann collision operator, we recall the Radon transform 
introduced by Lions~\cite{Lio94}:
\begin{equation}\label{Def:T}
\mathbb{T} h(x)=|x|^{\gamma}\int_{\omega\in S^2_+} b(\cos\theta)h(x-(x\cdot\omega)\omega) d\omega, 
\;0\leq\gamma\leq 1
\end{equation}
where 
$\cos\theta={(x\cdot\omega)}/{|x|}, x\neq 0$, $x=|x|(0,0,1)$ and 
$\omega=(\cos\varphi\sin\theta,\sin\varphi\sin\theta,\cos\theta)$, $0\leq\theta\leq \pi/2 $.
In our case, we have 
\[
b(\cos\theta)=\cos\theta.
\]
Lions~\cite{Lio94} assumed that $|x|^{\gamma}$ and $\theta$  both vary inside compact sets 
and the support of $\theta$ is away from $0$ and $\pi/2$.   Such a restriction has been removed in the second author’s 
work~\cite{JC12,JC20} in the proof of the $L^{2}$ based regularizing  estimates.  

By the inverse Fourier transform, we may write
\begin{equation}\label{D:T-Fourier}
\begin{split}
\mathbb{T} h(x)
&=(2\pi)^{-3}\int_{\mathbb{R}^3} e^{i{x\cdot\xi}} {A}(x,\xi) \widehat{h}(\xi) d\xi
\end{split}
\end{equation}
where 
\[
{A}(x,\xi)=|x|^{\gamma}\int_{\omega\in S^2_+} e^{-i(x\cdot\omega)(\xi\cdot\omega)} b(\cos\theta) d\omega.
\]

To explain the structure of~\eqref{D:T-Fourier}, we review some fact about the Fourier integral 
operators (FIOs). 
The Fourier integral operators which we consider have the form: 
\begin{equation}\label{D:FIO}
\mathscr{T} h(x)=\int_{\mathbb{R}^{3}} e^{i\psi(x,\xi)} p(x,\xi)\; \widehat{h}(\xi) d\xi
\end{equation}
where $p(x,\xi)$ is a SG-symbol of order $(\alpha,\beta)$ defined in~\eqref{D:symbol-definition-SG}
and the phase function $\psi$ satisfies the non-degeneracy condition, i.e., ~\eqref{D:nondegeneracy}
for all $(x,\xi)\in {\rm supp} \;p(x,\xi)$.
 
 Let $1<p<\infty$ and the pair of orders  $(m_{1},m_{2})$ satisfy
\begin{equation}\label{Symbol-order}
m_{1}\leq -2\;{\Big|}\frac{1}{2}-\frac{1}{p}{\Big |},\;\;m_{2}\leq -2\; {\Big |} 
\frac{1}{2}-\frac{1}{p}{\Big |},\;\;1<p<\infty.
\end{equation} 
When the symbol $p(x,\xi)$ has the compact support as a function of $x$, thus the first inequality 
of ~\eqref{Symbol-order} is automatically satisfied, it is called a local FIO.  
The local FIO whose symbol satisfies~\eqref{Symbol-order} enjoys 
\[
\|\mathscr{T} h\|_{L^p}\leq C\|h\|_{L^p}
\]
 is proved by Seeger, Sogge and Stein~\cite{SSS91}, see also Stein's book~\cite{Stein93}.  
An FIO whose symbol is not compact supported in $x$ variable is called a 
global FIO.  The fact that the SG-symbol global FIO whose symbol satisfies~\eqref{Symbol-order}  
enjoys the same estimate as the local one is proved by Coriasco and Ruzhansky~\cite{CR14},  
see also the references therein for more details.

It has been showed in~\cite{JC20} that the behavior of the operator $\mathbb{T}$ defined 
by~\eqref{Def:T} depends on its position in the phase space $(x,\xi)$
as well as the angle spanned by $x$ and $\xi$. By splitting 
the phase space into cones, we can write $\mathbb{T}$ as a sum of global FIOs 
and  degenerate operators where each FIO is defined on one cone whose size is related to the angle spanned by 
$x$ and $\xi$.
The SG-symbols of FIOs all have (we should call it a symbol later on) the order 
\begin{equation}\label{symbol-order}
m_{1}=\gamma-1,\;m_{2}=-1,
\end{equation}
while the phase functions has no uniform non-degeneracy condition, i.e., 
the constant $c$ in ~\eqref{D:nondegeneracy} may close to $0$ arbitrary and which 
is reason for splitting the phase space into cones.  The  degenerate operators need different 
analysis with the same geometric setup.  With such geometric consideration, the $L^2$-type 
estimate was build in~\cite{JC20} as good as a single global FIO.    

It is natural to expect that $\mathbb{T}$ also enjoys $L^p$-type of estimates as a single global FIO.  More precisely, 
\[
\|\mathbb{T} h\|_{L^p}\leq C\|h\|_{L^{\widetilde{p}}}\;{\rm where}\;\frac{1}{\widetilde{p}}=\frac{1}{p}+\frac{\gamma}{3}
\]
and  $p$ satisfies~\eqref{Symbol-order}.   The relations~\eqref{Symbol-order} and $\gamma-1$ in~\eqref{symbol-order}
is equivalent to 
\[
 \frac{\gamma}{2}\leq \frac{1}{p}\leq 1-\frac{\gamma}{2},
 \]
and both $\leq$ are replaced by $<$ when $\gamma=0$.
Here the exponent $\widetilde{p}$ is due to $m_{2}=(\gamma-1)-\gamma$ and Sobolev emdedding
 (see Section~\ref{Section 3} for more details). 
The proof of the estimate for $\mathbb{T}$ occupies the major part of our paper.

For our purpose, we need variants of the above estimate which are stated in the following 
Lemma~\ref{L:R-v-v_*-ineq}. To state it, 
we introduce the notation $\|F(v,v_*)\|_{L^p(v)}$ to denote the $L^p$ norm of 
$F(v,v_*)$ with respect to variable $v$ where the variable $v_*$ is regarded as a parameter
and vice verse for $\|F(v,v_*)\|_{L^p(v_*)}$.
We will see soon that the gain term estimates can be derived by the
Lemma~\ref{L:R-v-v_*-ineq} below which are concerning the estimates of $\mathbb{T}$ composed 
with translations.
\begin{lem}\label{L:R-v-v_*-ineq}
Let $\mathbb{T}$ be the operator defined by~\eqref{Def:T} and $\tau_m$ be the translation 
operator $\tau_m h(\cdot)=h(\cdot+m)$. Suppose $p$ and $\widetilde{p}$ satisfy
\begin{equation}\label{P-range}
 \frac{\gamma}{2}\leq \frac{1}{p}\leq 1-\frac{\gamma}{2},\; \frac{1}{\widetilde{p}}=\frac{1}{p}+\frac{\gamma}{3},\;
\end{equation}
and both $\leq$ are replaced by $<$ when $\gamma=0$.
Then we have
\begingroup
\large
\begin{equation}\label{E:R-v-ineq}  
\sup\limits_{v_*}\|(\tau_{-v_*}\circ \mathbb{T} \circ\tau_{v_*} )h(v)\|_{L^p(v)}
\leq C \|h\|_{L^{\widetilde{p}}}
\end{equation}
\endgroup
and
\begingroup
\large
\begin{equation}\label{E:R-v_*-ineq} 
\sup\limits_{v}\|(\tau_{-v_*}\circ \mathbb{T} \circ\tau_{v_*} )h(v)\|_{L^p(v_*)}
\leq C \|h\|_{L^{\widetilde{p}}}.
\end{equation}
\endgroup
\end{lem}

The proof of Lemma~\ref{L:R-v-v_*-ineq} will be given in Section~\ref{section-5} after the estimate for 
$\mathbb{T}$ is done.
Suppose Lemma~~\ref{L:R-v-v_*-ineq} holds, then we can prove Theorem~\ref{T:Gain-p-w-est} 
for the case $\ell_{0}=0$, i.e., the following estimate for the gain term.

\begin{thm}\label{Gain-p-est}
 Let $1<\mt{p}, \mt{q}, \mt{r} <\infty$,\; $0\leq \gamma\leq 1$ and 
 \begin{equation}\label{scaling-relation-0}
 \frac{1}{\mt{p}}+\frac{1}{\mt{q}}=1+\frac{1}{\mt{r}}+\frac{\gamma}{3},\;\; \frac{\gamma}{6}\leq \frac{1}{\mt{r}}
 \leq 1-\frac{5\gamma}{6},
 \end{equation}
and both $\leq$ are replaced by $<$ when $\gamma=0$.
Consider
\[
B(v-v_*,\omega)=|v-v_*|^{\gamma}\;\cos\theta.
\]
Then the bilinear operator $Q^{+}(f,g)$ is a bounded operator from $L^{\mt{p}}(\mathbb{R}^3)\times 
L^{\mt{q}}(\mathbb{R}^3)\rightarrow L^{\mt{r}}(\mathbb{R}^3)$ via the estimate 
\begin{equation}\label{G-conv-est}
\|Q^{+}(f,g)\|_{L^{\mt{r}}_{v}(\mathbb{R}^3)}\leq C\|f\|_{L^{\mt{q}}_{v}(\mathbb{R}^3)}\|g\|_{L^{\mt{p}}_{v}(\mathbb{R}^3)}.
\end{equation}
\end{thm}

\begin{proof}
By the duality, it suffices to show that
\[
{\big |}\lr{Q^+(f,g),h}{\big |}\leq C \|f\|_{L^{\mt{q}}}\|g\|_{L^{\mt{p}}}\|h\|_{L^{\mt{r}'}}
\]
holds for any $h\in L^{\mt{r}'}$ where $\mt{r}'$ is the conjugate exponent of ${\mt{r}}$.
Using change of variables and H\"{o}lder inequality
\[
\begin{split}
&{\big |}\lr{Q^+(f,g),h}{\big |}\\
&={\big |}\iiint f(v)g(v_*)B(v-v_*,\omega)h(v')d\omega dv_* dv{\big |}\\
&\leq \|f(v)\|_{L^{\mt{q}}}\|Q^{+t}_g(h)(v)\|_{L^{\mt{q}'}},
\end{split}
\]
where 
\[
\begin{split}
Q^{+,t}_g(h)(v)& =\iint g(v_*)B(v-v_*,\omega)
h(v')d\omega dv_* \\
&=\int g(v_*)(\tau_{-v_*}\circ \mathbb{T} \circ\tau_{v_*} )h(v) dv_*.
\end{split}
\]
where we used the the notations as Lemma~\ref{L:R-v-v_*-ineq} in the last line.
Thus it is reduced to proving that 
\begin{equation}\label{E:Q-adjoint-ineq1}
\|Q^{+,t}_g(h)(v)\|_{L^{\mt{q}'}}\leq C\|g\|_{L^p}\|h\|_{L^{\mt{r}'}}.
\end{equation}
where 
\[
\frac{1}{\mt{p}}+\frac{1}{\widetilde{\mt{r}}'}=1+\frac{1}{\mt{q}'}\;,\; 
\frac{1}{\widetilde{\mt{r}}} \overset{\rm def}{=}\frac{1}{\mt{r}}+\frac{\gamma}{3}.
\]

Define  $H(v,v_*)=(\tau_{-v_*}\circ \mathbb{T}\circ\tau_{v_*})h(v)$. By 
\[
\frac{1}{\mt{q}'}+\frac{\mt{q}'-\widetilde{\mt{r}}'}{\widetilde{\mt{r}}'\mt{q}'}+\frac{\mt{q}'-\mt{p}}{\mt{p}\mt{q}'}=1,
\]
and H\"{o}lder inequality, we have 
\[
\begin{split}
&{\big |}Q^{+t}_g(h)(v){\big |}\leq \int {\big |}H(v,v_*){\big |}{\big |}g(v_*){\big |}dv_*\\
&=\int {\Big(} {\big |} H(v,v_*){\big |}^{\widetilde{\mt{r}}'}{\big |}g(v_*){\big |}^{\mt{p}} {\Big )}^{1/\mt{q}'}{\big |}H(v,v_*){\big |}^{(\mt{q}'-\widetilde{\mt{r}}')/\mt{q}'}
{\big |}g(v_*){\big |}^{(\mt{q}'-\mt{p})/\mt{q}'} dv_*\\
&\leq \|( |H(v,v_*)|^{\widetilde{\mt{r}}'} |g(v_*)|^{\mt{p}})^{1/\mt{q}'}\|_{L^{\mt{q}'}(v_*)}\times \\
&{\hskip 2cm}\||H(v,v_*)|^{(\mt{q}'-\widetilde{\mt{r}}')/\mt{q}'}\|_{L^{\frac{\widetilde{\mt{r}}'\mt{q}'}
{\mt{q}'-\widetilde{\mt{r}}'}}(v_*)}
\||g(v_*)|^{(\mt{q}'-\mt{p})/\mt{q}'}\|_{L^{\frac{\mt{p}\mt{q}'}{\mt{q}'-\mt{p}}}(v_*)}.
\end{split}
\]
Here
\begin{equation}\label{E:G-v_*}
\||g(v_*)|^{(\mt{q}'-\mt{p})/\mt{q}'}\|_{L^{\frac{\mt{p}\mt{q}'}{\mt{q}'-\mt{p}}}(v_*)}=(\|g(v_*)\|_{L^{\mt{p}}(v_*)})^{\frac{\mt{q}'-\mt{p}}{\mt{q}'}},
\end{equation}
\begin{equation}\label{E:H-v-v_*-1}
\||H(v,v_*)|^{(\mt{q}'-\widetilde{\mt{r}}')/\mt{q}'}\|_{L^{\frac{\widetilde{\mt{r}}'\mt{q}'}{\mt{q}'-\widetilde{\mt{r}}'}}(v_*)}
=(\|(\tau_{-v_*}\circ T \circ\tau_{v_*} )h(v)\|_{L^{\widetilde{\mt{r}}'}(v_*)})^{\frac{\mt{q}'-\widetilde{\mt{r}}'}{\mt{q}'}},
\end{equation}
and 
\begin{equation}\label{E:HG-v-v_*}
\begin{split}
&\|( |H(v,v_*)|^{\widetilde{\mt{r}}'}|g(v_*)|^{\mt{p}})^{1/\mt{q}'}\|_{L^{\mt{q}'}(v_*)}\\
&=(\int {\big |}g(v_*){\big |}^{\mt{p}} {\big |}(\tau_{-v_*}\circ \mathbb{T} \circ\tau_{v_*} )h(v) {\big |}^{\widetilde{\mt{r}}'} dv_*)^{\frac{1}{\mt{q}'}}.
\end{split}
\end{equation}
Applying~\eqref{E:R-v_*-ineq} of Lemma~\ref{L:R-v-v_*-ineq} by plugging 
\[
\frac{1}{p}=\frac{1}{\widetilde{\mt{r}}'}=\frac{1}{\mt{r}'}-\frac{\gamma}{3}.
\] 
We see the inequality in~\eqref{scaling-relation-0} comes from~\eqref{P-range}  and that if 
$\mt{r}$ satisfies~\eqref{scaling-relation-0} then
\begin{equation}\label{E:H-v-v_*-2}
~\eqref{E:H-v-v_*-1}\leq C (\|h\|_{L^{\mt{r}'}})^{\frac{\mt{q}'-\widetilde{\mt{r}}'}{\mt{q}'}}
\end{equation}
where $C$ is independent of $v$.
Combine~\eqref{E:G-v_*},~\eqref{E:H-v-v_*-2} and
~\eqref{E:HG-v-v_*}, we have
\begin{equation}\label{E:combine-1}
\begin{split}
&\|Q^{+,t}_g(h)(v)\|_{L^{\mt{q}'}(v)}^{\mt{q}'} \\
&\leq C (\|g(v_*)\|_{L^\mt{p}(v_*)})^{(\mt{q}'-\mt{p})}(\|h\|_{L^{\mt{r}'}})^{(\mt{q}'-\widetilde{\mt{r}}')}\times\\
& {\hskip 2cm}\int\int {\big |}g(v_*){\big |}^\mt{p} {\big |}(\tau_{-v_*}\circ T\circ\tau_{v_*} )h(v) 
{\big |}^{\widetilde{\mt{r}}'} dv_* dv\\
&\leq C (\|g(v_*)\|_{L^\mt{p}(v_*)})^{(\mt{q}'-\mt{p})}(\|h\|_{L^{\mt{r}'}})^{(\mt{q}'-\widetilde{\mt{r}}')}\times\\
& {\hskip 2cm} \int |g(v_*)|^\mt{p} \sup\limits_{v_*}\|
(\tau_{-v_*}\circ T \circ\tau_{v_*} )h(v)\|_{L^{\widetilde{\mt{r}}'}_{v}}^{\widetilde{\mt{r}}'}  dv_*. \\
\end{split}
\end{equation}
Then we conclude~\eqref{E:Q-adjoint-ineq1} by  applying ~\eqref{E:R-v-ineq} of 
Lemma~\ref{L:R-v-v_*-ineq} to the last line of ~\eqref{E:combine-1} when  $\mt{r}$ satisfies~\eqref{scaling-relation-0}.
The proof of~\eqref{G-conv-est} is complete.

\end{proof}

For the purpose of the  application, we need the weighted estimates for the 
gain term, i.e., Theorem~\ref{T:Gain-p-w-est}. Now we prove it by using Theorem~\ref{Gain-p-est}.
For the convenience of readers, we restate Theorem~\ref{T:Gain-p-w-est} below, then give the proof.
\begin{thm*}
Let $\ell_{0}\geq 0$, $1<\mt{p}, \mt{q}, \mt{r} <\infty$,\; $0\leq  \gamma\leq 1$ and
\[
\frac{1}{\mt{p}}+\frac{1}{\mt{q}}=1+\frac{1}{\mt{r}}+\frac{\gamma}{3},
\;\frac{\gamma}{6}\leq \frac{1}{\mt{r}}\leq 1-\frac{5\gamma}{6}, 
\tag{\ref{scaling-relation}}
\]
and both $\leq$ are replaced by $<$ when $\gamma=0$.
Consider
\[
B(v-v_*,\omega)=|v-v_*|^{\gamma}\;\cos\theta.
\]
Then the bilinear operator $Q^{+}(f,g)$ satisfies  
\[
\|\lr{v}^{\ell_{0}}Q^{+}(f,g)\|_{L^{\mt{r}}_{v}(\mathbb{R}^3)}\leq C\|\lr{v}^{\ell_{0}} f\|_{L^{\mt{p}}_{v}(\mathbb{R}^3)}
\|\lr{v}^{\ell_{0}}g\|_{L^{\mt{q}}_{v}(\mathbb{R}^3)}. \tag{\ref{W-convolution}}
\]
If $\ell_{1}>3/m$ and $1<\mt{p}_{m}, \mt{q}_{m}, m, \mt{r}_{m} <\infty$ satisfy
\[
\frac{1}{\mt{p}_{m}}+\frac{1}{m}<1\;,\;\frac{1}{\mt{q}_{m}}+\frac{1}{m}<1, \tag{\ref{minus-size-condition}}
\]
and
\[
\frac{1}{\mt{p}_{m}}+\frac{1}{\mt{q}_{m}}+\frac{1}{m}=1+\frac{1}{\mt{r}_{m}}+\frac{\gamma}{3}, \;
\frac{\gamma}{6}\leq \frac{1}{\mt{r}_{m}}\leq 1-\frac{5\gamma}{6}, \tag{\ref{scaling-relation-w}}
\]
and both $\leq$ are replaced by $<$ when $\gamma=0$, then we have  
\[
\|\lr{v}^{\ell_{1}} Q^{+}(f,g)\|_{L^{\mt{r}_{m}}(\mathbb{R}^3)}\leq C(\mt{p}_{m},\ell) 
\|\lr{v}^{\ell_{1}} f\|_{L^{\mt{p}_{m}}(\mathbb{R}^3)} \|\lr{v}^{\ell_{1}} g\|_{L^{\mt{q}_{m}}(\mathbb{R}^3)}. 
\tag{\ref{W-convolution-m}}
\]
\end{thm*}
\begin{proof}

Recall $\lr{v}=(1+|v|^{2})^{1/2}$, we  consider the quantity 
\begin{equation}\label{gain-dual}
\lr{\lr{v}^{\ell}Q^{+}(f,g),\lr{v}^{-\ell}\psi},\;\ell \geq 0.
\end{equation}

From the conservation of  energy, either $\lr{v'}\leq 2\lr{v}$ or $\lr{v'}\leq 2\lr{v_{*}}$ has to be true. Hence
for any $\ell_{0}\geq 0$ we have either
\begin{equation}\label{weight-ineq}
\lr{v}^{-\ell_{0}}\leq C \lr{v'}^{-\ell_{0}} \;{\rm or}\;\lr{v_{*}}^{-\ell_{0}}\leq C \lr{v'}^{-\ell_{0}}.
\end{equation}
Recall $H(v,v_*):=(\tau_{-v_*}\circ T \circ \tau_{v_*})\psi(v)$ and let 
$\psi_{-\ell_{0}}(v)=\lr{v}^{-\ell_{0}}\psi(v)$. 
From~\eqref{weight-ineq}, we know that one of the following estimates is true, i.e.,
\begin{equation}\label{psi-weight-1}
\begin{split}
&{\big |} H(v,v_*){\big |}
\leq \int_{S^{2}} {\big |} \psi (v'){\big |} B(v-v_*,\omega)d\omega \\
&\leq C \lr{v}^{\ell_{0}}\int_{S^{2}}  {\big |}\lr{v'}^{-\ell_{0}} \psi(v') {\big |} B(v-v_*,\omega)d\omega\\
&= C\lr{v}^{\ell_{0}} (\tau_{-v_*}\circ \mathbb{T} \circ \tau_{v_*}) {\big |} \psi_{-\ell_{0}} {\big |}(v),
\end{split}
\end{equation}
or  
\begin{equation}\label{psi-weight-2}
\begin{split}
&{\big |} H(v,v_*){\big |}
\leq \int_{S^{2}} {\big |} \psi (v'){\big |} B(v-v_*,\omega)d\omega \\
&\leq C \lr{v_{*}}^{\ell_{0}} \int_{S^{2}}  {\big |}\lr{v'}^{-\ell_{0}} \psi(v') {\big |} B(v-v_*,\omega)d\omega\\
&=C \lr{v_{*}}^{\ell_{0}} (\tau_{-v_*}\circ \mathbb{T} \circ \tau_{v_*}) {\big |} \psi_{-\ell_{0}} {\big |}(v).
\end{split} 
\end{equation}
Denote $f_{\ell_{0}}(v)=\lr{v}^{\ell_{0}}f(v)$ and $g_{\ell_{0}}(v_{*})=\lr{v_{*}}^{\ell_{0}}g(v_{*})$. 
 Combining~\eqref{psi-weight-1} and~\eqref{psi-weight-2} , we have
\[
\begin{split}
& {\Big |} \int  Q^{+}(f,g)(v)\; \psi (v) dv {\Big |} \\
&\leq C \iint {\Big \{} {\big |} f_{\ell_{0}}(v)g(v_*){\big |}+{\big |}f(v)g_{\ell_{0}}(v_*) {\big |} {\Big \}} 
(\tau_{-v_*}\circ \mathbb{T} \circ \tau_{v_*}){\big |}\psi_{-\ell_{0}}{\big |}(v)dv_*dv.
\end{split}
\]

Following the proofs of the Theorem~\ref{Gain-p-est}, we  have 
\begin{equation}\label{dual-w-est}
\begin{split}
&{\Big |}\int  Q^{+}(f,g)(v)\; \psi (v) dv {\Big |}\\
&\leq C{\big (} \|f_{\ell_{0}}\|_{L^{\mt{p}}_{v}(\R^{3})}\|g\|_{L^{\mt{q}}_{v}(\R^{3})} + 
\|f\|_{L^{\mt{p}}_{v}(\R^{3})}\|g_{\ell_{0}}\|_{L^{\mt{q}}_{v}(\R^{3})} {\big)} \|\psi_{-\ell_{0}}\|_{L^{\mt{r}'}_{v}(\R^{3})}\\
&\leq C \|f_{\ell_{0}}\|_{L^{\mt{p}}_{v}(\R^{3})}\|g_{\ell_{0}}\|_{L^{\mt{q}}_{v}(\R^{3})} \|\psi_{-\ell_{0}}\|_{L^{\mt{r}'}_{v}(\R^{3})},
\end{split}
\end{equation}
thus we conclude~\eqref{W-convolution} by duality and~\eqref{gain-dual} . The  relation~\eqref{scaling-relation} 
follows from~\eqref{scaling-relation-0}.  

The proof of estimate~\eqref{W-convolution}  is an easy consequence of above argument which can be done 
by revising~\eqref{dual-w-est}. Let $1/a_{1}=1/\mt{p}_{m}+1/m<1$ and $1/a_{2}=1/\mt{q}_{m}+1/m<1$ and note that
we then have 
\[
\frac{1}{a_{1}}+\frac{1}{\mt{q}_{m}}=1+\frac{1}{\tau_{m}}+\frac{\gamma}{3}\;\;\;{\rm or}\;\;\;
\frac{1}{\mt{p}_{m}}+\frac{1}{a_{2}}=1+\frac{1}{\tau_{m}}+\frac{\gamma}{3}
\]
which is exactly~\eqref{scaling-relation}. Parallel to~\eqref{dual-w-est}, we have  
\[
\begin{split}
&{\Big |}\int  Q^{+}(f,g)(v)\; \psi (v) dv {\Big |}\\
&\leq C{\big (} \|f_{\ell_{1}}\|_{L^{\mt{p}_{m}}_{v}(\R^{3})}\|g\|_{L^{a_{2}}_{v}(\R^{3})} + 
\|f\|_{L^{a_{1}}_{v}(\R^{3})}\|g_{\ell_{1}}\|_{L^{\mt{q}_{m}}_{v}(\R^{3})} {\big)} \|\psi_{-\ell_{1}}\|_{L^{\mt{r}'_{m}}_{v}(\R^{3})}\\
&\leq C {\big (} \|f_{\ell_{1}}\|_{L^{\mt{p}_{m}}_{v}(\R^{3})}\|g_{\ell}\|_{L^{\mt{q}_{m}}_{v}(\R^{3})} \|\lr{v}^{-\ell_{1}}\|_{L^{m}_{v}(\R^{3})}\\
&\hskip2cm+ \|f\|_{L^{\mt{p}_{m}}_{v}(\R^{3})}\|\lr{v}^{-\ell_{1}}\|_{L^{m}_{v}(\R^{3})}\|g_{\ell_{1}}\|_{L^{\mt{q}_{m}}_{v}(\R^{3})} {\big)} \|\psi_{-\ell_{1}}\|_{L^{\mt{r}'_{m}}_{v}(\R^{3})}\\
&\leq C \|f_{\ell_{1}}\|_{L^{\mt{p}_{m}}_{v}(\R^{3})}\|g_{\ell_{1}}\|_{L^{\mt{q}_{m}}_{v}(\R^{3})} \|\psi_{-\ell_{1}}\|_{L^{\mt{r}'_{m}}_{v}(\R^{3})},
\end{split}
\]
where we used the condition $\ell m>3$ in the last inequality. By duality, we conclude~\eqref{W-convolution}. 
\end{proof}

\section{Decomposition of  $\mathbb{T}$ }\label{Section 3}

In this section, we reduce the  boundedness of $\mathbb{T}:  L^{\widetilde{p}}\rightarrow L^{p}$ 
 to $L^{p}\rightarrow L^{p}$ boundedness of  two operators $T_{A}$ and $T_{B}$ which will be 
 proved in Section 4 and Section 5 respectively. 

Recall that 
\begin{equation}\label{Def:T-2}
\mathbb{T} h(x)=|x|^{\gamma}\int_{\omega\in S^2_+} b(\cos\theta) h(x-(x\cdot\omega)\omega) d\omega, 
\;0\leq\gamma\leq 1
\end{equation}
where 
$\cos\theta={(x\cdot\omega)}/{|x|}, x\neq 0$, $x=|x|(0,0,1)$ and 
$\omega=(\cos\varphi\sin\theta,\sin\varphi\sin\theta,\cos\theta)$, $0\leq\theta\leq \pi/2 $. 
The main result of this section is the following. 

 \begin{thm}\label{Main-result-1}
Let $\mathbb{T}$ be the operator defined by~\eqref{Def:T-2}. Suppose $p$ and $\widetilde{p}$ satisfy 
\[
 \frac{\gamma}{2}\leq \frac{1}{p}\leq 1-\frac{\gamma}{2},\;
\frac{1}{\widetilde{p}}=\frac{1}{p}+\frac{\gamma}{3},
\]
and both $\leq$ are replaced by $<$ when $\gamma=0$, 
then we have
\begingroup
\large
\begin{equation}\label{E:R-ineq-pre}  
\| \mathbb{T} h\|_{L^p}\leq C \|h\|_{L^{\widetilde{p}}}.
\end{equation}
\endgroup
\end{thm}

To prove Theorem~\ref{Main-result-1}, we first reduce it to a $L^{p}$ to $L^{p}$ estimate.  
Recall that, by inverse Fourier transform, we have 
\[
\begin{split}
\mathbb{T} h(x) & =|x|^{\gamma}\int_{\omega\in S^2_+} b(\cos\theta)h(x-(x\cdot\omega)\omega) d\omega \\
&=(2\pi)^{-3}\int_{\mathbb{R}^3} e^{i{x\cdot\xi}} {A}(x,\xi) \widehat{h}(\xi) d\xi
\end{split}
\]
where 
\begin{equation}\label{A-x-xi}
{A}(x,\xi)=|x|^{\gamma}\int_{\omega\in S^2_+} e^{-i(x\cdot\omega)(\xi\cdot\omega)} b(\cos\theta) d\omega.
\end{equation}
We define 
\begin{equation}\label{symbol-a-1}
a(x,\xi)=A(x,\xi)|\xi|^{\gamma}=(|x||\xi|)^{\gamma}\int_{\omega\in S^2_+} e^{-i(x\cdot\omega)(\xi\cdot\omega)} 
b(\cos\theta) d\omega,
\end{equation}
and define 
\begin{equation}\label{Define-T}
T h(x)=(2\pi)^{-3}\int_{\mathbb{R}^3} e^{i{x\cdot\xi}} {a}(x,\xi) \widehat{h}(\xi) d\xi.
\end{equation}
Then we can rewrite 
\begin{equation}\label{T-and-T}
\mathbb{T} h(x) =T\circ (-\Delta)^{-\frac{\gamma}{2}} h(x)
\end{equation}
where $(-\Delta)^{-s/2}:=\mathcal{I}_{s}$ with $0<{\rm Re}\;s <\infty$ is the Riesz potential operator 
of order $s$. We recall the Sobolev embedding theorem concerning $\mathcal{I}_{s}$.
\begin{thm}[Theorem 1.2.3 of~\cite{Gra14}]\label{Embedding}
Let $s$ be a real number, with $0<s<n$, and let $1<p<q<\infty$ satisfy
\[
\frac{1}{p}-\frac{1}{q}=\frac{s}{d}.
\] 
Then there exist constants $C(d,s,p), C(s,d)<\infty$ such that for all $f$ belong to Schwarz space 
$\mathcal{S}(\mathbb{R}^{d})$ we have 
\[
\| \mathcal{I}_{s}(f)\|_{L^{q}}\leq C(d,s,p) \|f\|_{L^{p}},
\]
and 
\[
\| \mathcal{I}_{s}(f)\|_{L^{\frac{d}{d-s},\infty}}\leq C(d,s,p) \|f\|_{L^{1}}.
\]
Consequently $\mathcal{I}_{s}$ has a unique extension on $L^{p}(\mathbb{R}^{d})$ for all 
$p$ with $1\leq p<\infty$ such that the preceding estimates are valid.
\end{thm}

By Theorem~\ref{Embedding} and~\eqref{T-and-T}, the Theorem~\ref{Main-result-1} is  equivalent to the following. 
\begin{thm}\label{Main-result-2}
Let $T$ be the operator defined by~\eqref{Define-T} and~\eqref{symbol-a-1}. Suppose $p$ satisfies 
\begin{equation}\label{T-p-range}
 \frac{\gamma}{2}\leq \frac{1}{p}\leq 1-\frac{\gamma}{2},
\end{equation}
and both $\leq$ are replaced by $<$ when $\gamma=0$,
then we have
\begingroup
\large
\begin{equation}\label{E:R-ineq}  
\| T h\|_{L^p}\leq C \|h\|_{L^{p}}.
\end{equation}
\endgroup
\end{thm}
\begin{proof}
Recall that
\[
T h(x)
=(2\pi)^{-3}\int_{\mathbb{R}^3} e^{i{x\cdot\xi}}\; {a}(x,\xi) \;\widehat{h}(\xi) d\xi, \tag{\ref{Define-T}}
\]
where 
\[
{a}(x,\xi)=(|x||\xi|)^{\gamma}\int_{\omega\in S^2_+} e^{-i(x\cdot\omega)(\xi\cdot\omega)} b(\cos\theta) d\omega.
\tag{\ref{symbol-a-1}}
\]

As in~\cite{JC20}, we can give the function $a(x,\xi)$ a  concrete description by decomposing
the phase space $(x,\xi)$. 

First of all, we need a dyadic decomposition in the interval $(0,\pi)$ which is constructed below.
Let $\zeta(\theta)\in C^{\infty}$ be supported in the interval  $(\pi/8,\pi/2)$ and satisfy $\sum_{z\in\mathbb{Z}} \zeta(2^{-z}\theta)=1$
for all $\theta>0$. Let $\zeta_n(\theta)=\zeta(2^n \theta)$ $n\in\mathbb{N}$,  and $\widetilde{\zeta}(\theta)=1-\sum_{n\in\mathbb{N}}
\zeta_n(\theta)$.  Define $\zeta_0(\theta)$ equals $\widetilde{\zeta}(\theta)$, if $0< \theta\leq \pi/2$ and $\zeta_0(\theta)=\zeta_0(\pi-\theta)$, 
if $\pi/2<\theta<\pi$. This extension of  $\zeta_0$ from $0<\theta\leq \pi/2$
to $\pi/2<\theta<\pi$ by reflection keeps $\zeta_0$ a smooth function since $\widetilde{\zeta}$ equals $1$ near $\pi/2$. 
We also define $\zeta_{-n}(\theta)=\zeta_n(\pi-\theta)$ for $n\in\mathbb{N}$. Then we have the dyadic decomposition 
$1=\zeta_0(\theta)+\sum_{n\in\mathbb{N}}(\zeta_n(\theta)+\zeta_{-n}(\theta))$ in the interval $\theta\in (0,\pi)$. 
Abuse the notations, we define
\begin{equation}\label{D:chi-j}
\zeta_{0}(x,\xi)=\zeta_{0}(\arccos(\frac{x\cdot\xi}{|x||\xi|}))\;,\;\zeta_{\pm n}(x,\xi)=\zeta_{\pm n}(\arccos(\frac{x\cdot\xi}{|x||\xi|}))
\;,\;n\in\mathbb{N}.
\end{equation}
Thus the supports of $\zeta_0(x,\xi),\zeta_{n}(x,\xi),\zeta_{-n}(x,\xi)$ lie respectively
in the cones 
\begin{equation}\label{D:Gamma-n}
\begin{split}
 & \Lambda_0={\Big \{} (x,\xi)|\;  \frac{\pi}{8} < \arccos(\frac{x\cdot\xi}{|x||\xi|}) 
 < {\pi}-\frac{\pi}{8} {\Big \}}, \\
 & \Lambda_{n}= {\Big \{} (x,\xi)|\;  \frac{\pi}{2^{n+3}} < \arccos(\frac{x\cdot\xi}{|x||\xi|}) < \frac{\pi}{2^{n+1}} {\Big \}}\\
 & \Lambda_{-n}={\Big \{} (x,\xi)|\; \pi(1-\frac{1}{2^{n+1}}) < \arccos(\frac{x\cdot\xi}{|x||\xi|})< \pi(1-\frac{1}{2^{n+3}}) {\Big \}}.
\end{split}
\end{equation}

Define
\begin{equation}\label{D:a-z}
a_{z}(x,\xi)=\zeta_{z}(x,\xi)a(x,\xi),\;z\in\mathbb{Z}
\end{equation}
and write $a(x,\xi)=\sum_{z\in\mathbb{Z}}a_{z}(x,\xi)$.
Then we have $T=\sum_{z\in\mathbb{Z}} T_z$ where
\begin{equation}\label{D:T-z}
T_{z} h(x)=(2\pi)^{-3}\int_{\mathbb{R}^3} e^{i{x\cdot\xi}} a_{z}(x,\xi) \;\widehat{h}(\xi) d\xi,\;z\in\mathbb{Z}.
\end{equation}
Each $T_{z}$ is in fact the sum of two FIOs and a degenerate operator which means that we need a 
further decomposition for each of $a_z(x,\xi)$.  

First we introduce a dyadic partition of unity on $\mathbb{R}^3\setminus\{0\}$. 
We let $\rho\in C^{\infty}(\mathbb{R}),\; 0\leq \rho \leq 1$ be supported in 
the open interval $(4,16)\subset\mathbb{R}$
and satisfy 
\begin{equation}\label{partition-unity-rho}
1=\sum_{k\in\mathbb{Z}}\rho(2^{-k} r),\; r>0
\end{equation}
with property ${\rm supp}\;\rho(2^{j} r)\cap\; {\rm supp}\;\rho(2^k r)=\emptyset $ if $|j-k|\geq 2$.
For $x\in\mathbb{R}^3$ and $k\in\mathbb{Z}$ we define  
\begin{equation}\label{D:chi-k}
\rho_{k}(x)=\rho(2^{-k}|x|)
\end{equation}
 where $\rho$ is defined 
in~\eqref{partition-unity-rho}, then we have dyadic partition of unity on $\R^{3}\setminus\{0\}$,
\begin{equation}\label{D:chi-k-sum}
1=\sum_{k\in \mathbb{Z}} \rho_{k}(x) ,\; x\in\mathbb{R}^3\setminus  0.
\end{equation}
Thus we have a dyadic partition of unity on the phase space $\{(x,\xi)|x,\xi\in\mathbb{R}^3\setminus\{0\} \}$,
\[
1=\sum_{j\in \mathbb{Z}} \rho_{j}(x) \sum_{l\in \mathbb{Z}} \rho_{l}(\xi) ,\; x\neq0,\; \xi\neq 0.
\]

For each $a_z(x,\xi)$, we split it into two parts by writing 
\begin{equation}\label{a-z-decom}
\begin{split}
a_z(x,\xi)&=a_z(x,\xi)\sum_{j+l\geq 2|z|+2} \rho_{j}(x)\rho_{l}(\xi)+a_z(x,\xi)\sum_{j+l\leq 2|z|+1} 
\rho_{j}(x)\rho_{l}(\xi)\\
&:=\rho_{z,I}(x,\xi)\cdot a(x,\xi)+\rho_{z,II}(x,\xi)\cdot a(x,\xi).
\end{split}
\end{equation}
Note that we have  
\begin{equation}\label{supp-A-B}
\begin{split}
&{\rm supp}\; \rho_{z,I}(x,\xi)\subseteq\{(x,\xi)| (x,\xi)\in \Lambda_z\;{\rm and}\; |x||\xi|\geq 2^{2|z|}
\cdot 64 \},\\
&{\rm supp}\; \rho_{z,II}(x,\xi)\subseteq\{(x,\xi)| (x,\xi)\in \Lambda_z\;{\rm and}\; |x||\xi|\leq 2^{2|z|}\cdot 
512 \}.
\end{split}
\end{equation}
Then we decompose each $T_z$, defined in~\eqref{D:T-z}, into two parts as
\begin{equation}
T_z=T_{z,I}+T_{z,II}
\end{equation}
according to~\eqref{a-z-decom}. And we let 
\begin{equation}\label{T-A-B}
T_A=\sum_{z\in\mathbb{Z}} T_{z,I},\;T_B=\sum_{z\in\mathbb{Z}} T_{z,II}.
\end{equation}
We should show that  both $T_A$ and $T_B$ are $L^p$ bounded in the 
section~\ref{section-4} and section~\ref{section-5} respectively. 
Before that, we continue the description of $T_A$ and $T_{B}$.

Let $\theta_{0}$ be the angle spanned by $x$ and $\xi$, i.e.,
\begin{equation}\label{theta-0}
\theta_{0}=\arccos(\frac{x\cdot\xi}{|x||\xi|}).
\end{equation}
The  decomposition so far actually follows the following guideline. 
\begin{equation}\label{D:criterion}
\begin{array}{ll}
|x||\xi|\cos^2(\theta_0/2)\sin^2(\theta_0/2)>C_1>1& {\rm on}\; \cup_{z} {\rm supp\;} \rho_{z,I}:={\rm region\; I}  \\
|x||\xi|\cos^2(\theta_0/2)\sin^2(\theta_0/2)<C_2  &{\rm on}\; \cup_{z} {\rm supp\;} \rho_{z,II}:={\rm region\;II}
\end{array}
\end{equation}
for fixed constants $C_1,C_2$. And this criterion is based on the fact that $T_{z,I}$ is a sum of two FIOs
while $T_{z,II}$ degenerates since the stationary phase formula is not applicable on region II.

On the region I, by~\cite{JC20} (see pages 4087-4088) or~\cite{JC12, Lio94}, we have that 
\begin{equation}\label{E:a-I}
\begin{split}
a(x,\xi)&=c_1 e^{-i|x||\xi|\sigma_+(x,\xi)} \{p_{+}(x,\xi) +O_{+}(|x|^{\gamma-2}|\xi|^{\gamma-2})\} \\
& +c_2 e^{-i|x||\xi|\sigma_-(x,\xi)} \{ p_{-}(x,\xi) + O_{-}(|x|^{\gamma-2}|\xi|^{\gamma-2})\} \\
&+ s(x,\xi)
\end{split}
\end{equation}
where 
\begin{equation}\label{principle-symbol}
p_{+}(x,\xi)=\sin(\frac{\theta_0}{2})(|x||\xi|)^{\gamma-1},\;p_{-}(x,\xi)=\cos(\frac{\theta_0}{2})(|x||\xi|)^{\gamma-1}
\end{equation}
\[
\sigma_{\pm}(x,\xi)=\frac{1}{2}(\frac{x\cdot\xi}{|x||\xi|}\pm 1),
\]
$s \in SG^{-\infty,-\infty}$ is a SG symbol of the smooth operator and $p_{\pm}\in SG^{\gamma-1,\gamma-1}
$ are SG symbols of order $(\gamma-1,\gamma-1)$ 
satisfying
\[
|\partial^{\alpha}_{x}\partial^{\beta}_{\xi} p_{\pm}(x,\xi)|\leq C_{\alpha\beta}\lr{x}^{(\gamma-1)-|\alpha|}\lr{\xi}^{(\gamma-1)-|\beta|}.
\]
We remark that it is pointed in~\cite{JC20} our that $\sin(\theta_{0}/2),\;\cos(\theta_0/2)$ are symbols of order $0$.
The big $O_{\pm}$ of~\eqref{E:a-I} can be calculated explicitly using stationary phase 
asymptotic formula~\eqref{F:stationary-formula} which can be found in~\cite{JC20}.  
More precisely, for any $N>2$
\begin{equation}\label{E:lower-order}
\begin{split}
&{\Big |}\;O_{\pm}(|x|^{\gamma-2}|\xi|^{\gamma-2})- \\ 
&\sum_{k=2}^{N} \cos(\theta_0/2)\sin({\theta_0}/{2})\frac{q_{\pm k}(\cos(\theta_0/2),\sin({\theta_0}/{2}))} 
{(|x||\xi|\cos^2(\theta_0/2)\sin^2({\theta_0}/{2}))^{k-1}}(|x||\xi|)^{\gamma-1} {\Big |} \\
&\leq C_{\pm}(|x||\xi|)^{\gamma-N-1}
\end{split}
\end{equation}
where $q_{\pm k}(t,s)$ are  polynomials of $(t,s)$. 

There are two useful observations for the later analysis. 
First of all, take $N=5$, we note that $C_{\pm}(|x||\xi|)^{\gamma-6}$
in~\eqref{E:lower-order} and $s(x,\xi)$ in~\eqref{E:a-I} are integrable when they are regarded as part of the 
integral kernel, defined on region I, of the operator $T_{A}$. Thus the contribution from these two parts 
to the operator $T_{A}$ is $L^{p}$ bounded. 
Secondly, each term in the summation in~\eqref{E:lower-order}
enjoys better decay than principle symbols in~\eqref{principle-symbol} when $\theta_{0}$ thens to $0$ or $\pi$.
The better decay of former still holds even we compare the $x$ or $\xi$ derivatives of both.   Therefore we
conclude that it suffices to regard the $T_{z,I}$ whose symbols are given by~\eqref{principle-symbol} only, 
please also refer to the argument in the following two sections. 

According to above observations, it is harmless  to write 
\begin{equation}\label{E:a-A-I}
\rho_{z,I}\cdot a(x,\xi)=c_1 e^{-i|x||\xi|\sigma_+(x,\xi)} p_{z+}(x,\xi)
+c_2 e^{-i|x||\xi|\sigma_-(x,\xi)} p_{z-}(x,\xi)\\
\end{equation}
where 
\begin{equation}\label{p-z-pm}
p_{z\pm}(x,\xi)=\rho_{z,I}(x,\xi) \cdot p_{\pm}(x,\xi).
\end{equation}
We note that $\rho_{z,I}(x,\xi)\in SG^{0,0}$ by the definition. 

Then we write 
\begin{equation}\label{T-z-1}
\begin{split}
& T_{z,I}\; h(x):= T_{z+,I} \;h(x)+ T_{z-,I}\; h(x) \\
& = \int_{\R^{3}} e^{i \psi_{+}(x,\xi)} p_{z+}(x,\xi) \;\widehat{h}(\xi) d\xi  
+ \int_{\R^{3}} e^{i \psi_{-}(x,\xi)} p_{z-}(x,\xi) \;\widehat{h}(\xi) d\xi 
\end{split}
\end{equation}
where
\begin{equation}
\psi_{\pm}(x,\xi)=\frac{1}{2}(x\cdot\xi\mp |x||\xi|).
\end{equation}
With these notations, our estimate for $T_{A}$  is stated in the 
Proposition~\ref{P-T-A} of the section~\ref{section-4}.

For $T_B=\sum_{z\in\mathbb{Z}} T_{z,II}$, we have
\[
T_{z,II} h(x) =(2\pi)^{-3}\int_{\mathbb{R}^3} e^{ix\cdot\xi} \;\psi_{z,II}(x,\xi)\cdot a(x,\xi)
 \;\widehat{h}(\xi) d\xi.
\]
where 
\[
{\rm supp}\; \psi_{z,II}(x,\xi)\subseteq\{(x,\xi)| (x,\xi)\in\Lambda_z \;{\rm and}\;  
|x||\xi|\leq 512\cdot 2^{2|z|}\},
\]
and
\[
a(x,\xi)=(|x||\xi|)^{\gamma}\int_{S^{2}_{+}} e^{i(x,\omega)(\xi,\omega)} b(\cos\theta)d\omega.
\]
As we point out before that on the region II of phase space $(x,\xi)$, the stationary phase 
formula does not work for the calculation of $a(x,\xi)$. 
However by the estimate in~\cite{JC20} (see the calculation before equation (3.36) there), we have 
\begin{equation}\label{a-region-2-pre}
a(x,\xi)= C\cdot\sin\theta_0\cos\theta_0\cdot (|x||\xi|)^{\gamma-1},\;{\rm if}\;(x,\xi)\in\;{\rm region\;II},
\end{equation}
where $C$ is independent of $x,\xi$. Here $\cos\theta_0$ comes from angular function of the 
collision kernel and $\sin\theta_0$ is the Jacobian of the spherical coordinate.  
The estimate~\eqref{a-region-2-pre} is suffice for us to 
show the $L^{p}$ boundedness of $T_{B}$.  Our estimate for $T_{B}$  is stated in 
Proposition~\ref{P-T-B} of the section~\ref{section-5}.

The estimate for $T$ follows by combing the results of  Proposition~\ref{P-T-A} and Proposition~\ref{P-T-B}.
\end{proof}

The decomposition of phase space $(x,\xi)$ above is more brief than we did in~\cite{JC20}.  
The remark below explains their difference.
\begin{rem}
We split the phase space into two regions in~\eqref{D:criterion} and define operators $T_{A},\; T_{B}$
accordingly. Please note the region I in~\eqref{D:criterion} is the same as the region I plus region II
in (3.13) of~\cite{JC20} and region II in~\eqref{D:criterion} is the region III in (3.13) in~\cite{JC20}.  
Therefore $T_{A}$  in this paper is equivalent to $T_{A}+T_{B,1}+T_{B,2}$ defined after (3.5) of~\cite{JC20}
and $T_{B}$ in this paper is equal to $T_{C}$ of ~\cite{JC20}.  
This simplication is due to that, in~\cite{JC20}, the $L^{2}$ boundedness of $T_{B,1}$ and $T_{B,2}$ is
proved by using a scaling argument and $L^{2}$ boundedness of $T_{A}$ where the scaling argument
is essential the same as the proof of the Proposition~\ref{FIO-localized} below. 
The Proposition~\ref{FIO-localized} proved by Coriasco and Ruzhansky~\cite{CR14} is used in this paper in order to give a shorter proof of the $L^{p}$
 boundedness of current $T_{A}$. 
\end{rem}

\section{$L^p$ boundedness of $T_A$}\label{section-4}

The goal of this section is proving the following.  

\begin{thm}\label{P-T-A}
Recall $\theta_{0}=\arccos(x\cdot\xi/|x||\xi|)$.  Let $T_{A}=\sum_{z\in\mathbb{Z}} T_{z,I}$ where
\begin{equation}\label{T-z-I-sum}
\begin{split}
& T_{z,I}\; h(x):= T_{z+,I} \;h(x)+ T_{z-,I}\; h(x) \\
& = \int_{\R^{3}} e^{i \psi_{+}(x,\xi)} p_{z+}(x,\xi) \;\widehat{h}(\xi) d\xi  
+ \int_{\R^{3}} e^{i \psi_{-}(x,\xi)} p_{z-}(x,\xi) \;\widehat{h}(\xi) d\xi, 
\end{split}
\end{equation}
with 
\begin{equation}
\psi_{\pm}(x,\xi)= \frac{1}{2}(x\cdot\xi\mp |x||\xi|),
\end{equation}
\begin{equation}\label{T-z-symbol}
\begin{split}
& p_{z+}(x,\xi)=\rho_{z,I}(x,\xi)\cdot \sin(\frac{\theta_0}{2})(|x||\xi|)^{\gamma-1}\in {\rm SG}^{\gamma-1,\gamma-1}, \\
& p_{z-}(x,\xi)=\rho_{z,I}(x,\xi)\cdot \cos(\frac{\theta_0}{2})(|x||\xi|)^{\gamma-1}\in {\rm SG}^{\gamma-1,\gamma-1},
\end{split}
\end{equation}
and 
\[
{\rm supp}\; p_{z\pm}(x,\xi)={\rm supp}\; \rho_{z,I}(x,\xi)\subseteq\{(x,\xi)| (x,\xi)\in \Lambda_z\;{\rm and}\; |x||\xi|\geq 2^{2|z|}\cdot 64 \}.
\]
Suppose $p$ satisfies 
\[
 \frac{\gamma}{2}\leq \frac{1}{p}\leq 1-\frac{\gamma}{2},
\]
and both $\leq$ are replaced by $<$ when $\gamma=0$,
then we have
\begingroup
\large
\begin{equation}\label{E:R-ineq-A}  
\| T_{A} \;h\|_{L^p}\leq C \|h\|_{L^{p}}.
\end{equation}
\endgroup
\end{thm}

Theorem~\ref{P-T-A} will be proved after several reductions. We note that
\begin{equation}
\begin{split}
& |\det\partial_x\partial_{\xi} {\Big [} \psi_{\pm}(x,\xi)){\Big ]}|= 
|\det  {\Big [} \partial^2\psi_{\pm}(x,\xi)/\partial_{x_i}\partial_{\xi_j}  {\Big ]} |\\
&=|\det\frac{1}{2}{\Big [}I\mp\frac{x}{|x|}\otimes\frac{\xi}{|\xi|}{\Big ]}|
= (\frac{1}{2})^3|1\mp\cos\theta_0|
\end{split}
\end{equation} 
which are positive on support of $\rho_{z,I}$ for each $z$. Since each of $T_{z\pm,I}$ is a global FIO, 
we may apply the result  of~\cite{CR14} about the $L^{p}$ boundedness of one global FIO. 
\begin{thm}[Theorem 2.6 of~\cite{CR14}]\label{FIO-global}
Assume $1<p<\infty$, $\mathscr{T}$ is a global FIO takes the form:
\[
\mathscr{T} h(x)=\int_{\mathbb{R}^{d}} e^{i\psi(x,\xi)} p(x,\xi)\; \widehat{h}(\xi) d\xi
\]
with $p\in S^{m_{1},m_{2}}$ and  where
 \begin{equation}\label{FIO-global-symbol}
 m_{1},m_{2}\leq -(d-1){\big |} \frac{1}{p}-\frac{1}{2} {\big |}.
 \end{equation}
 The phase function $\psi$ satisfies
 \begin{equation}
 \begin{split}\label{psi-upper}
 & \psi(x,\tau\xi)=\tau\psi(x,\xi)\;{\rm for}\;\tau>0,\xi\neq 0,\; \partial^{\alpha}_{x}\psi(x,\xi)\leq C_{\alpha} \lr{x}^{1-|\alpha|}|\xi|, \\
 & \lr{\nabla_{\xi}\psi(x,\xi)}\approx \lr{x},\;\lr{\nabla_{x}\psi(x,\xi)}\approx \lr{\xi},
 \end{split}
 \end{equation}
 and there exists $c>0$ such that
\begin{equation}\label{Hessian-low}
|\det  {\Big [} \partial_{x}\partial_{\xi} \psi(x,\xi) {\Big ]} | 
=|\det  {\Big [} \partial^2\psi(x,\xi)/\partial_{x_i}\partial_{\xi_j}  {\Big ]} |\geq c>0 
\end{equation}
for all $(x,\xi)\in {\rm supp}\; p(x,\xi)$.
Then $\mathscr{T}$, initially defined on $\mathcal{S}(\R^{d})$, 
extends to a bounded operator from $L^{p}(\R^{d})$ to itself.  
\end{thm}

Note that $T_{z\pm,I}$ are global FIOs whose symbols lie in $S^{\gamma-1,\gamma-1}$ and phase
functions satisfy the conditions described in Theorem~\ref{FIO-global}.  
The condition~\eqref{FIO-global-symbol} of Theorem~\ref{FIO-global} implies the following result.
\begin{cor}
$T_{z\pm,I},\;z\in \mathbb{Z}$ are $L^{p}(\R^3)\rightarrow L^{p}(\R^3)$ bounded when 
\[
\frac{\gamma}{2}\leq \frac{1}{p}\leq 1-\frac{\gamma}{2},
\]
and both $\leq$ are replaced by $<$ when $\gamma=0$.
\end{cor}
On the other hand, to prove the $L^{p}$ boundedness of $T_{A}$, we need to show 
the series of norms $\|T_{z\pm,I}\|_{L^{p}\rightarrow L^{p}}$ converges.  
The proof of Theorem~\ref{FIO-global} suggests that it suffices to consider the 
same problem for localized $T_{z\pm,I}$ as we should explain below.   
A key tool is the Proposition~\ref{FIO-localized} stated below.  Before that, we recall the dyadic 
decomposition~\eqref{D:chi-k-sum} in $\R^{3}$,  
\[
1=\sum_{k\in\mathbb{Z}} \rho_{k}(x),\;x\neq 0
\] 
where $\rho_{k}(x)=\rho_{0}(2^{-k} x)$ and $\rho_{0}$ is supported in the shell $ 2^{2}\leq |x|\leq 2^{4}$.

\begin{prop}[Proposition 4.1 of~\cite{CR14}]\label{FIO-localized} 
Let $\mathscr{T}$ be the operator given in Theorem~\ref{FIO-global}. Then for any $k\geq -4$, 
\begin{equation}\label{Local-FIO-est}
\| \rho_{k}\cdot \mathscr{T} u\|_{L^p}\leq C\|u\|_{L^p}.
\end{equation}
where constant $C$ depends only on $\rho_{0}$, on upper bounds for a finite number of the constants 
in the estimate satisfied by $p(x,\xi)$ and $\psi(x,\xi)$, and on the lower bound $c$ for the determinant 
of the mixed Hessian of $\psi$ in~\eqref{Hessian-low}.
\end{prop}

We note that the constant $C$ in~\eqref{Local-FIO-est}  is independent of the choice of $k\geq -4$ and 
the statement after ~\eqref{Local-FIO-est} is exactly the factors which determine the $L^{p}$ bound of a 
local FIO. 

The  Proposition~\ref{FIO-localized} holds for the following reasons 
(Please also see ~\cite{CR14} for the details.):  Let $U_{\lambda}$ be the dilation operator 
defined by $U_{\lambda} u(x)=u (\lambda x)$ and write
\[
\rho_{k}\cdot  \mathscr{T}=U_{2^{-k}}\circ \mathscr{T}^{k}\circ U_{2^{k}},
\]
where 
\[
 \mathscr{T}^{k} u(x)=\int_{\mathbb{R}^{3}} e^{i\psi (2^{k} x, 2^{-k} \xi)} \rho_{0} (x)\; p( 2^{k} x, 2^{-k}\xi)\; \widehat{u}(\xi)\; d\xi.
\]
Then  $\| \rho_{k}\cdot  \mathscr{T} u\|_{L^{p}}=\|  \mathscr{T}^{k} u\|_{L^{p}}$. On the support of $\rho_{0}$, 
the condition $p\in S^{m_{1},m_{2}}$ implies 
\begin{equation}\label{local-cond-1}
|\partial_{x}^{\alpha} \partial_{\xi}^{\beta} p(2^{k}x,2^{-k}\xi) |\leq C_{\alpha,\beta}\lr{\xi}^{m_{2}-|\beta|}.
\end{equation}
Also  
\begin{equation}\label{local-cond-2}
|\partial_{x}^{\alpha} \partial_{\xi}^{\beta} \psi(2^{k}x,2^{-k}\xi) |\leq C'_{\alpha,\beta}|\xi|^{1-\beta}
\end{equation}
and 
\begin{equation}\label{local-cond-3}
|\det  {\Big [} \partial_{x}\partial_{\xi} \psi(2^{k} x, 2^{-k}\xi) {\Big ]} | 
=|\det  {\Big [} \partial^2\psi(2^{k} x,2^{-k}\xi)/\partial_{x_i}\partial_{\xi_j}  {\Big ]} |\geq c>0. 
\end{equation}
The $L^{p}$ norm of the local FIO is determined by bounds $C_{\alpha,\beta},\;C'_{\alpha,\beta}, c$  
while they are invariant under the scaling (independent of $k$).  
Thus 
\[
\| \rho_{k}\cdot  \mathscr{T} u\|_{L^{p}}=\| \mathscr{T}^{k} \|_{L^{p}}\leq C\|u\|_{L^{p}},\;k\geq -4,
\]
where the inequality holds by the theory of local FIO which we should look over more details later.

Combing Proposition~\ref{FIO-localized}  with Littlewood-Paley decomposition, the proof of 
Proposition 4.1 in~\cite{CR14} infers the following. 
\begin{prop}\label{Global-local}
Let $\mathscr{T}$ be the operator given in Theorem~\ref{FIO-global}. There exist $C_{1},C_{2}$
independent of $k\geq -4$ such that 
\begin{equation}
C_{1} \| \rho_{k}\cdot  \mathscr{T} u\|_{L^{p}} \leq  \|  \mathscr{T} u\|_{L^{p}}\leq C_{2}\| \rho_{k}\cdot  \mathscr{T} u\|_{L^{p}}.
\end{equation}
\end{prop} 
Theorem~\ref{FIO-global} is then a consequence of the theory of local FIO.  
Similarly, by Proposition~\ref{Global-local}, Theorem~\ref{P-T-A} is equivalent to the following. 

\begin{thm}\label{local-converge}
Use the same notations as Theorem~\ref{P-T-A}. Then we have 
\[
\| T_{A} \;h\|_{L^p}\leq C \|h\|_{L^{p}},
\]
since $ \|\rho_{k}\cdot T_{z\pm,I}\|_{L^{p}\rightarrow L^{p}}$ 
form a convergent series where $k\geq -4$ could be different for each $z$. 
\end{thm} 
 
In order to prove the Theorem~\ref{local-converge}, we need to review some background knowledge. 
The $L^{p}$ boundedness of one local FIO is proved by applying the complex interpolation to the 
$L^{2}$ and $L^{1}$ estimates for FIOs with corresponding  orders. Then extended the results to $p>2$
by the duality argument. We record the result of Seeger, Sogge and Stein~\cite{SSS91}  only for dimension 
3 in the following which can also be found in Chapter IX of~\cite{Stein93}.
\begin{prop}{\cite{SSS91}}\label{FIO-local}
Let $1<p<\infty$. Suppose  
\begin{equation}\label{D:FIO-2}
\mathscr{T}^{(m)} h(x)=\int_{\mathbb{R}^{3}} e^{i\psi(x,\xi)} p(x,\xi)\; \widehat{h}(\xi) d\xi,
\end{equation}
is a local FIO of order $m$ with $-1<m\leq 0$. 
Let $\mathscr{T}^{(0)}$ be the FIO of order $0$ by replacing $p(x,\xi)$ 
with $p(x,\xi)(1+|\xi|^{2})^{-m/2}$. Then we have 
\[
\|\mathscr{T}^{(0)} h(x)\|_{L^{2}} \leq A^{(0)} \|h\|_{L^{2}}. 
\]
 Let $\mathscr{T}^{(-1)}$  be the FIO of order $-1$ by replacing $p(x,\xi)$ 
with $p(x,\xi)(1+|\xi|^{2})^{-m/2-1}$. Then we have 
\[
\|\mathscr{T}^{(-1)} h(x)\|_{L^{1}} \leq A^{(-1)} \|h\|_{{\mcH}^{1}}. 
\]
where ${\mcH}^{1}$ is the Hardy space of order $1$. Then, $\mathscr{T}^{(m)}$, 
initially  defined on $\mathcal{S}(\R^{3})$, extends to a bounded operator from $L^{p}(\R^{3})$ to itself, 
whenever 
\[
m\leq -2{\Big |} \frac{1}{2}-\frac{1}{p} {\Big |}.
\]  
And 
\begin{equation}\label{interpolation-norm}
\|\mathscr{T}^{(m)} h(x)\|_{L^{p}} \leq (A^{(0)})^{t_{0}} (A^{(-1)})^{1-t_{0}} \|h\|_{{L}^{p}},
\end{equation}
where $t_{0}=2/p$ if $p>2$ and $t_{0}=2[1-(1/p)]$ if $1<p<2$.
\end{prop}

We introduce some notations in order to apply the result of Proposition~\ref{FIO-local}.  
According to Theorem~\ref{P-T-A}, let $\lambda\in\R$ and 
 $z\in\mathbb{Z}$, we define 
 \begin{equation}\label{T-lambda-p}
 T^{(\lambda)}_{z,I}=T^{(\lambda)}_{z+,I}+T^{(\lambda)}_{z-,I},\;{\rm where}\;\;
 T^{(\lambda)}_{z\pm, I}=\int_{\R^{3}} e^{i \psi_{\pm}(x,\xi)} p^{\lambda}_{z\pm}(x,\xi) \;\widehat{h}(\xi) d\xi 
 \end{equation} 
 where 
 \begin{equation}\label{psi-pm}
 \psi_{\pm}(x,\xi)= \frac{1}{2}(x\cdot\xi\mp |x||\xi|),
 \end{equation}
 \begin{equation}\label{symbol-lambda}
 \begin{split}
&  p^{\lambda}_{z+}(x,\xi)=\rho_{z,I}(x,\xi)\cdot \sin(\frac{\theta_0}{2})(|x||\xi|)^{\lambda},\\
& p^{\lambda}_{z-}(x,\xi)=\rho_{z,I}(x,\xi)\cdot \cos(\frac{\theta_0}{2})(|x||\xi|)^{\lambda},
 \end{split}
 \end{equation}
 and recall
\begin{equation}
{\rm supp}\; p^{\lambda}_{z\pm}(x,\xi) \subseteq\{(x,\xi)| (x,\xi)\in \Lambda_z\;{\rm and}\; |x||\xi|\geq 2^{2|z|}\cdot 64 \}.
\end{equation}
Note that our notation means $T_{z\pm, I}=T^{(\gamma-1)}_{z\pm, I}$.

By Proposition~\ref{FIO-local}, in particular~\eqref{interpolation-norm}, the Theorem~\ref{local-converge} 
holds if we can prove the following Proposition~\ref{L-2-theorem} and Proposition~\ref{L-1-theorem}.  
\begin{prop}\label{L-2-theorem}
Let $T^{(0)}_{z\pm,I}$ be defined by~\eqref{T-lambda-p} and $k\geq -4$, then
\begin{equation}\label{L-2-1}
 \|   \rho_{k}\cdot  T^{(0)}_{(z+1)\pm,I}\|_{L^{2}\rightarrow L^{2}} \leq \frac{1}{2}\;
 \|  \rho_{k}\cdot T^{(0)}_{z\pm,I}\|_{L^{2}\rightarrow L^{2}},\;{\rm if}\;z\geq 0.
\end{equation}
and 
\begin{equation}\label{L-2-2}
 \|  \rho_{k}\cdot T^{(0)}_{(z-1)\pm,I}\|_{L^{2}\rightarrow L^{2}} \leq \frac{1}{2}\;
 \|  \rho_{k}\cdot T^{(0)}_{z\pm,I}\|_{L^{2}\rightarrow L^{2}},\;{\rm if}\;z< 0.
\end{equation}
\end{prop}
\begin{proof}
From the Part I and II of the proof of Lemma 3.1 of~\cite{JC20}, we have 
\[
 \|     T^{(0)}_{(z+1)\pm,I}\|_{L^{2}\rightarrow L^{2}} \leq \frac{1}{2}\;
 \|   T^{(0)}_{z\pm,I}\|_{L^{2}\rightarrow L^{2}},\;{\rm if}\;z\geq 0.
\]
and 
\[
 \|   T^{(0)}_{(z-1)\pm,I}\|_{L^{2}\rightarrow L^{2}} \leq \frac{1}{2}\;
 \|   T^{(0)}_{z\pm,I}\|_{L^{2}\rightarrow L^{2}},\;{\rm if}\;z< 0.
\]
The result follows by above relations and  Proposition~\ref{Global-local}.
\end{proof}

\begin{prop}\label{L-1-theorem}
Let $T^{(-1)}_{z\pm,I}$ be defined by~\eqref{T-lambda-p} and $k\geq -4$, then
\[
 \|  \rho_{k}\cdot T^{(0)}_{(z+1)\pm,I}\|_{L^{1}\rightarrow {\mcH}^{1}} \leq (\frac{1}{2})^{1/2}\;
 \|  \rho_{k}\cdot T^{(0)}_{z\pm,I}\|_{L^{1}\rightarrow {\mcH}^{1}},\;{\rm if}\;z\geq 0.
\]
and 
\[
 \|  \rho_{k}\cdot T^{(0)}_{(z-1)\pm,I}\|_{L^{1}\rightarrow {\mcH}^{1}} \leq (\frac{1}{2})^{1/2}\;
 \|  \rho_{k} \cdot T^{(0)}_{z\pm,I}\|_{L^{1}\rightarrow {\mcH}^{1}},\;{\rm if}\;z< 0.
\]
\end{prop}

Thus we conclude the Theorem~\ref{P-T-A} if we can show Proposition~\ref{L-1-theorem},
and which will be done it in the next subsection. 

\begin{rem}\label{shrink-II}
The paragraph after (3.38) in~\cite{JC20} pointed out that the region I  defined by~\eqref{supp-A-B}
and~\eqref{D:criterion}  can be enlarged a bit by replacing $2^{2|z|}$ with $2^{2|z|(1-\delta)}$ 
in~\eqref{supp-A-B} where $\delta$ is a small positive number, while the convergence of $L^{2}$ bounds
still holds. Please see also the discussion after~(3.33) in~\cite{JC20}.
Now we choose it to be $0<\delta<1/50$.
Then the ratio $1/2$ and $(1/2)^{1/2}$ in Proposition~\ref{L-2-theorem} and Proposition~\ref{L-1-theorem} will be 
replaced with $(1/2)^{3/5}$ and $(1/2)^{1/10}$ respectively. 
The shrink of region II will be useful later for the $L^{p}$ estimate of $T_{B}$ which is similar to what we did 
in~\cite{JC20} for the $L^{2}$ estimate.
\end{rem}

\subsection{Proof of Proposition~\ref{L-1-theorem}}\label{H-1-subsection}\hfill

Although the enlargement of the region I in remark~\eqref{shrink-II} is valid for our estimates in the previous 
subsection, we should retain the original definition of the region I for the clear of representation and easier understanding. 
However we should put a remark in the end of this subsection to conclude how the enlargement  of the region
I affects the estimates in this subsection.  

We should employ the  $L^{2}$ result to calculate bounds for the $L^{1}$ estimate, i.e., 
$C_{z\pm},\;z\in\mathbb{Z}\setminus\{0\}$ for the  estimates
\begin{equation}\label{L-1-estimate}
\| {\rho}_{1}\cdot T^{(-1)}_{z\pm,I}\; u\|_{L^1}  \leq  C_{z\pm} \| u\|_{\mcH^{1}} 
\end{equation}
where $\mcH^{1}$ is the Hardy space of order $1$ (see Definition~\ref{Def-atom} below). 
Apparently, our goal is to show that $C_{z\pm}$ forms a convergent series. 
Therefore we need to review the proof of $\mcH^{1}\rightarrow L^{1}$ estimate for the local 
FIO given in~\cite{SSS91, Stein93}.
We should duplicate the proof  in~\cite{Stein93} for single local FIO, then observe at the same time how $C_{z\pm}$ varies with $z$. 
First of all, let's recall the definition of  the atom and result of the atom decomposition in Hardy space $\mathcal{H}^{1}(\R^n)$.

\begin{defn}\label{Def-atom}
An $\mathcal{H}^1$ atom, is a function $u$ so that
\begin{enumerate}[(i)]
\item $u$ is supported in a ball $D$,
\item $|u|\leq |D|^{-1}$ almost everywhere, and 
\item $\int \;u(x)dx=0$.
\end{enumerate}
\end{defn}

The following atomic decomposition is proved by Coifman~\cite{Coi74} and Latter~\cite{Lat78},
please also see~\cite{Stein93}.
\begin{thm}[Atomic decomposition for $\mathcal{H}^{1}$]\label{T:atomic-decomposition} 
If $\{u_{k}\}$ is a collection of $\mathcal{H}^{1}$ atoms and $\{\alpha_{k}\}$ is a sequence of 
complex number with $\sum|\alpha_{k}|<\infty$, then the series $f=\sum_{k}\alpha_{k}u_{k}$
converges in the sense of distribution and $f\in\mathcal{H}^{1}$ with 
$\|f\|_{\mathcal{H}^{1}}\leq C\sum|\alpha_{k}|$. Conversely, every $f\in\mathcal{H}^{1}$ can 
be written as a sum of $\mathcal{H}^{1}$ atoms, $f=\sum_{k}\alpha_{k} u_{k}$ converges in 
$\mathcal{H}^{1}$ norm, moreover we have $\sum_{k}|\alpha_{k}|<C \|f\|_{\mathcal{H}^{1}}$.
\end{thm}

According to the Theorem~\ref{T:atomic-decomposition}, to prove~\eqref{L-1-estimate}, it suffices to show that 
for an arbitrary atom $u$, it holds that for any  fixed $z$
\begin{equation}\label{Hardy-est}
\int_{\mathbb{R}^{3}} {\big |}\rho_{1}\cdot T^{(-1)}_{z\pm,I}\; u(x) {\big |}  dx \leq C_{z\pm},
\end{equation}
where the constant $C_{z\pm}$ is independent of $u$. Also the Proposition~\ref{L-1-theorem} follows 
if $C_{z\pm}$ form a convergent series. 

Thus Theorem~\ref{L-1-theorem} holds if the following two Propositions hold. We separate them for that 
the first one is easier to prove. 
\begin{prop}\label{L-1-1}
 Suppose an atom $u$ is  supported in a ball $D$ whose radius is larger than 1. Then 
\begin{equation}\label{Hardy-est-1}
\int_{\mathbb{R}^{3}} {\big |}\rho_{1}\cdot T^{(-1)}_{z\pm,I}\; u(x) {\big |}  dx \leq C_{z\pm},
\end{equation}
with 
\begin{equation}\label{L-1-C-1}
C_{(z+1)\pm}\leq \frac{1}{2} C_{z\pm} {\rm if}\; z\geq 0 \;{\rm and}\; C_{(z-1)\pm}\leq\frac{1}{2} 
C_{z\pm}\;{\rm if}\; z< 0.
\end{equation}

\end{prop}

 \begin{prop}\label{L-1-2} 
 Suppose an atom $u$ is  supported in a ball $D$ whose radius is less than 1. Then 
\begin{equation}\label{Hardy-est-2}
\int_{\mathbb{R}^{3}} {\big |}\rho_{1}\cdot T^{(-1)}_{z\pm,I}\; u(x) {\big |}  dx \leq C_{z\pm},
\end{equation}
\begin{equation}\label{L-1-C-2}
C_{(z+1)\pm}\leq (\frac{1}{2})^{1/2}\cdot C_{z\pm} \;\;{\rm if}\; z\geq 0 \;{\rm and}\; C_{(z-1)\pm}\leq(\frac{1}{2})^{1/2} 
\cdot C_{z\pm}\;{\rm if}\; z< 0.
\end{equation}

\end{prop}

\begin{proof}{(Proof of Proposition~\ref{L-1-1})}
Suppose that $u$ is an atom associated to a ball $D$ as the Definition~\ref{Def-atom}.
Since the radius of $D$ exceeds $1$, the proof of~\eqref{Hardy-est-1} is an easy consequence 
 of $L^2$ estimate.  Observe that
\begin{equation}\label{radius-b1}
\int_{\mathbb{R}^{3}} |\rho_{1} \cdot T^{(-1)}_{z\pm,I} \;u | dx \leq C' \|\rho_{1}\cdot  T^{(-1)}_{z\pm,I}(u)\|_{L^{2}}
\leq C_{z\pm}\|u\|_{L^{2}}
\end{equation}
where the first inequality holds because $\rho_{1}\cdot T^{(-1)}_{z\pm,I}(u)$ has fixed compact support 
determined by $\rho_{1}$ ; the second follows from the 
$L^{2}$ boundedness of $\rho_{1}\cdot  T^{(-1)}_{z\pm,I}T$ since the order of symbol is $-1\leq 0$
and thus the $L^{2}$  estimates hold. 
Since $u$ is an atom, $|u(x)|\leq |D|^{-1}$. Note that radius of $D$ is larger than 1, we 
have $\|u\|_{L^{2}}\leq |D|^{-1/2}\leq c$ and~\eqref{Hardy-est-1} holds.  Note that $C_{z\pm}$ 
in~\eqref{radius-b1} comes from $L^{2}$ estimate of $\rho_{1}\cdot  T^{(-1)}_{z\pm,I}(u)$. Therefore, by 
Proposition~\ref{L-2-theorem}, we conclude~\eqref{L-1-C-1}.  
\end{proof}

Before we prove the case radius of $D\leq 1$, Proposition~\ref{L-1-2}, we need following embedding result. 
\begin{prop}[P.399 of~\cite{Stein93}]\label{embedding}
Let $\mathscr{T}$ be a {\bf local} Fourier integral operator takes the form
\begin{equation}\label{D:FIO-3}
\mathscr{T} f(x)=\int_{\mathbb{R}^{d}} e^{i\psi(x,\xi)} p(x,\xi)\; \widehat{f}(\xi) d\xi
\end{equation}
with symbol $p\in S^{m}$
where $-d/2<m< 0$. Then $\mathscr{T}$, initially defined on $\mathcal{S}(\R^{d})$, extends to a bounded operator 
from $L^{p}(\mathbb{R}^{d})$ to $L^{2}(\mathbb{R}^{d})$, if $1/p=1/2-m/d$.
\end{prop}
We also need to sketch the proof of Proposition~\ref{embedding} for the later analysis.
\begin{proof}{(Sketch of the proof of Proposition~\ref{embedding})}
  It is to write 
$\mathscr{T} =\mathscr{T}_{0} \mathscr{T}_{1} $, where $\mathscr{T}_{0}$ is a FIO with symbol 
\[
p_{0}(x,\xi)=p(x,\xi)\cdot (1+|\xi|^{2})^{-m/2}
\]  
which is of order $0$. And $ \mathscr{T}_{1}$ is the pseudo-differential operator with symbol $(1+|\xi|^{2})^{m/2}$.
Note that $\mathscr{T}_{0}:L^{2}(\R^{d})\rightarrow L^{2}(\R^{d})$ is bounded by the theory of local FIO or 
Proposition~\ref{FIO-local}. Also $\mathscr{T}_{1}$ has the representation
\[
\mathscr{T}_{1} h(x)=\int_{\R^{d}} K(x-y) f(y) dy
\]
where locally integrable $K$ satisfies $|K(x,y)|\leq C|y|^{-d-m}$ and $C$ is independent of $x$. 
Thus $\mathscr{T}_{1}:L^{p}(\R^{d})\rightarrow L^{2}(\R^{d})$  comes from the Hardy-Littlewood-Sobolev 
inequality. Then Proposition~\ref{embedding} follows.
\end{proof}

Combining Proposition~\ref{embedding} and Proposition~\ref{L-2-theorem}, we have the following. 
\begin{prop}\label{embedding-ratio}
Let $T^{(-1)}_{z\pm,I}$ be defined by~\eqref{T-lambda-p} and $k\geq -4$, then
\begin{equation}\label{L-2-3}
 \|   \rho_{k}\cdot  T^{(-1)}_{(z+1)\pm,I}\|_{L^{\frac{6}{5}}\rightarrow L^{2}} \leq \frac{1}{2}\;
 \|  \rho_{k}\cdot T^{(-1)}_{z\pm,I}\|_{L^{\frac{6}{5}}\rightarrow L^{2}},\;{\rm if}\;z\geq 0.
\end{equation}
and 
\begin{equation}\label{L-2-4}
 \|  \rho_{k}\cdot T^{(-1)}_{(z-1)\pm,I}\|_{L^{\frac{6}{5}}\rightarrow L^{2}} \leq \frac{1}{2}\;
 \|  \rho_{k}\cdot T^{(-1)}_{z\pm,I}\|_{L^{\frac{6}{5}}\rightarrow L^{2}},\;{\rm if}\;z< 0.
\end{equation}
\end{prop}

Now we are ready to prove the Proposition~\ref{L-1-2}.
\begin{proof}{(Proof of the Proposition~\ref{L-1-2})}
Now we assume $u$ is an atom associated to a ball $D$ whose radius  is $\delta\leq 1$.
Recall~\eqref{T-lambda-p}, 
\begin{equation}
\begin{split}
{\rho}_{1}\cdot T^{(-1)}_{z,I}\; u(x)&= {\rho}_{1}\cdot T^{(-1)}_{z+,I} \;u(x)+ {\rho}_{1}\cdot T^{(-1)}_{z-,I}\; u(x) \\
& = {\rho}_{1}\cdot\int_{\R^{3}} e^{i \psi_{+}(x,\xi)}\; p^{-1}_{z+}(x,\xi)\;\widehat{u}(\xi) d\xi    \\
&\hskip 1cm + {\rho}_{1}\cdot\int_{\R^{3}} e^{i \psi_{-}(x,\xi)} \;p^{-1}_{z-}(x,\xi) \;\widehat{u}(\xi) d\xi , 
\end{split}
\end{equation}
with
\begin{equation}\label{supp-p-z}
{\rm supp}\; p_{z\pm}(x,\xi) \subseteq\{(x,\xi)| (x,\xi)\in \Lambda_z\;{\rm and}\; |x||\xi|\geq 2^{2|z|}\cdot 64 \}.
\end{equation}
We rewrite 
\begin{equation}
\begin{split}
{\rho}_{1}\cdot T^{(-1)}_{z\pm,I} u(x)&=\int_{\R^{3}}\int_{\R^{3}} e^{i [\psi_{\pm}(x,\xi)-y\cdot\xi]}\; \rho_{1}(x)
\; p^{-1}_{z\pm}(x,\xi)\; u(y) d\xi dy \\
&=\int_{R^{3}} K_{z\pm}(x,y) u(y) dy.
\end{split}
\end{equation} 
where
\[
K_{z\pm}(x,y)=\int_{\R^{3}} e^{i [\psi_{+}(x,\xi)-y\cdot\xi]}\; \rho_{1}(x)\; p^{-1}_{z\pm}(x,\xi) d\xi.
\]
We further assume $D$ has center $\overline{y}$, radius $\delta\leq 1$  and denote it by $D(\overline{y}, \delta)$.

To prove~\eqref{Hardy-est-2} and~\eqref{L-1-C-2}, 
the region of integration in~\eqref{Hardy-est} is divided into two parts, i.e., 
\begin{equation}\label{L-1-two-parts-1}
\int_{D^{*}_{z\pm}} {\big |}\rho_{1}\cdot T^{(-1)}_{z\pm,I}\; u(x) {\big |}  dx,
\end{equation} 
and 
\begin{equation}\label{L-1-two-parts-2}
\int_{\prescript{c}{}D^{*}_{z\pm}} {\big |}\rho_{1}\cdot T^{(-1)}_{z\pm,I}\; u(x) {\big |}  dx,
\end{equation} 
where $D^{*}_{z\pm}$ is the region where the phase function of  $T^{(-1)}_{z\pm,I}$ oscillates less   
and depends on region $D$, while $\prescript{c}{}D^{*}_{z\pm}$ is  complement of $D^{*}_{z\pm}$ where 
 phase function of  $T^{(-1)}_{z\pm,I}$ oscillates more and we should use integration by parts to capture the 
 decay of the kernel $K_{z\pm}$. 
The precisely definition of $D^{*}_{z\pm}$ is given in the following paragraphs. 
 
We first estimate integrals in~\eqref{L-1-two-parts-1}.
 
For each $j\in\mathbb{N}$, we consider a roughly equally spaced set of points with grid length $2^{-j/2}$
on the unit sphere $S^{2}$. Namely, we fix a collection $\{\xi_{j}^{\nu}\}_{\nu}$ of unit vectors $|\xi_{j}^{\nu}|=1$,
that satisfy:
\begin{enumerate}[(i)]
\item $|\xi_{j}^{\nu}-\xi_{j}^{\nu'}|\geq 2^{-j/2}$, if $\nu\neq\nu'$.
\item If $\xi\in S^{2}$, then there exists a $\xi_{j}^{\nu}$ so that $|\xi-\xi_{j}^{\nu}|<2^{-j/2}$. 
\end{enumerate}
Note that there are at most $c\;2^{j}$ elements in the collection $\{\xi_{j}^{\nu}\}_{\nu}$.  Furthermore
we need an elementary fact for later use. For for any fixed $x$ and $j\in\mathbb{N}$, 
recalling~\eqref{D:Gamma-n}, we define 
\[
{\rm card}_{z}={\rm cardinality\; of}\;\{\xi_{j}^{\nu}\in \Lambda_{z}\}.
\] 
By setup of $\xi_{j}^{\nu}$ on $S^{2}$, we have ${\rm card}_{z}\leq c2^{j}$ and
\begin{equation}\label{card-xi}
{\rm card}_{z+1} \leq \frac{1}{4} \cdot
{\rm card}_{z}, \;{\rm if}\; z\geq 0,\;{\rm and}\;
 {\rm card}_{z-1} \leq \frac{1}{4} \cdot
{\rm card}_{z}, \;\;{\rm if}\; z< 0.
\end{equation}

Then we define the 
set  $\widetilde{R}_{z\pm,j}^{\nu}=\widetilde{R}_{z\pm,j}^{\nu}(D)$ in the $y-$space by   
\begin{equation}
\widetilde{R}_{z\pm,j}^{\nu}=\{ y: |y-\overline{y}|\leq \overline{c} 2^{-j/2},\;|\pi_{j}^{\nu}(y-\overline{y})|\leq \overline{c} 2^{-j}\}
\end{equation}
where $\pi_{j}^{\nu}$ is the orthonormal projection in the direction $\xi_{j}^{\nu}$ and $\overline{c}$ is a large constant
(independent of $j$) depending only on size of $D$, i.e., $\delta$. Note that in fact $\widetilde{R}_{z\pm,j}^{\nu}$ 
are independent of $z$. 
The mapping 
\begin{equation}\label{change-variable-H}
x\rightarrow y=(\psi_{\pm})_{\xi}(x,\xi):=\nabla_{\xi}\;\psi_{\pm}(x,\xi)
\end{equation}
has, for each $\xi$, a non-vanishing Jacobian denoted by $J_{\xi}(x)$, i.e., 
\begin{equation}\label{Jacobian} 
\begin{split}
J_{\xi}(x)=& |\det\partial_x\partial_{\xi} {\Big [} \psi_{\pm}(x,\xi)){\Big ]}|= 
|\det  {\Big [} \partial^2\psi_{\pm}(x,\xi)/\partial_{x_i}\partial_{\xi_j}  {\Big ]} |\\
&=|\det\frac{1}{2}{\Big [}I\mp\frac{x}{|x|}\otimes\frac{\xi}{|\xi|}{\Big ]}|
= (\frac{1}{2})^3|1\mp\cos\theta_0|.
\end{split}
\end{equation} 
For each fixed $\xi$, by~\eqref{supp-p-z} and~\eqref{D:Gamma-n}, we have   
\begin{equation}\label{Jacobian-ratio}
\begin{split}
& \inf_{x\in {\rm supp}\; p_{(z+1)\pm}} J_{\xi}(x) \geq \frac{1}{4}\cdot \inf_{x\in {\rm supp}\; p_{z\pm}} J_{\xi}(x)\;{\rm if\;}z\geq 0 \\
&  \inf_{x\in {\rm supp}\; p_{(z-1)\pm}} J_{\xi}(x) \geq \frac{1}{4}\cdot \inf_{x\in {\rm supp}\; p_{z\pm}} J_{\xi}(x)\;{\rm if\;}z< 0
\end{split}
\end{equation}

Since $j\in\mathbb{N}$, $2^{-j}\leq 1$, we take $R_{z\pm,j}^{\nu}$ to be the inverse under 
$(\psi_{\pm})_{\xi}$, with $\xi=\xi_{j}^{\nu}$, of set $\widetilde{R}_{z\pm,j}^{\nu}$:
\begin{equation}
\begin{split}
R_{z\pm,j}^{\nu}& =\{x:  |\overline{y}-(\psi_{\pm})_{\xi}(x,\xi_{j}^{\nu})|\leq \overline{c} 2^{-j/2},\; 
\\& |\pi_{j}^{\nu}(\overline{y}-(\psi_{\pm})_{\xi}(x,\xi_{j}^{\nu}))|\leq \overline{c}\cdot 2^{-j},
       {\rm and}\;  (x,\xi_{j}^{\nu})\in  {\rm supp}\;{\rho}_{1}(x)\cdot p_{z\pm}( x, \xi) \}.
\end{split}
\end{equation}

Let $D^{*}_{z\pm}=\cup_{2^{-j}\leq\delta}\cup_{\nu} R_{z\pm,j}^{\nu}$; then 
\begin{equation}\label{singular-size}
\begin{split}
|D^{*}_{z\pm}| &\leq \sum_{2^{-j}\leq\delta}\sum_{\nu} |R_{z\pm,j}^{\nu}|\leq c_{z\pm} 
\sum_{2^{-j}\leq\delta}\sum_{\nu} |\widetilde{R}_{z\pm,j}^{\nu}| \\
&\leq c_{z\pm}\sum_{2^{-j}\leq\delta} 2^{-j(3+1)/2}\cdot 2^{j}=c_{z\pm}\sum_{2^{-j}\leq\delta} 2^{-j}\leq c_{z\pm}\delta
\end{split}
\end{equation}
where $c_{z\pm}$ depends on the value of Jacobian ~\eqref{Jacobian}  on the support of $p_{z\pm}$. 
The relation~\eqref{Jacobian-ratio}  means the longest vector inside $R_{z+1,j}^{\nu}$ comparing to that of  
$R_{z,j}^{\nu}$  if $z\geq 0$( $R_{(z-1),j}^{\nu}$ vs $R_{z,j}^{\nu}$ if $z<0$) is no more than 
$4$ times  when a fixed $\overline{c}$  is given.  
One the other hand, the angle spanned by $x$ and $\xi$ in the support of $p_{(z+1)\pm}$ is $1/2$ of 
that of $p_{z\pm}$ if  $z\geq 0$ ($p_{(z-1)\pm}$ vs $p_{z\pm}$ if $z<0$).
Combing these two facts, we see that 
\[
|D_{(z+1)\pm}^{*}|\leq 2|D_{z\pm}^{*}|\;\;{\rm if \;}z\geq 0\;{\rm and}\;|D_{(z-1)\pm}^{*}|\leq 2|D_{z\pm}^{*}|\;\;{\rm if \;}z< 0,
\]
that is $c_{z\pm}$ of~\eqref{singular-size} satisfy  
\begin{equation}\label{singular-size-ratio}
c_{(z+1)\pm} \leq 2\cdot c_{z\pm} \;\;{\rm if \;}z\geq 0\;{\rm and}\;\;c_{(z-1)\pm} \leq 2\cdot c_{z\pm} \;\;{\rm if \;}z< 0.
\end{equation}

By H\"{o}lder inequality,~\eqref{singular-size} and Proposition~\ref{embedding}
\begin{equation}\label{singular-L-p}
\begin{split}
\int_{D^{*}_{z\pm}} {\big |}\rho_{1}\cdot T^{(-1)}_{z\pm,I}\; u(x) {\big |}  dx & \leq \|\rho_{1}\cdot T^{(-1)}_{z\pm,I}\; u(x)\|_{L^{2}} 
\cdot |D_{z\pm}^{*}|^{1/2}  \\
&\leq c'_{z\pm}\|u\|_{L^{\frac{6}{5}}}\cdot (c_{z\pm})^{1/2}\cdot \delta^{1/2}
\end{split}
\end{equation}
 since $T^{(-1)}_{z\pm,I}$ has symbol of order $-1$. 
Also $\|u\|_{L^{\frac{6}{5}}}\leq |D|^{-1+\frac{5}{6}}$, since the atom $u(x)\leq |D|^{-1}$ and is supported in $D$. 
Thus we get 
\[
\int_{D^{*}_{z\pm}} {\big |}\rho_{1}\cdot T^{(-1)}_{z\pm,I}\; u(x) {\big |}  dx  \leq  c'_{z\pm}\cdot \delta^{3(-\frac{1}{6})} \cdot
(c_{z\pm})^{1/2}\cdot \delta^{1/2}=c'_{z\pm}\cdot (c_{z\pm})^{1/2}:=C_{z\pm}.
\]

From Proposition~\ref{embedding-ratio}, we have 
\begin{equation}\label{embedding-ratio-2}
c'_{(z+1)\pm} \leq \frac{1}{2}\cdot c'_{z\pm} \;\;{\rm if \;}z\geq 0\;{\rm and}\;\;c'_{(z-1)\pm} 
\leq \frac{1}{2}\cdot c'_{z\pm} \;\;{\rm if \;}z< 0.
\end{equation} 
 Combining~\eqref{singular-size-ratio} and~\eqref{embedding-ratio-2}, we have 
\begin{equation}\label{H-1-ratio-2}
C_{(z+1)\pm}\leq (\frac{1}{2})^{1/2}\cdot C_{z\pm}\; {\rm if}\; z\geq 0 \;{\rm and}\; C_{(z-1)\pm}\leq(\frac{1}{2})^{1/2} 
\cdot C_{z\pm}\;{\rm if}\; z\leq 0.
\end{equation}
Thus the  integrals in~\eqref{L-1-two-parts-1} satisfy Proposition~\ref{L-1-2}.

Next we estimate the integrals in~\eqref{L-1-two-parts-2}. 

We remind again that each integral is bounded by a constant is already proved in~\cite{SSS91},
please see Chapter IX of~\cite{Stein93}. 
Thus we only duplicate the necessary part of the proof which is sufficient for us to observe how $C_{z\pm}$ vary.

Let $\Gamma_j^{\nu}$ denote the cone in the $\xi$-space whose central direction is $\xi_j^{\nu}$, that is,
\begin{equation}
\Gamma_{j}^{\nu}=\{\xi: |\frac{\xi}{|\xi|}-\xi_j^{\nu}|\leq 2\cdot 2^{-\frac{j}{2}} \}.
\end{equation}
And there is an associated partition of unity
\begin{equation}\label{refined-LP}
\sum_{\nu} \chi_j^{\nu}(\xi)=1,\;{\rm for\;all\;}\xi\neq 0\; {\rm and\;all}\; j\in\mathbb{N}.
\end{equation}
where each  $\chi_j^{\nu}$ is of homogeneous of degree $0$ in $\xi$, supported in $\Gamma_j^{\nu}$
and 
\begin{equation}
|\partial_{\xi}^{\alpha} \chi_j^{\nu}(\xi)| \leq C_{\alpha} 2^{|\alpha|j/2}|\xi|^{-\alpha}.
\end{equation}

We recall the dyadic decomposition~\eqref{D:chi-k-sum} and define $\widetilde{\rho}_{0}$ by
\begin{equation}
1=\sum_{j\in \mathbb{Z}} \rho_{j}(\xi):=\widetilde{\rho}_{0}(\xi)+\sum_{j\in \mathbb{N}} \rho_{j}(\xi).
\end{equation}
Using~\eqref{refined-LP}, we have the refined Littlewood-Paley decomposition:
\begin{equation}
1=\widetilde{\rho}_{0}(\xi)+\sum_{j=1}^{\infty}\sum_{\nu}\chi_{j}^{\nu}(\xi){\rho}_{j}(\xi).
\end{equation}  
 For $j\in\mathbb{N}$,  we define
\[
\begin{split}
{\big (}T^{(-1)}_{z\pm,I}{\big)}^{\nu}_{j} u(x)& =\int_{\mathbb{R}^{3}} e^{ i\psi_{\pm}(x,\xi)}\;a_{z\pm,j}^{\nu}(x,\xi)\;\widehat{u}(\xi) d\xi, \\
&=\int_{\mathbb{R}^{3}} K_{z\pm,j}^{\nu}(x,y) u(y) dy 
\end{split}
\]
where
\begin{equation}
a_{z\pm,j}^{\nu}(x,\xi)=\chi_{j}^{\nu}(\xi)\cdot\rho_{j}(\xi)\cdot \rho_{1}(x) \; p^{-1}_{z\pm}(x,\xi),
\end{equation}
and the kernel of ${\big (}T^{(-1)}_{z\pm ,I}{\big)}^{\nu}_{j}$ is given by 
\begin{equation}\label{second-kernel}
K_{z\pm,j}^{\nu}(x,y)=\int_{\mathbb{R}^{3}} e^{i[\psi_{\pm}(x,\xi)-y\cdot\xi]}\; a_{z\pm,j}^{\nu}(x,\xi) d\xi.
\end{equation}
We also define 
\begin{equation}\label{sum-nu}
{\big (}T^{(-1)}_{z\pm,I}{\big)}_{j}=\sum_{\nu} {\big (}T^{(-1)}_{z\pm,I}{\big)}^{\nu}_{j}
\end{equation}
which having symbols $a_{z\pm,j}(x,\xi)=\sum_{\nu} a_{z\pm,j}^{\nu}$ and kernel $K_{z\pm,j}(x,\xi)$.
 
To estimate~\eqref{second-kernel}, we  choose axes 
in the $\xi=(\xi_1,\xi_2,\xi_3)$-space so that $\xi_1$ is in the direction of $\xi_j^{\nu}$ and 
$\xi'=(\xi_2,\xi_3)$ is perpendicular to $\xi_j^{\nu}$. Write the phase functions as
\[
\begin{split}
\psi_{\pm}(x,\xi)-y\cdot\xi &=[(\psi_{\pm})_{\xi}(x,\xi_j^{\nu})-y]\cdot \xi+ 
[\psi_{\pm}(x,\xi)-(\psi_{\pm})_{\xi}(x,\xi_j^{\nu})\cdot\xi ] \\
&:=[(\psi_{\pm})_{\xi}(x,\xi_j^{\nu})-y]\cdot \xi+h_{\pm}(\xi)
\end{split}
\]
Rewrite~\eqref{second-kernel} as 
\begin{equation}
K_{z\pm,j}^{\nu}(x,y)=\int_{\mathbb{R}^{3}} e^{i[\psi_{\pm}(x,\xi_j^{\nu})-y\cdot\xi]}\; [e^{ih_{\pm}(\xi)}\;a_{z\pm,j}^{\nu}(x,\xi)] d\xi,
\end{equation}  
and introduce the operator $L$ defined by 
\[
L=I-2^{2j}\frac{\partial^2}{\partial\xi_1^2}-2^{j}\nabla_{\xi’}.
\] 
Then we have (see the details in~\cite{Stein93})    
\[
|L^N ( [e^{ih_{\pm}(\xi)}\;a_{n,j}^{\nu}(x,\xi)] )|\leq C_N \cdot 2^{-j},
\] 
and 
\[
\begin{split}
&L^N e^{i[\psi_{\pm}(x,\xi_j^{\nu})-y\cdot\xi]} \\
&=\{1+2^{2j}|((\psi_{\pm})_{\xi}(x,\xi_j^{\nu})-y)_1|^2 +2^{j}|((\psi_{\pm})_{\xi}(x,\xi_j^{\nu})-y)'|^2 \}^N 
\cdot e^{i[\psi_{\pm}(x,\xi_j^{\nu})-y\cdot\xi]}.
\end{split}
\] 
Using integration by parts and above, we obtain 
\begin{equation}\label{major-2}
|K_{z\pm,j}^{\nu}(x,y)|\leq c2^j\{1+2^{j}|((\psi_{\pm})_{\xi}(x,\xi_j^{\nu})-y)_1| +2^{j/2}|((\psi_{\pm})_{\xi}(x,\xi_j^{\nu})-y)'| \}^{-2N}
\end{equation} 
We use~\eqref{major-2} and change of variables 
\begin{equation}\label{change-variable-H2}
x\rightarrow y=(\psi_{\pm})_{\xi}(x,\xi)
\end{equation}
whose non-vanishing Jacobain~\eqref{Jacobian}, $J_{\xi}(x)$, satisfying~\eqref{Jacobian-ratio}.   
Therefore 
\begin{equation}\label{K-n-j-nu}
\begin{split}
\int |K_{z\pm,j}^{\nu}(x,y)| dx & \leq c_{z\pm} 2^{j}\int (1+|2^{j} (x-y)_1| +|2^{j/2}(x-y)'|)^{-2N} dx \\
&\leq c_{z\pm} 2^{-j}
\end{split}
\end{equation}
 if we choose $2N>3$. Here we abuse the notations $c_{\pm}$,  since for the same reason mentioned early 
 for ~\eqref{singular-size-ratio}, $c_{z\pm}$ satisfy the relation   
 \begin{equation}\label{singular-size-ratio-1}
c_{(z+1)\pm} \leq 2\cdot c_{z\pm} \;\;{\rm if \;}z\geq 0\;{\rm and}\;\;c_{(z-1)\pm} \leq 2\cdot c_{z\pm} \;\;{\rm if \;}z< 0.
\end{equation}

Recalling supp\;$p_{z\pm}(x,\xi)\subset \Lambda_{z}$,  ${\rm card}_{z}\leq c2^{j}$ and~\eqref{card-xi},  
we conclude that   
 \begin{equation}\label{kernel-j}
\int_{\mathbb{R}^{3}} |K_{z\pm,j}(x,y)| dx\leq C_{z\pm},\;{\rm all\;}y\in\mathbb{R}^{3}.
\end{equation}
From~\eqref{card-xi} and~\eqref{singular-size-ratio-1}, we have   
\begin{equation}\label{H-1-ratio-3}
C_{(z+1)\pm}\leq (\frac{1}{2}) \cdot C_{z\pm}\; {\rm if}\; z\geq 0 \;{\rm and}\; C_{(z-1)\pm}\leq(\frac{1}{2})
\cdot C_{z\pm}\;{\rm if}\; z\leq 0.
\end{equation}

Another two estimates about $K_{z\pm,j}$ which are the following. 
\begin{equation}\label{kernel-differ}
\int_{\mathbb{R}^{3}} |K_{z\pm,j}(x,y)-K_{z\pm,j}(x,y')| dx\leq C_{z\pm}\cdot |y-y'|\cdot 2^j,\;{\rm all\;}y, y'\in\mathbb{R}^{3},
\end{equation} 
\begin{equation}\label{kernel-complement}
\int_{\prescript{c}{}D_{z\pm}^*} |K_{z\pm,j}(x,y)| dx\leq \frac{C_{z\pm}\cdot 2^{-j}}{\delta},\;{\rm if\;}y\in B,
{\rm and}\; 2^{-j}\leq\delta,
\end{equation} 
where the bound $C_{z\pm}$ is independent of $j, y, y'$ and $\delta$. (Recall that $\delta$ is the radius of ball $D$) 
 They are proved by the similar argument, we omit the details as one can find them in~\cite{Stein93}.
We abuse the notation $C_{z\pm}$ in above two estimates, as the relation~\eqref{H-1-ratio-3} holds for the same reason.       

Now we are ready to estimate the series of the integrals in~\eqref{L-1-two-parts-2}.  Recall~\eqref{sum-nu}, we write 
\begin{equation}\label{two-sum}
\rho_{1}\cdot T^{(-1)}_{z\pm,I}\; u(x)=\sum_{j=0}^{\infty} {\big (}  T^{(-1)}_{z\pm,I}{\big)}_{j}
=\sum_{2^{j}\leq \delta^{-1}} +\sum_{2^{j}>\delta^{-1}}.
\end{equation}
For the second sum, ~\eqref{kernel-complement} yields
\[
\sum_{2^{j}>\delta^{-1}} \int_{\prescript{c}{}D_{z\pm}^*} |{\big (}  T^{(-1)}_{z\pm,I}{\big)}_{j} u(x)| dx
\leq C_{z\pm} {\Big (}\int |u(y)|dy {\Big )}\cdot {\Big (} \sum_{2^{j}\leq \delta^{-1}} 2^{-j} {\Big )}\cdot \delta^{-1}\leq C_{z\pm}
\] 
where $C_{z\pm}$ form a convergent series by~\eqref{H-1-ratio-3}. 

For the first sum of~\eqref{two-sum}, using property of atom $\int u(y)dy=0$, we rewrite  
\[
{\big (}  T^{(-1)}_{z\pm,I}{\big)}_{j} u(x)=\int K_{z\pm,j}(x,y) u(y)dy=\int_{D} [K_{z\pm,j}(x,y)-K_{z\pm,j}(x,\overline{y})] u(y)dy
\]
where $\overline{y}$ is the center of $D$.
Using~\eqref{kernel-differ}, we have 
\[
\int |{\big (}  T^{(-1)}_{z\pm,I}{\big)}_{j} u(x)| dx\leq C_{z\pm}\cdot 2^{j} \|u\|_{L^{1}} \cdot\delta
\]
and 
\[
\sum_{2^{j}\leq \delta^{-1}} \int |{\big (}  T^{(-1)}_{z\pm,I}{\big)}_{j} u(x)| dx\leq C_{z\pm}  
\sum_{2^{j}\leq \delta^{-1}} 2^{j}\cdot\delta \leq C_{z\pm}.
\] 
where $C_{z\pm}$ form a convergent series. Thus we conclude Proposition~\ref{L-1-2} and thus 
 the boundedness of $T_{A}$ by combing all the  results. 
\end{proof}

\begin{rem}\label{shrink-II-2}
Similarly to the remark~\ref{shrink-II},  the enlargement  of region I  does not affect  the $L^{1}$ estimates
we did in this subsection. This enlargement will effects two places. First the ratio ${1}/{2}$ in~\eqref{L-1-C-1} 
will be  replaced with $({1}/{2})^{3/5}$ as it comes from the $L^{2}$ estimates.  
Secondly the ratio $1/2$  in~\eqref{H-1-ratio-2} will be  replaced with $({1}/{2})^{-\frac{3}{5}+\frac{1}{2}}=(1/2)^{-1/10}$
 since $c'_{z\pm}$ in~\eqref{singular-L-p} also come  from the $L^{2}$ estimates in the proof of Proposition~\ref{embedding}. 
 Thus we see the estimates in this subsection still hold on  the enlarged  region I.
\end{rem}

\section{$L^p$ boundedness of $T_B$ and proof of Lemma~\ref{L:R-v-v_*-ineq}}\label{section-5}

The goal of this section is proving  the following. 

\begin{thm}\label{P-T-B}
Let $T_{B}=\sum_{z\in\mathbb{Z}}\; T_{z,II}$ where 
\[
T_{z,II} h(x) =\int_{\mathbb{R}^3} e^{ix\cdot\xi} \;\psi_{z,II}(x,\xi)\cdot a(x,\xi)
 \;\widehat{h}(\xi) d\xi,
\]
\[
{\rm supp}\; \psi_{z,II}(x,\xi)\subseteq\{(x,\xi)| (x,\xi)\in\Lambda_z \;{\rm and}\;  
|x||\xi|\leq 512\cdot 2^{2|z|}\},
\]
and
\[
a(x,\xi)=(|x||\xi|)^{\gamma}\int_{S^{2}_{+}} e^{i(x,\omega)(\xi,\omega)} b(\cos\theta)d\omega.
\]
Suppose $p$ satisfies 
\[
1<p<\infty,
\]
then we have
\begingroup
\large
\begin{equation}\label{E:R-ineq-B}  
\| T_{B} \;h\|_{L^p}\leq C \|h\|_{L^{p}}.
\end{equation}
\endgroup
\end{thm}
\begin{proof}

We should first show that $T_{0,II}$ is $L^p$ bounded. Then the $L^p$ boundedness of all 
$T_{z,II}$ hold by the same reasons. To show those $L^p$ bounds form a convergent series, 
we need to observe how the bounds varies with $z$. 
As we point out before that on the region II of phase space $(x,\xi)$, the stationary phase 
formula does not apply for the calculation of $a(x,\xi)$.  On the other hand, the 
quantity~\eqref{a-region-2} below still implicitly contains a phase function which could have 
large oscillation. This means that in the estimate 
of the kernel of the operator $T_B$, we must avoid using integration by parts argument since 
the derivative of $a(x,\xi)$ is not controllable.      

Recalling~\eqref{a-region-2-pre}, we have 
\begin{equation}\label{a-region-2}
a(x,\xi)=  C\cdot\sin\theta_0\cos\theta_0\cdot (|x||\xi|)^{\gamma-1},\;{\rm if}\;(x,\xi)\in\;{\rm region\;II},
\end{equation}
where $C$ is independent of $x,\xi$ and $\theta_{0}=\arccos(x\cdot\xi/|x||\xi|)$ is the angle spanned 
by $x$ and $\xi$. 


We expand $\widehat{h}(\xi)$ and rewrite 
\begin{equation}\label{T-K}
T_{z,II} h(x)=\int_{\mathbb{R}^{3}} K_z(x,y) h(y)dy
\end{equation}
where 
\[
K_z(x,y)=\int_{\mathbb{R}^{3}}\psi_{z,II}(x,\xi)\; e^{i(x-y)\cdot\xi}\; a(x,\xi) d\xi.
\]
We use generalized Schur's lemma (see Lemma~\ref{Schur}) to show the $L^{p}$ boundedness of 
the $T_0$ operator by replacing $K_0(x,y)$ in~\eqref{T-K} with $|K_0(x,y)|$. Thus we need to 
find two positive measurable functions $h_{1}$ and $h_{2}$ such that the followings hold:
\begin{equation}\label{Schur-desire}
\begin{split}
& \int |K_0(x,y)|^{\delta_{1}p'} h_{1}(y)^{p'}dy \leq C_{p'} h_{2}(x)^{p'},\\
&\int |K_0(x,y)|^{\delta_{2}p } h_{2}(x)^{p}dx \leq C_{p} h_{1}(y)^{p}
\end{split}
\end{equation}
where $1<p,\;p' <\infty$, $1/p+1/p'=1$, $0<\delta_{1}, \delta_{2}<1$ and $\delta_1+\delta_{2}=1$.
Then we will show that the bounds of $T_{z,II}$ can be obtained by a scaling argument and 
verify that the bounds of $T_{z,II}$ form a convergent series of $z\in\mathbb{Z}$.

Next we estimate $|K_{0}(x,y)|$ according to~\eqref{a-region-2} and the sizes of $|x|$ and $|y|$.

\noindent{\bf Case 1.} If $|y|\leq 2|x|$, using~\eqref{a-region-2},
\begin{equation}\label{x-l-y-s}
\begin{split}
|K_0(x,y)| & \leq c'_0\; |x|^{\gamma-1}\int_{\{(x,\xi)\in\Lambda_{0},\;|x||\xi|<512 \} } |\xi|^{\gamma-1} d\xi  \\
&= c''_0\; |x|^{-1+\gamma} \int_{0}^{512/|x|} r^{-1+\gamma} \cdot r^2 dr 
\leq C_0 \;(512)^{2+\gamma} |x|^{-3}. 
\end{split}
\end{equation}

Before we go to the next case, let us discuss how the coefficients vary with $z$. 
The coefficient  $c'_0$ represents the contribution from $\cos\theta_0 \sin\theta_0$, inside the definition of $a(x,\xi)$, 
over ${\rm supp}\;\psi_{0,II}$. We note that $c''_0$ inherits  $c'_0$  and furthermore represents 
the use of spherically coordinate on 
${\rm supp}\;\psi_{0,II}$, i.e., $\Lambda_{0}$.  Suppose the similar calculation for  $K_z(x,y)$ produces 
coefficients $ c'_{|z|}, \;c''_{|z|}$ and $ C_{|z|}$.
From definition of $\Lambda_{z}$ and above explanation
\[
  c'_{z}=c'_{|z|},\;c'_{|z|+1}=\frac{1}{2}\cdot c'_{|z|} \;{\rm and}\; c''_{|z|+1}=\frac{1}{2}\cdot \frac{1}{4}\cdot c''_{|z|}.
\]
Thus 
\begin{equation}\label{C-ratio}
C_{z}=C_{|z|},\;C_{|z|+1}= (\frac{1}{2})^3\cdot C_{|z|}.
\end{equation}

\noindent{\bf Case 2.} If $|y|> 2|x|$, let $\lambda=|y|/|x|$ and write $\xi=\lambda\eta$.  Then $|x||\xi|=|y||\eta|<512$ and  
$\frac{|x-y|}{\lambda} |\eta| <\frac{3}{2}\cdot 512$. Thus we have, 
\begin{equation}\label{x-s-y-l}
\begin{split}
|K_{0}(x,y)| & \leq c'_0\; {\Big |}\int_{\{(x,\xi)\in\Lambda_{0},\; |x||\xi|<512\}} e^{i\lr{x-y,\xi}} (|x||\xi|)^{\gamma-1}d\xi {\Big |}\\
& = c'_0\; {\Big |}\int_{\{(y,\eta)\in\Lambda_{0},\;|y||\eta|<512\}} 
e^{i\lr{(x-y)/\lambda,\eta}} (|y||\eta|)^{\gamma-1} \lambda^{-3} d\eta {\Big |}\\
&\leq c''_0 \cdot\lambda^{-3} \; |y|^{-1+\gamma} \int_{0}^{512/|y|} r^{-1+\gamma} \cdot r^2 dr \\
& \leq C_0 \;(512)^{2+\gamma} |y|^{-3}{\big (} \frac{|x|}{|y|} {\big )}^{3}.
\end{split}
\end{equation}
Again~\eqref{C-ratio} is satisfied as before if we follow the same
method to estimate $K_z(x,y)$.

In order to prove~\eqref{Schur-desire}, we choose a set of parameters according to given $p$ and its conjugate $p'$.  
We take
\begin{equation}\label{parameters}
\begin{split}
& \alpha=\frac{28}{10}\cdot\frac{1}{p'}\;,\;\beta=\frac{28}{10}\cdot\frac{1}{p}\;,\; \ell_1=\frac{1}{2}\;,\;
\ell_2=1+\frac{p}{2p'} \\
&\delta_1=\frac{28}{30}\cdot\frac{1}{p}+\frac{1}{30}\cdot\frac{1}{p'}\;,\;
\delta_2=\frac{29}{30}\cdot\frac{1}{p'}+\frac{2}{30}\cdot\frac{1}{p},
\end{split}
\end{equation}
where $\ell_{2}>1,\;0<\delta_{1},\;\delta_{2}<1$ and $\delta_{1}+\delta_{2}=1$.
They are chosen so that the following inequalities we need later are satisfied: 
\begin{equation}\label{ineqs}
\begin{split}
& -\alpha p' +3>0\;,\; -6\delta_1p' -\alpha p' +3<0\;,\;-3\delta_{1} p'+\ell_1(-\alpha p' +3)=-\beta p' \\
& -\beta p+3>0\;,\; -3\delta_2 p-\beta p+3<0\;,\; -3\delta_{2} p +\ell_2 ( -\beta p +3)=-\alpha p. 
\end{split}
\end{equation}

With these in hand and taking $h_{1}(y)=|y|^{-\alpha}$,\; $h_{2}(x)=|x|^{-\beta}$, 
we are ready to prove~\eqref{Schur-desire}. 

Let $x$ be fixed, we claim that 
\begin{equation}\label{K-y-tail}
 \int_{|y|>|x|^{\ell_{1}}}  |K_0(x,y)|^{\delta_{1} p'} h_{1}(y)^{ p'} dy
 \leq C_{0}(512)^{2+\gamma} |x|^{-3\delta_{1} p'+\ell_{1}(-\alpha  p' +3)}.
\end{equation} 
Here the notation $C_{0}(512)^{2+\gamma}$ is abused since it is different with that appears in~\eqref{x-l-y-s} and 
~\eqref{x-s-y-l}, but we keep using the same notation as it inherits all the properties in~\eqref{C-ratio}.
 Therefore it is also used in this sense till the end of this section.     

 Use the spherically coordinate with $r=|y|$. For $|y|$ large enough, estimate~\eqref{x-s-y-l} gives
\[
 |K_0(x,y)|^{\delta_{1} p'} h_{1}(y)^{ p'} r^{2} dr \approx |x|^{3\delta_{1} p'} 
 r^{-6\delta_{1}p'} r^{-\alpha p'+2} dr.  
\]
This implies the integral in~\eqref{K-y-tail} converges since $-6\delta_{1}p' -\alpha p'+3<0$ by~\eqref{ineqs}.
Thus the value of the left hand side 
of ~\eqref{K-y-tail} is determined by the antiderivative of the integrand at $r=|x|^{\ell_{1}}$. 
If $|x|^{\ell_{1}}\leq 2 |x|$ holds, by~\eqref{x-l-y-s},  
\begin{equation}\label{y-small}
\begin{split}
 \int_{|y|>|x|^{\ell_{1}}}  |K_0(x,y)|^{\delta_{1} p'} h_{1}(y)^{ p'} dy & \leq C_{0}(512)^{2+\gamma} |x|^{-3 \delta_{1} p'}
 \cdot r^{-\alpha  p' +3}{\Big |}_{r=|x|^{\ell_{1}}} \\
& =C_{0}(512)^{2+\gamma} |x|^{-3\delta_{1} p'+\ell_{1}(-\alpha  p' +3)}.
\end{split}
\end{equation}
If $|x|^{\ell_{1}}>2 |x|$ holds, by~\eqref{x-s-y-l},
 \begin{equation}\label{y-large}
\begin{split}
\int_{|y|>|x|^{\ell_{1}}}  |K_0(x,y)|^{\delta_{1} p'} h_{1}(y)^{ p'} dy &\leq C_{0}(512)^{2+\gamma} 
|x|^{3 \delta_{1} p'}\cdot ( r^{-6 \delta_{1} p'} \cdot r^{-\alpha  p' +3}){\Big |}_{r=|x|^{\ell_{1}}} \\
& \leq C_{0}(512)^{2+\gamma} |x|^{-3\delta_{1} p'+\ell_{1}(-\alpha  p' +3)},
\end{split}
\end{equation}
thus we conclude~\eqref{K-y-tail}. Similarly for fixed $x$, we have 
\begin{equation}\label{K-y-head}
 \int_{|y|\leq |x|^{\ell_{1}}}  |K_0(x,y)|^{\delta_{1} p'} h_{1}(y)^{ p'} dy
 \leq C_{0}(512)^{2+\gamma}|x|^{-3\delta_{1} p'+\ell_{1}(-\alpha  p' +3)},
\end{equation}  
because the integral above converges (at $r=0$) as $-\alpha_{1}  p' +3>0$ by~\eqref{ineqs}. 
And the calculations~\eqref{y-small} and~\eqref{y-large} at $r=|x|^{\ell_{1}}$ give~\eqref{K-y-head}.

Thus for fixed $x$ we have, 
\begin{equation}
\begin{split}
& \int_Y |K_0(x,y)|^{\delta_{1} p'} h_{1}(y)^{ p'} dy \\
& \leq C_{0} (512)^{2+\gamma} |x|^{-3\delta_{1} p'+\ell_1(-\alpha p' +3)} \leq C_{0}(512)^{2+\gamma} |x|^{-\beta p' },
\end{split}
\end{equation}
since~\eqref{ineqs} gives
\begin{equation}\label{Schur-x-small}
-3\delta_{1} p'+\ell_1(-\alpha p' +3)=-\beta p' .
\end{equation}
Therefore we conclude the first inequality of~\eqref{Schur-desire}.

Let $y$ be fixed,  we claim that 
\begin{equation}\label{K-x-head}
 \int_{|x|<|y|^{\ell_{2}}}  |K_0(x,y)|^{\delta_{2} p} h_{2}(x)^{ p} dx
 \leq C_{0}(512)^{2+\gamma} |y|^{-3\delta_{2} p+\ell_{2}(-\beta  p +3)}.
\end{equation} 
We use the spherically coordinate $r=|x|$. For $r$ is close to $0$,  estimate~\eqref{x-s-y-l} gives
\[
|K_0(x,y)|^{\delta_{2} p} h_{2}(x)^{ p}\approx |y|^{-6} r^{3\delta_{2}p} r^{-\beta p+2},
\] 
and which means the integral converges at $r=0$ since $3\delta_{2}p>0$ and $-\beta_{1}  p +3>0$.
Therefore the value of the the integral in~\eqref{K-x-head} is determined by the antiderivative of 
the integrand at $r=|y|^{\ell_{2}}$.
If $|y|^{\ell_{2}}\geq |y|/2$, 
\begin{equation}\label{x-small}
\begin{split}
\int_{|x|<|y|^{\ell_{2}}}  |K_0(x,y)|^{\delta_{2} p} h_{2}(x)^{ p} dx
& \leq C_{0}(512)^{2+\gamma} r^{-3\delta_{2}p} r^{-\beta p+3}{\Big |}_{r=|y|^{\ell_{2}}} \\
&\leq C_{0}(512)^{2+\gamma} |y|^{-3\delta_{2} p+\ell_{2}(-\beta  p +3)}.
\end{split}
\end{equation}
If $|y|^{\ell_{2}}\leq |y|/2$,
\begin{equation}\label{x-large}
\begin{split}
\int_{|x|<|y|^{\ell_{2}}}  |K_0(x,y)|^{\delta_{2} p} h_{2}(x)^{ p} dx
& \leq C_{0}(512)^{2+\gamma} |y|^{-6\delta_{2}p} r^{3\delta_{2}p} r^{-\beta p+3}{\Big |}_{r=|y|^{\ell_{2}}} \\
&\leq C_{0}(512)^{2+\gamma} |y|^{-3\delta_{2} p+\ell_{2}(-\beta  p +3)}.
\end{split}
\end{equation}
Similarly for fixed $y$, we have
\begin{equation}\label{K-x-tail}
\int_{|x|>|y|^{\ell_{2}}}  |K_0(x,y)|^{\delta_{2} p} h_{2}(x)^{ p} dx
 \leq C_{0}(512)^{2+\gamma} |y|^{-3\delta_{2} p+\ell_{2}(-\beta  p +3)},
\end{equation}
since the integral above converges (at $r=\infty$) by $-3\delta_{2}p  {-\beta p}+3<0$ of 
~\eqref{ineqs}. And the calculations~\eqref{x-small} and~\eqref{x-large} at $r=|y|^{\ell_{2}}$ 
gives~\eqref{K-x-tail}. 

Thus for fixed $y$ we have, 
\begin{equation}
\begin{split}
& \int_X |K_0(x,y)|^{\delta_{1} p'} h_{1}(y)^{ p'} dx \\
& \leq C_{0}(512)^{2+\gamma}  |y|^{-3\delta_{2} p+\ell_2(-\beta p +3)} \leq C_{0}(512)^{2+\gamma} |y|^{-\alpha p },
\end{split}
\end{equation}
since~\eqref{ineqs} gives
\begin{equation}\label{Schur-x-small-2}
-3\delta_{2} p+\ell_2(-\beta p +3)=-\alpha p.
\end{equation}
Therefore we conclude the second inequality of~\eqref{Schur-desire} and $T_{0,II}$ is $L^{p}$
bounded with bound $C_{0}(512)^{2+\gamma}$ .

Since the $L^p$ boundedness of $T_{z,II}$ can be obtained by the same argument, we  
should compute the the ratios of bound of $T_{z,II}$ to $C_0\cdot(512)^{2+\gamma}$. Recall that 
\[
T_{1,II} h(x)=\int_{\R^3} K_1(x,y) h(y)dy,\; 
\]
where 
\[
K_1(x,y)=\int_{\R^3} \psi_{1,II}(x,\xi)\;  e^{i(x-y)\cdot\xi} \;a(x,\xi)  d{\xi}
\]
and
\begin{equation}\label{sup-1-II}
{\rm supp}\;\psi_{1,II}\subset\{(x,\xi)| (x,\xi)\in\Lambda_1 \;\;{\rm and}\;\; 0\leq |x||\xi|\leq  512\cdot 2^2  \}.
\end{equation}

Let $x=2\widetilde{x}, y=2\widetilde{y}$ and $\xi=2\widetilde{\xi}$. 
Define 
\[
\begin{split}
& \widetilde{a}(\widetilde{x},\widetilde{\xi}):=a(x,\xi),\;\widetilde{K}_1(\widetilde{x},\widetilde{y}):=K_1(x,y),
\; \widetilde{h}(\widetilde{y}):=h(2\widetilde{y}),\\
& \widetilde{T}_{1,II} \widetilde{h}(\widetilde{x}):=T_{1,II} h(x)
=\int_{\R^3} \widetilde{K}_1(\widetilde{x},\widetilde{y}) \widetilde{h}(\widetilde{y}) 2^3 d\widetilde{y}.
\end{split}
\]
Since 
\[
\|h(y)\|_{L^p}=2^{\frac{3}{p}}\|\widetilde{h} (\widetilde{y})\|_{L^p},\;
\|T_{1,II} h(x)\|_{L^p}=2^{\frac{3}{p}}\|\widetilde{T}_{1,II} \widetilde{h}(\widetilde{x})\|_{L^p},
\]
we have 
\[
\|T_{1,II} h(x)\|_{L^p\rightarrow L^{p}}=\|\widetilde{T}_{1,II} \widetilde{h}(\widetilde{x})\|_{L^p\rightarrow L^{p}}
\]
and we should estimate the later one. 

 By~\eqref{a-region-2}, we have
\begin{equation}\label{a-scale}
\widetilde{a}(\widetilde{x},\widetilde{\xi})=C\cdot\sin\theta_0\cos\theta_0\cdot(|x||\xi|)^{\gamma-1}=C\cdot\sin\theta_0\cos\theta_0\cdot(2^2|\widetilde{x}||\widetilde{\xi}|)^{\gamma-1},
\end{equation}
and
\begin{equation}\label{K-scale}
\begin{split}
&\widetilde{K}_1(\widetilde{x},\widetilde{y})=K_1(2\widetilde{x}, 2\widetilde{y})\\
&=\int_{\R^3} \widetilde{\psi}_{1,II}(\widetilde{x},\widetilde{\xi})
e^{i\cdot 2^2(\widetilde{x}-\widetilde{y})\cdot\widetilde{\xi}} \;\widetilde{a} (\widetilde{x},\widetilde{\xi}) 2^3 d\widetilde{\xi}
\end{split}
\end{equation}
where 
\begin{equation}\label{sup-scale}
{\rm supp}\;\tilde{\psi}_{1,II}(\widetilde{x},\widetilde{\xi})\subset\{(\widetilde{x},\widetilde{\xi})
| (\widetilde{x},\widetilde{\xi})\in\Lambda_1 \;\;{\rm and}\;\; 0\leq |\widetilde{x}||\widetilde{\xi}|\leq  512  \}.
\end{equation}

Now we are ready to compute the ratio of the $L^{p}$ bound of $\widetilde{T}_{1,II}$ to that of  $T_{0,II}$. 
The conditions~\eqref{a-scale} and~\eqref{K-scale} have factor $2^{2(\gamma-1)} \cdot 2^{3}$ in estimate 
of $\widetilde{K}_{1}$ compare to that of $K_{0}$. On the other hand, the condition~\eqref{sup-scale} 
implies that $\widetilde{a} (\widetilde{x},\widetilde{\xi})$ contributes to $\widetilde{K}_{1}$ is $2^{-3}$ times 
compare to $a(x,\xi)$ to $K_{0}$ by~\eqref{C-ratio}.

Therefore we have 
\begin{equation}\label{ratio-2}
\frac{\|{T}_{1,II}\|_{L^{p}\rightarrow L^{p}}}{\|T_{0,II}\|_{L^{p}\rightarrow L^{p}}} 
=\frac{\|\widetilde{T}_{1,II}\|_{L^{p}\rightarrow L^{p}}}{\|T_{0,II}\|_{L^{p}\rightarrow L^{p}}} 
\approx 2^{2(\gamma-1)}. 
\end{equation}
The ratio of $L^{p}$ bound of $T_{|z|+1,II}$  to that of $T_{|z|,II}$ is the order of $2^{2(\gamma-1)}$ for 
any $z\in\mathbb{Z}$ by the same argument. Note that we also have 
$\|{T}_{z,II}\|_{L^{p}\rightarrow L^{p}}=\|{T}_{|z|,II}\|_{L^{p}\rightarrow L^{p}}$.
When $0\leq \gamma<1$, the series $\|{T}_{z,II}\|_{L^{p}\rightarrow L^{p}}$ converges and thus 
$T_{B}$ is $L^{p}$ bounded. When $\gamma=1$, we can adjust the definition of $\psi_{z,I},\;\psi_{z,II},\;z\neq 0$ 
in ~\eqref{a-z-decom} so that $2^{2|z|},\;z\neq 0$ in~\eqref{supp-A-B} is replaced with 
$2^{2(|z|(1-\delta))},\;z\neq 0$ 
where $0<\delta<1/50$. Such a modification of the definition of $T_A$ does not alter the result in 
Proposition~\eqref{P-T-A} as we pointed out in Remark~\ref{shrink-II} and Remark~\ref{shrink-II-2}. 
On the other hand, this modification of the 
definition of $T_B$ means that the number $512$ in~\eqref{sup-scale} will be replaced with 
$512\cdot  2^{-2\delta}$. 
Thus the ratio in~\eqref{ratio-2} is replaced with $2^{2(\gamma-1)}\cdot 2^{-2\delta(2+\gamma)}$. Then we can 
conclude that $T_B$ is $L^p$ bounded when $\gamma=1$.    
\end{proof}

So far we have proved Theorem~\ref{Main-result-1}. The proof of the Lemma~\ref{L:R-v-v_*-ineq} comes easily 
by modifying that proof. For the convenience of the readers, we restate Lemma~\ref{L:R-v-v_*-ineq} here. 
\begin{lem*}
Let $\mathbb{T}$ be the operator defined by~\eqref{Def:T-2} and $\tau_m$ be the translation 
operator $\tau_m h(\cdot)=h(\cdot+m)$. Suppose $p$ and $\widetilde{p}$ satisfy $1<p,\;\widetilde{p}<\infty$ and 
\[
 \frac{\gamma}{2}\leq \frac{1}{p}\leq 1-\frac{\gamma}{2}\;\;{\rm and}\;\;\frac{1}{\widetilde{p}}
 =\frac{1}{p}+\frac{\gamma}{3}, 
\]
 and both $\leq$ are replaced by $<$ when $\gamma=0$,
 then we have
\[
\sup\limits_{v_*}\|(\tau_{-v_*}\circ \mathbb{T} \circ\tau_{v_*} )h(v)\|_{L^{p}(v)} 
\leq C \|h\|_{L^{\widetilde{p}}} \tag{\ref{E:R-v-ineq}}
\]
and
\[
\sup\limits_{v}\|(\tau_{-v_*}\circ \mathbb{T} \circ\tau_{v_*} )h(v)\|_{L^p(v_*)}
\leq C \|h\|_{L^{\widetilde{p}}}. \tag{\ref{E:R-v_*-ineq}}
\]
\end{lem*}
\begin{proof}
We note that $L^{p}$ spaces are translation invariant, thus~\eqref{E:R-v-ineq} follows from~\eqref{E:R-ineq-pre} immediately. 

The proof of~\eqref{E:R-v_*-ineq} is less obvious, but this is essential the reminiscence of the proof for 
$L^{\widetilde{p}} \rightarrow L^{p}$ boundedness of the  operator $\mathbb{T}$.  
Recalling $A(x,\xi)$ defined in~\eqref{A-x-xi} and $T$ is defined in~\eqref{Define-T}, we write 
\[
\begin{split}
&(\tau_{-v_*}\circ \mathbb{T}\circ \tau_{v_*})\; h(v)\\
&=|v-v_*|^{\gamma}\int_{\omega\in S^2_+} b(\cos\theta)  h(v-
((v-v_*)\cdot\omega)\;\omega)d\omega\\
&=(2\pi)^{-3}\int_{\mathbb{R}^3}e^{iv\cdot\xi}\;\mathbb{A}(v-v_*,\xi)|\xi|^{\gamma}\;(|\xi|^{-\gamma}\widehat{h}(\xi))\;d\xi \\
&=(\tau_{-v_*}\circ T \circ \tau_{v_*}) \circ (-\Delta)^{-\frac{\gamma}{2}} h(v)
\end{split}
\]
where , 
\[
\cos\theta=(v-v_*,\omega)/|v-v_*|,\;0\leq\theta\leq \pi/2,
\]
Using Theorem~\ref{Embedding}, the proof of~\eqref{E:R-v_*-ineq} is equivalent to 
\begin{equation}
\sup\limits_{v}\|(\tau_{-v_*}\circ T \circ\tau_{v_*} )h(v)\|_{L^p(v_*)}
\leq C \|h\|_{L^{p}}. 
\end{equation}
Then the region I and II in~\eqref{D:criterion} is determined by the quantity 
\[
|v-v_{*}||\xi|\cos^2(\theta_0/2)\sin^2(\theta_0/2)
\]
and 
\begin{equation}\label{A+B}
(\tau_{-v_*}\circ T \circ\tau_{v_*} )h(v)=(\tau_{-v_*}\circ T_{A} \circ\tau_{v_*} )h(v)+(\tau_{-v_*}\circ T_{B} \circ\tau_{v_*} )h(v).
\end{equation}

As before, we can write $(\tau_{-v_*}\circ T_{A} \circ\tau_{v_*} )h(v)$ as the  sum of two operators, 
\begin{equation}\label{E:expression}
\int_{\mathbb{R}^3} e^{iv\cdot\xi}e^{-\frac{i}{2} [(v-v_*)\cdot\xi\pm|v-v_*||\xi|]}
\;p_{\pm}(v-v_*,\xi)\;\widehat{h}(\xi)d\xi,
\end{equation}
and it suffices to show that they are $L^{2}\rightarrow L^{2}$ and $\mcH^{1}\rightarrow L^{1}$ bounded
when the corresponding symbols are given. 
For any fixed $v$, as~\eqref{T-lambda-p}, we consider
\begin{equation}\label{E:expression-2}
\int_{\mathbb{R}^3} e^{\frac{i}{2}[(v_*-v)\cdot\xi\mp |v_*-v||\xi|]} p^0_{\pm}(v-v_*,\xi).
\; \widehat{\tau_{-v}h}(\xi)  d\xi.
\end{equation}
The $L^{2}\rightarrow L^{2}$ boundedness of above operator clearly follows from the proof of the $L^{2}\rightarrow L^{2}$ boundedness for $T$ by replacing $x$ with $v_{*}-v$.  To see the boundedness of 
~$(\tau_{-v_*}\circ T_{A} \circ\tau_{v_*} )h(v): \mcH^{1}\rightarrow L^{1}$, we consider two cases according to $\delta$,  
the radius of the support $D$ of one atom. If radius $\delta>1$, the discussion on the $L^{2}$ boundedness ensure the 
result. If $\delta\leq 1$,  we write~\eqref{E:expression}  as  
\begin{equation}
\eqref{E:expression}=\int K_{\pm}(v-v_{*},y) h(y)dy 
\end{equation}
where
\[
K_{\pm}(v-v_{*},y)=\int_{\R^{3}} e^{i [\psi_{\pm}(v-v_{*},\xi)-(y-v)\cdot\xi]}\; \rho_{1}(v-v_{*})\; p^{-1}_{z+}(v-v_{*},\xi) d\xi.
\]
Therefore the argument in the Case 2 of the subsection~\ref{H-1-subsection} still works when $x$ is replaced with 
$v-v_{*}$ and $y$ is replaced with $y-v$ for any fixed $v$. 

For any fixed $v$, we write $(\tau_{-v_*}\circ T_{B} \circ\tau_{v_*} )h(v)$ as 
\[
(\tau_{-v_*}\circ T_{B} \circ\tau_{v_*} )h(v)=\int_{\R^{3}} K_{B}(v-v_{*},y) h(y)dy 
\]
where 
\[
K_B(v-v_{*},y)=\sum_{z}\int_{\mathbb{R}^{3}} \psi_{z,II}(v-v_{*},\xi)\; e^{i((v-v_{*})-y)\cdot\xi}\; a(v-v_{*},\xi) d\xi.
\]
and 
\[
a(v-v_{*},\xi)=(|v-v_{*}||\xi|)^{\gamma}\int_{S^{2}_{+}} e^{i(v-v_{*},\omega)(\xi,\omega)} b(\cos\theta)d\omega.
\]
For any fixed $v$, the above expression is exactly the same form as~\eqref{T-K} by replacing $x$ with $v-v_{*}$,
 and the $L^{p}$ estimates for $T_{B}$ also work here. 
\end{proof}

\section{Well-posedness}\label{well-posed}

In this section, we prove Theorem~\ref{result1} which concerns the Cauchy problem for the cutoff Boltzmann equation 
\begin{equation}\label{E:Cauchy-2}
\left\{
\begin{aligned}
&{\partial_t f}+v\cdot\nabla_x f =Q(f,f)\\
&f(0,x,v)=f_0(x,v)
\end{aligned}
\right.
\end{equation}
in $(0,\infty)\times\mathbb{R}^3\times\mathbb{R}^3,$ where the initial data is small in $L^{3}_{x,v}$ space. 
As we mentioned in the introduction, we should follow the approach of~\cite{HJ23}.  Thus we should make effort to 
minimize the length  while retain the self-contained of the proof.

\subsection{Solve Gain-term only Boltzmann equation}

First we recall the Strichartz estimates for the kinetic transport equation,
\begin{equation}\label{E:KT}
\left\{
\begin{aligned}
&\partial_t u(t,x,v)+v\cdot\nabla_x u(t,x,v) =F(t,x,v),\;\;(t,x,v)\in (0,\infty)\times\mathbb{R}^d\times\mathbb{R}^d,\\
& u(0,x,v)=u_0(x,v).
\end{aligned}
\right.
\end{equation}
To state the Strichartz estimates for~\eqref{E:KT},  we need the following definition.
\begin{defn}\label{D:admissible}
We say that the exponent triplet $(q,r,p)$, for $1\leq p,q,r\leq\infty$ is KT-admissible if 
\begin{equation}
\frac{1}{q}=\frac{d}{2}{\Big ( \frac{1}{p}-\frac{1}{r} }{\Big )}
\end{equation}
\begin{equation}\label{pr-star-2}
1\leq a\leq\infty,\;\;p^*(a)\leq p\leq a, \;\;a\leq r\leq r^*(a)
\end{equation}
except in the case $d=1,\;(q,r,p)=(a,\infty,a/2)$. Here by $a=$HM$(p,r)$ we have denoted the harmonic 
means of the exponents $r$ and $p$, i.e.,
\begin{equation}
\frac{1}{a}=\frac{1}{2}{\Big ( \frac{1}{p}+\frac{1}{r} }{\Big )}
\end{equation}
Furthermore, the exact lower bound $p^*$ to $p$ and the exact upper bound $r^*$ to $r$ are 
\begin{equation}\label{pr-star-1}
\left\{
\begin{array}{lll}
p^*(a)=\frac{da}{d+1}, & r^*(a)=\frac{da}{d-1} & {\rm if}\;\frac{d+1}{d}\leq a\leq\infty, \\
p^*(a)=1, & r^*(a)=\frac{a}{2-a} & {\rm if}\; 1\leq a\leq \frac{d+1}{d}.
\end{array}
\right.
\end{equation}
\end{defn}
The triplets of the form $(q,r,p)=(a,r^*(a),p^*(a))$ for 
$\frac{d+1}{d}\leq a<\infty$ are called endpoints.  
The endpoint Strichartz estimate for the kinetic equation is false in all dimensions has been proved recently by 
Bennett, Bez, Guti\'{e}rrez and Lee~\cite{BBGL14}.

The mild solution of the kinetic equation~\eqref{E:KT} can be written as 
\begin{equation}\label{integal-equation}
u=U(t)u_0+W(t)F
\end{equation}
where 
\begin{equation}\label{D:solution-maps}
U(t)u_0=u_0(x-vt,v)\;,\; W(t)F=\int_0^{t} U(t-s)F(s)ds.
\end{equation}
The estimates for the operator $U(t)$ and $W(t)$ respectively in the mixed Lebesgue norm 
$\|\cdot\|_{L^q_tL^r_xL^p_v}$ are called homogeneous and inhomogeneous Strichartz 
estimates respectively. We record the Strichartz estimates  for the equation~\eqref{integal-equation} 
in the following Proposition.
\begin{prop}[\cite{Ovc11},\cite{BBGL14}]\label{Strichartz-est} 
Let $u$ satisfies~\eqref{E:KT}.  The estimate
\begingroup
\large
\begin{equation}\label{E:Strichartz}
\|u\|_{L^q_tL^r_xL^p_v}\leq C(q,r,p,d)( \|u_0\|_{L^a_{x,v}} +\|F\|_{L^{\tilde{q}'}_tL^{\tilde{r}'}_xL^{\tilde{p}'}_v}  )
\end{equation}
\endgroup
holds for all $u_0\in L^a_{t,x}$ and all $F\in {L^{\tilde{q}'}_tL^{\tilde{r}'}_xL^{\tilde{p}'}_v} $ if and only if 
$(q,r,p)$ and $(\tilde{q},\tilde{r},\tilde{p})$ are two KT-admissible exponents triplets and $a=$HM$(p,r)=$HM$(\tilde{p}',\tilde{r}')$ with the exception of $(q,r,p)$ being an endpoint triplet. Here $\tilde{p}'$ is the conjugate 
exponent of $\tilde{p}$ and so on. 
\end{prop}

We also need to build the weighted Strichartz estimates.   
We consider the weight $\lr{v}^{\ell},\;\ell\in\mathbb{R}$ as the multiplication operator. 
Using the  notations of~\eqref{D:solution-maps}, we note that the following communication relations hold:
\begin{equation}
 \lr{v}^{\ell} U(t)u_{0}=U(t) \lr{v}^{\ell} u_{0},\; \lr{v}^{\ell} W(t)F=W(t) \lr{v}^{\ell} F.
\end{equation}
Combining above facts with Proposition~\ref{Strichartz-est}, we have the following result.
\begin{cor}\label{W-Strichartz}
Let $\ell\in\mathbb{R}$ and $u$ satisfy the kinetic transport equation~\eqref{E:KT}.  The estimate
\begingroup
\large
\begin{equation}\label{E:Strichartz-w}
\begin{split}
& \|  \lr{v}^{\ell} u \|_{L^q_tL^r_xL^p_v}\\
&{\hskip 1cm}\leq C(q,r,p,d){\big (} \|   \lr{v}^{\ell} u_0\|_{L^a_{x,v}} +\|   \lr{v}^{\ell} F\|_{L^{\tilde{q}'}_tL^{\tilde{r}'}_xL^{\tilde{p}'}_v}  {\big)}
\end{split}
\end{equation}
\endgroup
holds for all $\lr{v}^{\ell} u_0\in L^a_{t,x}$ and all $\lr{v}^{\ell} F\in {L^{\tilde{q}'}_tL^{\tilde{r}'}_xL^{\tilde{p}'}_v} $ if and only if 
$(q,r,p)$ and $(\tilde{q},\tilde{r},\tilde{p})$ are two KT-admissible exponents triplets and $a=$HM$(p,r)=$HM$(\tilde{p}',\tilde{r}')$ with the exception of $(q,r,p)$ being an endpoint triplet.
\end{cor}

First we study  the Cauchy problem for the gain term only equation 
\begin{equation}\label{G-Cauchy}
\left\{
\begin{aligned}
&{\partial_t f_+}+v\cdot\nabla_x f_+ =Q^{+}(f_+,f_+),\;\;(t,x,v)\in (0,\infty)\times\mathbb{R}^3\times\mathbb{R}^3,\\
&f_+(0,x,v)=f_0(x,v).
\end{aligned}
\right.
\end{equation}
We define the solution map by
\begin{equation}\label{G-integral-equation}
Sf_+(t,x,v)=U(t)f_0+W(t)Q^{+}(f_+,f_+).
\end{equation}
From~\eqref{G-integral-equation} and Corollary~\ref{W-Strichartz}, we will see that it holds 
the estimates
\begin{equation}\label{contraction}
\begin{split}
\|\lr{v}^{\ell} Sf_+\|_{L^q_tL^r_xL^p_v}& \leq C_{0}\|\lr{v}^{\ell} f_{0}\|_{L^a_{x,v}}+C_{1}\|\lr{v}^{\ell} 
Q^+(f_+,f_+)\|^{2}_{L^{\tilde{q}'}_tL^{\tilde{r}'}_xL^{\tilde{p}'}_v} \\
&  \leq C_{0}\|\lr{v}^{\ell} f_{0}\|_{L^a_{x,v}}+C_{2}\|\lr{v}^{\ell} f_+\|^2_{L^q_tL^r_xL^p_v},
\end{split}
\end{equation}
for suitable $\ell\in\mathbb{R}$ and Strichartz spaces $L^q_tL^r_xL^p_v$ and $L^{\tilde{q}'}_tL^{\tilde{r}'}_xL^{\tilde{p}'}_v$.
Then the contraction mapping argument will work if the weighted initial data is small in space $L^a_{x,v}$.
The key ingredient is that  there exist suitable $\ell\in\mathbb{R}$, admissible triplets 
$(q,r,p)$ and $(\tilde{q}, \tilde{r},\tilde{p})$
with $HM(p,r)=HM(\tilde{p}',\tilde{r}')$ such that  the estimate
\begin{equation}\label{desire-est}
\|\lr{v}^{\ell} Q^{+}(f_+,f_+)\|_{L^{\tilde{q}'}_tL^{\tilde{r}'}_xL^{\tilde{p}'}_v} \leq
 C\|\lr{v}^{\ell} f_+\|^{2}_{L^q_tL^r_xL^p_v}
\end{equation}
holds. 

From view point of scaling invariance, when $N=3,\;\ell=0$, one must have $(\mbq,\mbr,\mbp)$ lie in the set
\begin{equation}\label{solvable-g}
\{ (\mbq,\mbr,\mbp) | \;\frac{1}{\mbq}=\frac{3}{\mbp}-1\;,\;\frac{1}{\mbr}=\frac{2}{3}-\frac{1}{\mbp}\;,\;
\frac{1}{3}<\frac{1}{\mbp}<\frac{4}{9}\}, 
\end{equation} 
so that the nonlinear estimates~\eqref{desire-est} hold for corresponding $(\tilde{q}',\tilde{r}',\tilde{p}')$. 
At the same time it requires $\gamma=-1$ and $f_0\in L^3_{x,v}$. Then the Cauchy problem for the gain term only equation 
has a unique solution $f_+\in C([0,\infty),L^3_{x,v})\cap L^{\mbq}_tL^{\mbr}_xL^{\mbp}_v$ as authors pointed out 
in~\cite{HJ17}.  Later on, the authors prove in~\cite{HJ23} that the gain term only equation with 
cutoff soft potential model satisying $-1<\gamma\leq 0$ has a unique solution $\lr{v}^{(1+\gamma)+} f_+\in C([0,\infty),L^3_{x,v})\cap L^{\mbq}_tL^{\mbr}_xL^{\mbp}_v$ by
building the weighted version of~\eqref{desire-est} with $\ell=(1+\gamma)+$ for the same 
Strichartz  triplet $(\mbq,\mbr,\mbp)$ lie in~\eqref{solvable-g}.  The idea in~\cite{HJ23} is that  the introduction of 
the weighted space can modify the original  scaling exponents of the the functions spaces for the case $-1<0 \leq 0$. 
By choosing weight $\ell=(1+\gamma)+$, the altered exponents due to different $\gamma$ match the admissible exponents 
given in~\eqref{solvable-g}.

The above method does not fit with the weighted estimate for the gain term 
by Alonso, Carneiro and Gamba~\cite{ACG10}
with $\gamma\geq 0$, i.e., $1/p+1/q=1+1/r,\;\ell\in\mathbb{R}$ and 
\begin{equation}
\|\lr{v}^{\ell} Q^{+}(f,g)\|_{L^{r}_v}\leq C\|\lr{v}^{\ell+\gamma} f\|_{L^{p}_v}
\|\lr{v}^{\ell+\gamma} g\|_{L^{q}_v},
\end{equation}  
due to  the increase of moment. Fortunately, we have the Theorem~\ref{T:Gain-p-w-est}.

\begin{thm*}
Let $\ell_{0}\geq 0$, $1<\mt{p}, \mt{q}, \mt{r} <\infty$,\; $0\leq  \gamma\leq 1$ and
\[
\frac{1}{\mt{p}}+\frac{1}{\mt{q}}=1+\frac{1}{\mt{r}}+\frac{\gamma}{3},\;\;
\frac{\gamma}{6}\leq \frac{1}{\mt{r}}\leq 1-\frac{5\gamma}{6}.  \tag{\ref{scaling-relation}}
\]
Consider
\[
B(v-v_*,\omega)=|v-v_*|^{\gamma}\;\cos\theta.
\]
Then the bilinear operator $Q^{+}(f,g)$ satisfies  
\[
\|\lr{v}^{\ell_{0}}Q^{+}(f,g)\|_{L^{\mt{r}}_{v}(\mathbb{R}^3)}\leq C\|\lr{v}^{\ell_{0}} f\|_{L^{\mt{p}}_{v}(\mathbb{R}^3)}
\|\lr{v}^{\ell_{0}}g\|_{L^{\mt{q}}_{v}(\mathbb{R}^3)}. \tag{\ref{W-convolution}}
\]
If $\ell_{1}>3/m$ and $1<\mt{p}_{m}, \mt{q}_{m}, m, \mt{r}_{m} <\infty$ satisfy
\[
\frac{1}{\mt{p}_{m}}+\frac{1}{m}<1\;,\;\frac{1}{\mt{q}_{m}}+\frac{1}{m}<1, \tag{\ref{minus-size-condition}}
\]
and
\[
\frac{1}{\mt{p}_{m}}+\frac{1}{\mt{q}_{m}}+\frac{1}{m}=1+\frac{1}{\mt{r}_{m}}+\frac{\gamma}{3}, \;\;
\frac{\gamma}{6}\leq \frac{1}{\mt{r}_{m}}\leq 1-\frac{5\gamma}{6} \tag{\ref{scaling-relation-w}}
\]
then we have  
\[
\|\lr{v}^{\ell_{1}} Q^{+}(f,g)\|_{L^{\mt{r}_{m}}_{v}(\mathbb{R}^3)}\leq C(\mt{p}_{m},\ell_{1}) 
\|\lr{v}^{\ell_{1}} f\|_{L^{\mt{p}_{m}}_{v}(\mathbb{R}^3) }\|\lr{v}^{\ell_{1}} g\|_{L^{\mt{q}_{m}}_{v}(\mathbb{R}^3)}.
\tag{\ref{W-convolution-m}}
\]
\end{thm*}

Now we demonstrate the application of the  Theorem~\ref{T:Gain-p-w-est}. 
We select one admissible triplet $(\mbq,\mbr,\mbp)$ from~\eqref{solvable-g} 
and its corresponding triplet $({1}/{\widetilde{\mbq}'},{1}/{\widetilde{\mbr}'},{1}/{\widetilde{\mbp}'})$  
which is given by 
\[
(\frac{1}{\widetilde{\mbq}'},\frac{1}{\widetilde{\mbr}'},\frac{1}{\widetilde{\mbp}'})
=(\frac{6}{\mbp}-2 ,\;\frac{4}{3}-\frac{2}{\mbp},\;\frac{2}{\mbp}-\frac{2}{3}).
\]
More precisely, for any given $0<\varepsilon< {1}/{9}$, we let
\begin{equation}\label{solvable-triplets}
\begin{split}
& (\frac{1}{\mbq},\;\frac{1}{\mbr}, \;\frac{1}{\mbp})
=( \frac{1}{3}-3\varepsilon,\; \frac{2}{9}+\varepsilon,\; \frac{4}{9}-\varepsilon),\\
& (\frac{1}{\widetilde{\mbq}'},\; \frac{1}{\widetilde{\mbr}'}, \; \frac{1}{\widetilde{\mbp}'})
=( \frac{2}{3}-6\varepsilon, \;\frac{4}{9}+2\varepsilon, \;\frac{2}{9}-2\varepsilon).
\end{split}
\end{equation}
Then the relation~\eqref{scaling-relation-w},~\eqref{minus-size-condition} are satisfied as
\begin{equation}
\frac{1}{\mbp}+\frac{1}{\mbp}+\frac{1+\gamma}{3}=1+\frac{1}{\widetilde{\mbp}'}+\frac{\gamma}{3},
\;\frac{1}{\mbp}+\frac{1}{3}<1.
\end{equation}
with 
\[
\frac{1}{m}=\frac{1+\gamma}{3}
\]
Thus if $\ell_{1}>1+\gamma$, ~\eqref{W-convolution-m} in Theorem~\ref{T:Gain-p-w-est} becomes
\begin{equation}\label{desire-est-H}
\|\lr{v}^{\ell_{1}}Q^+(f,f)\|_{L^{\widetilde{\mbp}'}_{v}} \leq C \|\lr{v}^{\ell_{1}} f\|^{2}_{L^{{\mbp}}_{v}}.
\end{equation}
However the inequality in~\eqref{scaling-relation} becomes 
\[
\frac{\gamma}{6}\leq \frac{1}{\widetilde{\mbp}'} \leq 1-\frac{5\gamma}{6} 
\]
and which means that $\varepsilon$ needs to be further restricted to
\[
\max\{0, \frac{5\gamma}{12}-\frac{7}{18} \}  \leq \varepsilon \leq \frac{1}{9}-\frac{\gamma}{12}.
\]


Using~\eqref{desire-est-H} and the same argument as~\cite{HJ17,HJ23}, we conclude the following result.
\begin{prop}\label{Thm-W-Gain} 
Assume the collision kernel $B$ takes the form~\eqref{D:kernel} with $0\leq \gamma \leq 1$. 
Let $N=3,\;\ell_{1}>1+\gamma$ and $\varepsilon$ satisfy
\begin{equation}
\max\{0,  \frac{5\gamma}{12}-\frac{7}{18}   \}  \leq \varepsilon \leq \frac{1}{9}-\frac{\gamma}{12}.
\end{equation}
The Cauchy problem for the gain term only Boltzmann equation~\eqref{G-Cauchy} is globally 
wellposed in weighted $L^3_{x,v}$ when the initial data is small enough. 
More precisely, there exists a small number $\eta>0$ such that if the initial data 
$f_0$ is in the set
\[B^{\ell}_{\eta}=\{f_0\in L^3_{x,v} (\mathbb{R}^3\times
\mathbb{R}^3):  \| \lr{v}^{\ell_{1}} f_0\|_{L^3_{x,v}}<\eta\}\subset L^3_{x,v},
\] 
there exists a globally unique  mild solution 
\[
\lr{v}^{\ell_{1}} f_+\in C([0,\infty),L^3_{x,v})\cap
L^{\mbq}([0,\infty],L^{\mbr}_xL^{\mbp}_v) 
\]
where the triplet 
\begin{equation}\label{solvable-g-2}
 (\frac{1}{\mbq},\frac{1}{\mbr},\frac{1}{\mbp})=(\frac{1}{3}-3\varepsilon,\;\frac{2}{9}+\varepsilon,\;\frac{4}{9}-\varepsilon). 
\end{equation} 
The solution map $\lr{v}^{\ell_{1}} f_0\in B^{\ell}_{\eta}\rightarrow \lr{v}^{\ell_{1}} f_+\in 
L^{\mbq}_tL^{\mbr}_xL^{\mbp}_v$ 
is Lipschitz continuous and the solution $\lr{v}^{\ell_{1}} f_+$ scatters with respect to the kinetic 
transport operator in $L^3_{x,v}$.  
\end{prop}

Before we consider the full Boltzmann equation, we also need the following result. 
\begin{cor}\label{pos-gain} 
Under the same conditions as Proposition~\ref{Thm-W-Gain}, if we furthermore  assume 
$f_0\geq 0$, then the solution $f_{+}$ is also non-negative. 
\end{cor}
\begin{proof}[Proof of Corollary~\ref{pos-gain}]
When $f_{0}\geq 0$, we can see that the solution is non-negative by iterating Duhamel's formula: 
\begin{equation}\label{positive}
\begin{split}
f_+(t)&=U(t)f_{0}+\int_{0}^{t}U(t-t_{1})Q^{+}(U(t_{1})f_{0},U(t_{1})f_{0}) dt_{1} \\
&+\int_{0}^{t}\int_{0}^{t_{1}}U(t-t_{1})Q^{+}{\big (}U(t_{1}-t_{2})Q^{+}(U(t_{2})f_{0},U(t_{2})f_{0}
{\big )}, U(t_{1})f_{0})dt_{2}dt_{1}+\cdots
\end{split}
\end{equation}
Since each term in the right hand side is non-negative, it suffices to show the series converges. 
Repeatedly using weighted Strichartz estimates of Corollary~\ref{W-Strichartz} with 
triplets~\eqref{solvable-triplets}, in particular~\eqref{desire-est-H}, we get 
that the Strichartz norm of $\lr{v}^{\ell_{1}}\; f_+$ for 
chosen triplets is bounded by a series of  $L^3_{x,v}$ norm of $\lr{v}^{\ell_{1}}\;f_0$ which 
converges since the weighted initial data is small enough. 
\end{proof}

\begin{prop}\label{P-p2-est-w}
Let $N=3$ and $a_2=15/8$. Under the same assumptions as Proposition~\ref{Thm-W-Gain}, if we  further assume 
$\|\lr{v}^{\ell_{1}} f_0\|_{L^{a_2}_{x,v}}<\infty$, then solution $f_+$ in Proposition~\ref{Thm-W-Gain} also satisfies
\begin{equation}
\|\lr{v}^{\ell_{1}}f_+\|^2_{L^{q_2}_t L^{r_2}_x L^{p_2}_v}<\infty
\end{equation} 
where $(1/q_2,1/r_2,1/p_2)=(1/2,\;11/30,\;21/30)$ is a KT-admissible triplet with $1/p_2+1/r_2=2/a_2$. 
\end{prop}
\begin{proof}  By the assumption that $\| \lr{v}^{\ell_{1}} f_0\|_{L^{a_2}_{x,v}}<\infty$,
we claim that there exist  KT-admissible triplets $(q_{2},r_{2},p_{2})$ and
$(\tilde{q}_{2},\tilde{r}_{2},\tilde{p}_{2})$ such that 
\begin{equation}\label{Gain-p2}
\|\lr{v}^{\ell_{1}} Q^+(f_+,f_+)\|_{L^{\tilde{q}'_{2}}_tL^{\tilde{r}'_{2}}_xL^{\tilde{p}'_{2}}_v} \leq C
\|\lr{v}^{\ell_{1}} f_+\|_{L^{\mbq}_tL^{\mbr}_xL^{\mbp}_v}\|\lr{v}^{\ell_{1}} f_+\|_{L^{q_{2}}_tL^{r_{2}}_xL^{p_{2}}_v},
\end{equation}
and $a_{2}=HM(p_{2},r_{2})=HM(\tilde{p}'_{2},\tilde{r}'_{2})$ where $\tilde{p}'_2$ means 
the conjugate of $\tilde{q}_2$ and so on.
From this together with weighted Strichartz estimate \eqref{E:Strichartz-w}, we have
\begin{equation*}
\begin{split}
\|\lr{v}^{\ell_{1}} f_+\|_{L^{q_{2}}_tL^{r_{2}}_xL^{p_{2}}_v}&=\|\lr{v}^{\ell_{1}} Sf_+\|_{L^{q_{2}}_tL^{r_{2}}_xL^{p_{2}}_v}\\
& \leq C_{0}\|\lr{v}^{\ell_{1}} f_{0}\|_{L^{a_{2}}_{x,v}}+C_{1}\|\lr{v}^{\ell_{1}} Q^+(f_+,f_+)\|
_{L^{\tilde{q}'_{2}}_tL^{\tilde{r}'_{2}}_xL^{\tilde{p}'_{2}}_v} \\
&  \leq C_{0}\|\lr{v}^{\ell_{1}} f_{0}\|_{L^{a_{2}}_{x,v}}+C_{2}\|\lr{v}^{\ell_{1}}f_+\|_{L^{\mbq}_tL^{\mbr}_xL^{\mbp}_v}
\|\lr{v}^{\ell_{1}} f_+\|_{L^{q_{2}}_tL^{r_{2}}_xL^{p_{2}}_v}.
\end{split}
\end{equation*}
The proof of Proposition~\ref{Thm-W-Gain} implies that $C_2\|\lr{v}^{\ell_{1}} f_+\|_{L^{\mbq}_tL^{\mbr}_xL^{\mbp}_v}<1$ 
 where $(\mbq,\mbr,\mbp)$ is given in~\eqref{solvable-g-2}.  Thus we have 
\[
\|\lr{v}^{\ell_{1}} f_+\|_{L^{q_{2}}_tL^{r_{2}}_xL^{p_{2}}_v} \leq C_3 \|\lr{v}^{\ell_{1}} f_{0}\|_{L^{a_{2}}_{x,v}}<\infty.
\]
\indent To prove that there exist triplets $(q_{2},r_{2},p_{2})$ and
$(\tilde{q}_{2},\tilde{r}_{2},\tilde{p}_{2})$ so that~\eqref{Gain-p2} holds, 
we define $\tilde{p}'_{2}$ and $\tilde{r}'_{2}$ as  follows: 
\begin{equation}\label{choose-p-2}
\frac{1}{\tilde{q}'_{2}}=\frac{1}{\mbq}+\frac{1}{2},\;\frac{1}{\tilde{r}'_{2}}:=\frac{1}{\mbr}+\frac{11}{30}
,\;\frac{1}{\tilde{p}'_{2}}:=\frac{1}{\mbp}+\frac{1}{30}.
\end{equation}
 It is easy to check that the following conditions are satisfied,
\begin{subequations}
\begin{align}
&\frac{1}{\mbp}+\frac{1}{p_{2}}+\frac{1+\gamma}{3}=1+\frac{1}{\tilde{p}'_{2}}+\frac{\gamma}{3}\;,\; 
\frac{1}{\mbp}<\frac{1}{\tilde{p}'_{2}}<\frac{1}{p_{2}}\;, \;\label{uni-v}\\
&\frac{1}{\mbr}+\frac{1}{r_{2}}=\frac{1}{\tilde{r}'_{2}} \;, \label{uni-x}\\
&\frac{1}{p_{2}}+\frac{1}{r_{2}}=\frac{1}{\tilde{p}'_{2}}+\frac{1}{\tilde{r}'_{2}}=\frac{2}{a_{2}}=\frac{16}{15}\;, \label{uni-har}\\
&\frac{1}{q_{2}}=\frac{3}{2}\;(\frac{1}{p_{2}}-\frac{1}{r_{2}})
\;,\;\frac{1}{\tilde{q}_{2}}=\frac{3}{2}\;(\frac{1}{\tilde{p}_{2}}-\frac{1}{\tilde{r}_{2}})\;,\;\qquad\qquad\;\label{uni-q}. \\
&\frac{1}{\mbq}+\frac{1}{q_{2}}= \frac{1}{\tilde{q}'_{2}}.\;\label{uni-time}
\end{align}
\end{subequations}
Considering the first equation of~\eqref{uni-v}, and letting 
\[
1/\mbp=1/\mt{p}_{m},\;1/p_2=1/\mt{q}_{m},\;1/\tilde{p}'_{2}=1/\mt{r}_{m},
\]
 we see that it is exactly the relation in~\eqref{scaling-relation-w} with $1/m=(1+\gamma)/3$.  
Thus H\"{o}lder inequality,~\eqref{W-convolution},
~\eqref{uni-x}, and~\eqref{uni-time} imply that~\eqref{Gain-p2}  holds. 
Also the relations~\eqref{uni-har} and~\eqref{uni-q} ensure that $(q_{2},r_{2},p_{2})$ and
$(\tilde{q}_{2},\tilde{r}_{2},\tilde{p}_{2})$ are KT-admissible. 
\end{proof}

We need the estimates for the loss term whose exponents satisfying~\eqref{scaling-relation-w}, in particular 
we take $1/m=(1+\gamma)/3$ as before. We have the following.  
\begin{prop}\label{P:Loss}
Assume $B(v-v_*,\omega)$ defined in~\eqref{D:kernel} satisfies $0\leq \gamma\leq 1$. 
Let  the exponents $(\mbq,\mbr,\mbp),(q_{2},r_{2},p_{2})$,$(\tilde{q}_{2},\tilde{r}_{2},\tilde{p}_{2})$ 
and $a_{2}$ the same as the Proposition~\ref{P-p2-est-w}.  Let 
\begin{equation}\label{L-2-5}
\ell_{2}>\gamma+\frac{9}{10},\;\ell_{3}=(\ell_{2}+\gamma)+1/9
\end{equation}
Suppose $\lr{v}^{\ell_{3}} f_{1}\in L^{\mbq}_tL^{\mbr}_xL^{\mbp}_v$ and $\lr{v}^{\ell_{2}} f_{2}\in L^{q_{2}}_tL^{r_{2}}_xL^{p_{2}}_v$.
Then we have 
\begin{equation}\label{Loss-ineq}
\| \lr{v}^{\ell_{2}} Q^{-}(f_1,f_2)\|_{L^{\tilde{q}'_{2}}_tL^{\tilde{r}'_{2}}_xL^{\tilde{p}'_{2}}_v} \leq C
\|\lr{v}^{\ell_{3}} f_1\|_{L^{\mbq}_tL^{\mbr}_xL^{\mbp}_v}\|\lr{v}^{\ell_{2}} f_2\|_{L^{q_{2}}_tL^{r_{2}}_xL^{p_{2}}_v}.
\end{equation}
\end{prop}
\begin{proof}
The inequality~\eqref{Loss-ineq} also follows from~\eqref{uni-v}-\eqref{uni-time}, provided we 
can prove that
\begin{equation}\label{Loss-v}
\| \lr{v}^{\ell_{2}} Q^{-}(f_1,f_2)\|_{L^{\tilde{p}'_{2}}_v} \leq C
\|\lr{v}^{\ell_{3}} f_1\|_{L^{\mbp}_v}\|\lr{v}^{\ell_{2}} f_2\|_{L^{p_{2}}_v}
\end{equation}
where $\tilde{p}'_{2},\;\mbp,\;p_{2}$ satisfy~\eqref{uni-v}.  It is convenient to rewrite~\eqref{uni-v} as
\begin{equation}
\frac{1}{\mbp}+\frac{1}{p_{2}}+\frac{1}{3}=1+\frac{1}{\tilde{p}'_{2}}.
\end{equation} 

Since $0\leq\gamma\leq 1$, we have $|v-v_{*}|^{\gamma}\leq 2\max\{\lr{v}^{\gamma},\lr{v_{*}}^{\gamma}\}$. Therefore
\begin{equation}
\begin{split}
&{\big |} \lr{v}^{\ell_{2}} Q^{-}(f_{1},f_{2}) {\big |} \\
&\leq C{\Big (} {\big |} \lr{v}^{\ell_{2}+\gamma} f_{1}(v)\int f_{2}(v_{*})  dv_{*} {\big |} +
{\big |} \lr{v}^{\ell_{2}} f_{1}(v)\int f_{2}(v_{*}) \lr{v_{*}}^{\gamma} dv_{*} {\big |} {\Big )}.
\end{split}
\end{equation}

Note that $1/\tilde{p}'_{2}=1/\mbp+1/30$ by~\eqref{choose-p-2}. 
Using the H\"{o}lder inequality, we have
\begin{equation}\label{loss-1}
\begin{split}
& {\big \|} \lr{v}^{\ell_{2}+\gamma} f_{1}(v)\int f_{2}(v_{*})  dv_{*} {\big \|}_{L^{\tilde{p}'_{2}}_v}  \\
& \leq {\big \|} \lr{v}^{\ell_{3}} f_{1}(v) {\big \|}_{L^{\mbp}}
{\big \|} \lr{v}^{-1/9} {\big \|}_{L^{30}}
{\big \|} \lr{v}^{\ell_{2}} f_{2} {\big \|}_{L^{p_{2}}} {\big \|} \lr{v}^{-\ell_{2}} {\big \|}_{L^{\frac{10}{3}}} \\
& \leq C {\big \|} \lr{v}^{\ell_{3}} f_{1}(v) {\big \|}_{L^{\mbp}} {\big \|} \lr{v}^{\ell_{2}} f_{2} {\big \|}_{L^{p_{2}}}.
\end{split}
\end{equation}
where the last inequality holds because~\eqref{L-2-5} implies $\frac{10}{3}\ell_{2}>3$. 
Similarly, the assumption~\eqref{L-2-5} gives $\frac{10}{3}(\ell_{2}-\gamma)>3$, and we have 
\begin{equation}\label{loss-2}
\begin{split}
& {\big \|} \lr{v}^{\ell_{2}} f_{1}(v)\int f_{2}(v_{*}) \lr{v_{*}}^{\gamma} dv_{*} {\big \|}_{L^{\tilde{p}'_{2}}_v}  \\
& \leq {\big \|} \lr{v}^{\ell_{3}} f_{1}(v) {\big \|}_{L^{\mbp}}
{\big \|} \lr{v}^{-(\gamma+1/9)} {\big \|}_{L^{30}}
{\big \|} \lr{v}^{\ell_{2}} f_{2} {\big \|}_{L^{p_{2}}} {\big \|} \lr{v}^{-\ell_{2}+\gamma} {\big \|}_{L^{\frac{10}{3}}} \\
& \leq C {\big \|} \lr{v}^{\ell_{3}} f_{1}(v) {\big \|}_{L^{\mbp}} {\big \|} \lr{v}^{\ell_{2}} f_{2} {\big \|}_{L^{p_{2}}}.
\end{split}
\end{equation}
The result follows from~\eqref{loss-1} and~\eqref{loss-2}.
\end{proof}

\subsection{Solve the full equation}
Now we are ready to prove the well-posedness result. 
For the convenience of reader, we restate the main result Theorem~\ref{result1} in the following. 
\begin{prop}\label{result2}
 Assume the kernel B defined in~\eqref{D:kernel} satisfies  $0\leq \gamma\leq 1$.  
Let $\ell>2\gamma+\frac{10}{9}$ and $\varepsilon>0$ satisfy
\[
\max {\big \{ } 0,  \frac{5\gamma}{12}-\frac{7}{18}  \}  \leq \varepsilon \leq \frac{1}{9}-\frac{\gamma}{12}.
\] 
There exists a small number $\eta>0$ such that if the initial data 
$$
f_{0}\in {\mathbb B}^{\ell}_{\eta}=\{f_{0} | f_0\geq 0,\; \| \lr{v}^{\ell} f_{0}\|_{L^{3}_{x,v}}<\eta,\;
\|\lr{v}^{\ell}f_0\|_{L^{15/8}_{x,v}}<\infty\}\subset L^3_{x,v},
$$
then the Cauchy problem~\eqref{E:Cauchy} admits a unique and non-negative mild solution 
\[
\lr{v}^{\ell} f\in C([0,\infty),L^3_{x,v})\cap
L^{\mbq}([0,\infty],L^{\mbr}_xL^{\mbp}_v), 
\]
where the triple 
\begin{equation}
(\frac{1}{\mbq},\frac{1}{\mbr},\frac{1}{\mbp})=(\frac{1}{3}-3\varepsilon,\;\frac{2}{9}+\varepsilon,\;\frac{4}{9}-\varepsilon). 
\end{equation}
The solution map $\lr{v}^{\ell}f_0\in {\mathbb B}^{\ell}_{\eta}\rightarrow \lr{v}^{\ell} 
f\in L^{\mbq}_tL^{\mbr}_xL^{\mbp}_v$ 
is Lipschitz continuous and the solution $\lr{v}^{\ell} f$ scatters with respect to the kinetic 
transport operator in $L^3_{x,v}$. 
\end{prop}
\begin{proof}
Recall the definition of the loss term~\eqref{gain-loss-def} and denote 
\[
Q^{-}(f_{1},f_{2})=f_{1}L(f_{2}),
\]
for the convenience of the analysis later.
Next we recall that the Kaniel-Shinbrot iteration ensures that if there exist measurable functions $h_1,h_2,g_1,g_2$ which 
satisfy the beginning condition, i.e.,   
\begin{equation}\label{begin-condition}
 0\leq h_{1}\leq h_{2}\leq g_{2}\leq g_{1},
\end{equation}
then the iteration($n\geq 2$),
\begin{equation}\label{Iteration}
\begin{split}
& \partial_{t} g_{n+1}+v\cdot\nabla_{x} g_{n+1}+ g_{n+1}L(h_{n})=Q^{+}(g_{n},g_{n}) \\
& \partial_{t} h_{n+1}+v\cdot\nabla_{x} h_{n+1}+ h_{n+1} L (g_{n})=Q^{+}(h_{n},h_{n}) \\
& g_{n+1}(0)=h_{n+1}(0)=f_{0},
\end{split}
\end{equation}
will induce the monotone sequence of measurable functions
\begin{equation}\label{monotone-sequence}
0\leq h_1\leq h_n\leq h_{n+1}\leq g_{n+1}\leq g_n\leq g_1.
\end{equation}
Thus the monotone convergence theorem implies the existence of the limits $g, h$ with $0\leq h\leq g\leq g_1$
which satisfy 
\begin{equation}\label{Iteration-limit}
\begin{split}
& \partial_{t} g+v\cdot\nabla_{x} g+ g L(h )=Q^{+}(g,g) \\
& \partial_{t} h+v\cdot\nabla_{x} h+ h L (g)=Q^{+}(h,h) \\
& g(0)=h(0)=f_{0}
\end{split}
\end{equation}
Then the Cauchy problem for the full equation~\eqref{E:Cauchy} has a solution if we can further prove $g=h(=f)$.
Note that the uniqueness of the solution is not a consequence of above argument. 

Based on the Proposition~\ref{Thm-W-Gain}, it is natural to choose $h_1\equiv 0$ and $g_1=f_{+}$ 
where $f_{+}\geq 0$ is the solution of gain term only Boltzmann equation~\eqref{G-Cauchy} 
with initial data  $\| \lr{v}^{\ell} f_{0}\|_{L^{3}_{x,v}}<\eta,\;\ell>2\gamma+\frac{10}{9}$. 
Choose $\ell_{1}$ satisfies $1+\gamma<\ell_{1}\leq \ell $, the Proposition~\ref{Thm-W-Gain}  ensures
\begin{equation}\label{f-plus-finite}
\lr{v}^{\ell_{1}} f_+\in C([0,\infty),L^3_{x,v})\cap
L^{\mbq}([0,\infty],L^{\mbr}_xL^{\mbp}_v),
\end{equation}
where $(\mbq,\mbr,\mbp)=(({1}/{3})-3\varepsilon,({2}/{9})+\varepsilon,({4}/{9})-\varepsilon)$.

According to~\eqref{Iteration}, we want to 
find $h_{2}$ and $g_{2}$ through 
\begin{equation}\label{h-g-2}
\begin{split}
& \partial_{t} g_{2}+v\cdot\nabla_{x} g_{2}+ g_{2}L(h_{1})=Q^{+}(g_{1},g_{1}) \\
& \partial_{t} h_{2}+v\cdot\nabla_{x} h_{2}+ h_{2} L (g_{1})=Q^{+}(h_{1},h_{1}) \\
& g_{2}(0)=h_{2}(0)=f_{0}
\end{split}
\end{equation}
Since $h_{1}\equiv 0$, the first equation of~\eqref{h-g-2}  is exactly the gain term only Boltzmann equation. Hence 
we have $g_{2}=g_{1}=f_{+}$ by Proposition~\ref{Thm-W-Gain} and
\begin{equation}\label{g-2}
g_{2}(t)=U(t)f_{0}+\int_{0}^{t} U(t-s)Q^{+}(g_{1},g_{1})(s)ds. 
\end{equation}

Next we want to solve the second equation of~\eqref{h-g-2}  with the given $g_{1}$. Actually, we have 
\begin{equation}\label{h-2}
  h_{2}(t)=U(t)f_{0}\; e^{-\int_{0}^{t} U(t-s) L(g_{1})(s) ds},
\end{equation}
if $L(g_{1})$ is pointwisely a.e. well-defined  when  $g_{1}=f_{+}$ satisfies~\eqref{f-plus-finite}. 
We recall that the assumption $\|\lr{v}^{\ell} f_{0}\|_{L^{15/8}_{x,v}}<\infty$ and Proposition~\ref{P-p2-est-w}
ensure $g_{1}=f_{+}$ and 
\begin{equation}\label{p2boundf+}
\lr{v}^{\ell_{1}} g_{1}\in L^{q_{2}}([0,\infty],L^{r_{2}}_xL^{p_{2}}_v),\;1+\gamma<\ell_{1}\leq \ell.
\end{equation}
Note that we are looking for a solution $h_{2}$ with $\lr{v}^{\ell} h_{2} \in L^{\mbq}([0,\infty],L^{\mbr}_xL^{\mbp}_v)$. 
In order to apply Proposition~\ref{P:Loss}, we choose $\ell_{3}=\ell$, so that 
\begin{equation}\label{L-2-3-chosen}
\ell_{3}=\ell>2\gamma+\frac{10}{9},\;\ell_{2}=\ell_{3}-(\gamma+\frac{1}{9})>1+\gamma. 
\end{equation}
 Thus $\lr{v}^{\ell_{2}}g_{1}\in  L^{q_{2}}([0,\infty],L^{r_{2}}_xL^{p_{2}}_v)$ 
by~\eqref{p2boundf+} and  if $\lr{v}^{\ell} h_{2} \in L^{\mbq}([0,\infty],L^{\mbr}_xL^{\mbp}_v)$, then we have 
\[
\|\lr{v}^{\ell_{2}} h_{2}L(g_1)\|_{L^{\tilde{q}'_{2}}_tL^{\tilde{r}'_{2}}_xL^{\tilde{p}'_{2}}_v} \leq C
\|\lr{v}^{\ell} h_{2}\|_{L^{\mbq}_tL^{\mbr}_xL^{\mbp}_v}\|\lr{v}^{\ell_{2}} g_1\|_{L^{q_{2}}_tL^{r_{2}}_xL^{p_{2}}_v},
\] 
we know that if $\phi\in L^{\tilde{q}_{2}}_{t}L^{\tilde{r}_{2}}_{x}L^{\tilde{p}_{2}}_{v}$ then $\lr{\lr{v}^{\ell_{2}}h_{2}L(g_{1}),\phi}$
is bounded. Since 
\[ 
\lr{ \lr{v}^{\ell_{2}} h_{2}L(g_{1}),\phi}=\lr{\lr{v}^{\ell_{2}} L(g_{1}),h_{2}\phi},
\]
we have $\lr{v}^{\ell_{2}} L(g_{1})\in L^{\mbq_{2}}_{t}L^{\mbr_{2}}_{x}L^{\mbp_{2}}_{v}$ where $(1/\mbq_{2})'=1/\tilde{q}_{2}+1/\mbq$
, $(1/\mbr_{2})'=1/\tilde{r}_{2}+1/\mbr$ and $(1/\mbp_{2})'=1/\tilde{p}_{2}+1/\mbp$. Therefore $\lr{v}^{\ell_{2}} L(g_{1})$ 
as well as $ L(g_{1})$  are pointwisely a.e. well-defined.

Now we can compute $h_2$ by~\eqref{h-2}. Clearly, we  have $h_1\equiv 0\leq h_2\leq U(t)f_0\leq g_2$ by 
the non-negativity of $g_1$.  Therefore we conclude
the beginning condition~\eqref{begin-condition}.  From~\eqref{p2boundf+}, we also have that
\[
\lr{v}^{\ell_{1}}  h_{2},\;\lr{v}^{\ell_{1}}  g_{2}\in L^{q_{2}}([0,\infty],L^{r_{2}}_xL^{p_{2}}_v), \;1+\gamma
<\ell_{1}\leq \ell.
\]
Thus we can repeat the above argument to check that each term in~\eqref{Iteration} is well-defined. 
Therefore the method of Kaniel-Shinbrot 
ensures the existence of monotone sequence~\eqref{monotone-sequence} and the limit functions 
$g,\;h$ satisfy~\eqref{Iteration-limit}.

We also note that from the monotone convergence theorem and~\eqref{monotone-sequence}, we have
\begin{equation}\label{g-h-finite}
\begin{split}
& \lr{v}^{\ell_{1}}  g,\;\lr{v}^{\ell_{1}}  h \in  L^{\mbq}([0,\infty],L^{\mbr}_xL^{\mbp}_v), \;1+\gamma<\ell_{1}\leq \ell.\\
& \lr{v}^{\ell_{1}}  g,\;\lr{v}^{\ell_{1}}  h \in  L^{q_{2}}([0,\infty],L^{r_{2}}_xL^{p_{2}}_v), \\
& \lr{v}^{\ell_{1}} Q^{+}(g,g), \;\lr{v}^{\ell_{1}}  Q^{+}(h,h)\in L^{\tmbq'}_tL^{\tmbr'}_xL^{\tmbp'}_{v}\cup 
L^{\tilde{q}'_{2}}_tL^{\tilde{r}'_{2}}_xL^{\tilde{p}'_{2}}_{v},\\
&\lr{v}^{\ell-(\gamma+\frac{1}{9})} Q^{-}(g,g),\; \lr{v}^{\ell-(\gamma+\frac{1}{9})} Q^{-}(h,h)
\in L^{\tilde{q}'_{2}}_tL^{\tilde{r}'_{2}}_xL^{\tilde{p}'_{2}}_{v}.
\end{split}
\end{equation}
Thus we have a solution for the full Boltzmann equation if $g=h$.  

To prove that $g=h$, we let $w=g-h\geq 0$.  By~\eqref{Iteration-limit} the difference $w$ satisfies the equation
\[
 \partial_t w+v\cdot\nabla_x w=Q^+(g,w)+Q^+(w,h)+Q^-(g,w)-Q^-(w,g)
\]
with zero initial data. By   Lemma~\ref{w-equation} below we know that this equation has a unique solution 
$w\equiv 0$.  Thus $g=h$ and we conclude the global existence of the  non-negative mild solution for the 
full Boltzmann equation. The uniqueness of this solution can be proved by a standard continuity argument 
and the fact that the solutions are non-negative. We include it in subsection~\ref{s-unique-sol}.  Also the continuity 
in time, scattering of the solution and Lipschitz continuous of the solution map is included in the subsection
~\ref{s-continuity}.
 Thus we conclude the Proposition~\ref{result2}.

\end{proof}

We need a lemma before we start the proof of the uniqueness. 
\begin{lem}\label{w-equation}
Let $g,\;h$ be non-negative functions satisfy~\eqref{g-h-finite}.  Suppose that $w\geq 0$ is a mild solution 
of 
\begin{equation}\label{g-h-equation}
\left\{
\begin{array}{l}
 \partial_t w+v\cdot\nabla_x w=Q^+(g,w)+Q^+(w,h)+Q^-(g,w)-Q^-(w,g) \\
 w(0)=0.
\end{array}
\right.
\end{equation}
Then $w\equiv 0$. 
\end{lem}
\begin{proof}

Consider the given time interval $[0,T]$ and define
\[
t_{0}=\inf {\Big \{} t\in [0,T]{\Big |} \|w(t)\|_{L^{q_{2}}([0, t],L^{r_{2}}_{x}L^{p_{2}}_{v})}>0 {\Big \}}.
\]
 Then $w\equiv 0$ for $0\leq t\leq t_{0}$. Let $t_{0}\leq s\leq T$.

From~\eqref{g-h-equation}, we have 
\[
 w(s)=\int_{0}^{s} U(t-\tau){\big [}Q^{+}(g,w)+Q^{+}(w,h)+Q^{-}(g,w)-Q^{-}(w,g){\big ]}(\tau) d\tau. 
\]
Noting that  $w\ge 0$,  $0\leq h\leq g$ and the operators $U(t), Q^{+}, Q^{-}$ are non-negative,  we have 
\begin{equation}\label{w-zero}
0\leq  w(s)\leq \int_{0}^{s} U(t-\tau){\big [}Q^{+}(g,w)+Q^{+}(w,h)+Q^{-}(g,w){\big ]}(\tau) d\tau. 
\end{equation}
Using~\eqref{L-2-3-chosen}, applyying Strichartz estimates as Proposition~\ref{P-p2-est-w} and also the estimates of Proposition~\ref{P:Loss}, then we have  
\begin{equation}\label{w-est}
\begin{split}
&\| \lr{v}^{\ell_{2}} w\|_{L^{q_{2}}([t_{0},s],L^{r_{2}}_xL^{p_{2}}_v)}\\
&\leq C( \| \lr{v}^{\ell} g\|_{L^{\mbq}([t_{0},s],L^{\mbr}_xL^{\mbp}_v)}+\\
&\hskip 3cm \| \lr{v}^{\ell_{2}} h\|_{L^{\mbq}([t_{0},s],L^{\mbr}_xL^{\mbp}_v)})
\| \lr{v}^{\ell_{2}} w\|_{L^{q_{2}}(([t_{0},s],L^{r_{2}}_xL^{p_{2}}_v)}\\
&:=C(g,h,s) \| \lr{v}^{\ell_{2}}  w\|_{L^{q_{2}}([t_{0},s],L^{r_{2}}_xL^{p_{2}}_v)}.
\end{split}
\end{equation}
Letting $s \rightarrow t_{0}+$, clearly we have $C(g,h,s)<1$ which means that 
\[
\| w\|_{L^{q_{2}}([t_{0},s],L^{r_{2}}_xL^{p_{2}}_v)}=0
\] for 
some $s> t_0$. By continuity, we have $w|_{[0,T]}\equiv 0$ for any $T>0$.
\end{proof}

\subsection{Uniqueness of the solution}\label{s-unique-sol}
Assume that $g\geq 0$ and $ h\geq 0$ both are mild solutions which satisfy~\eqref{E:Cauchy}. Let 
$w=g-h$. Comparing to Lemma~\ref{w-equation}, the function $w$ satisfies~\eqref{g-h-equation}, 
but we do not have the property $ w\geq 0$.  For our convenience, we rewrite~\eqref{g-h-equation} as
\[
\left\{
\begin{array}{l}
 \partial_t w+v\cdot\nabla_x w+wL(g)=Q^+(g,w)+Q^+(w,h)+Q^-(g,w) \\
 w(0)=0,
\end{array}
\right.
\]
and want to show $w\equiv 0$. Since $g,\;h$ and thus $w$ satisfy~\eqref{g-h-finite}, the term $L(g)$ is pointwisely
a.e. well-defined. Thus  the function $w$ satisfies 
\[
\begin{split}
& w(t)=\int_0^t  e^{-\int_s^t U(t-\tau)L(g)(\tau)d\tau} \\
&\hskip 3cm \cdot U(t-s){\big [} Q^+(g,w)+Q^+(w,h)+Q^-(g,w) {\big]}(s)ds.
\end{split}
\]
Using the  fact that $L(g)\geq 0$ since $g\geq 0$, we have 
\begin{equation}\label{w-uni}
| w(t) | \leq   \int_0^t   U(t-s){\big [} Q^+(g,|w|)+Q^+(|w|,h)+Q^-(g,|w|) {\big]}(s)ds.
\end{equation}
The equation~\eqref{w-uni} is in place of~\eqref{w-zero}  for the proof of $w\equiv 0$, thus we conclude 
the uniqueness of the solution.

\subsection{Continuity in $L^{3}_{x,v}$ and Scattering of the solution}\label{s-continuity}

Now we show that $\lr{v}^{\ell} f\in C([0,T],L^3_{x,v})$ for any $T\in[0,\infty]$. From the formula 
\begin{equation}\label{Duhamel-positive}
\int_{0}^{t} U(t-s)Q^{-}(f,f)(s)ds+f(t)=U(t)f_{0}+\int_{0}^{t} U(t-s)Q^{+}(f,f)(s)ds
\end{equation}
and the observation that each term in~\eqref{Duhamel-positive} is non-negative, we have 
\begin{equation}\label{Duhamel-ineq}
0\leq f(t)\leq U(t)f_{0}+\int_{0}^{t} U(t-s)Q^{+}(f,f)(s)ds.
\end{equation} 
It has been observed by Ovcharov~\cite{Ovc09} that  
$U(t)f_0\in C(\mathbb{R};L^N_{x,v})$ when $f_{0}\in L^N_{x,v}$. As we point out before that 
multiplication of weight $\lr{v}^{\ell}$ commutes with $U(t)$, thus we have 
$\lr{v}^{\ell} U(t) f_0\in C(\mathbb{R};L^3_{x,v})$.
Thus it suffices to show that $W(t)$ (see \eqref{D:solution-maps}) is also continuous. 
Let $0\leq t\in (0,\infty]$.  
Applying inhomogeneous Strichartz with triplet
$(\tmbq',\tmbr',\tmbp')$ used in~\eqref{solvable-triplets}, we see that 
\[
\| \lr{v}^{\ell} W(t)Q^{+}(f,f)\|_{L^{\infty}([0,t];L^3_{x,v})}=\int_0^t \|U(t-s) \lr{v}^{\ell} Q^{+}(f,f)\|_{L^3_{x,v}}ds
\]
is bounded.  Since $U(t)$ is continuous, we conclude that $W(t)$ is continuous from above expression. 
Also the solution map $\lr{v}^{\ell}  f_0\in B_{\eta}\subset L^3_{x,v}\rightarrow \lr{v}^{\ell}  f\in 
L^{\mbq}_tL^{\mbr}_xL^{\mbp}_v$ is Lipschitz continuous.

Next we want to show that the solution $f$ scatters, i.e., there exists  a function $f_{\infty}$ with 
 $\lr{v}^{\ell} f_{\infty}\in L^{3}_{x,v}$ such that 
\[
\|\lr{v}^{\ell} f(t)-U(t)(\lr{v}^{\ell} f_{\infty})\|_{L^3_{x,v}}\rightarrow 0\;\;{\rm as}\;t\rightarrow \infty.
\]
The above statement is equivalent to prove that 
\begin{equation}
\|U(-t)(\lr{v}^{\ell} f)(t)-\lr{v}^{\ell} f_{\infty}\|_{L^3_{x,v}}\rightarrow 0\;\;{\rm as}\;t\rightarrow \infty,
\end{equation}
since $U(t)$ preserves the $L^3_{x,v}$ norm. 

By the Duhamel formula, we have 
\begin{equation}\label{Duhamel}
U(-t)f(t)=f_0+\int_0^t U(-s)Q(f,f)(s)ds.
\end{equation}
Hence the scattering of $f(t)$ is confirmed if we have  the convergence of the integral 
\[
\lr{v}^{\ell} \int_0^{\infty} U(-t)Q(f,f)(t)dt
\]  
in $L^{3}_{x,v}$. In this case $f_{\infty}$ is given by
\begin{equation}\label{f-infty}
f_{\infty}=f_0+\int_0^{\infty} U(-t)Q(f,f)(t)dt.
\end{equation}

We rewrite~\eqref{Duhamel} as
\[
U(-t)f(t)+\int_0^t U(-s)Q^{-}(f,f)(s)ds=f_0+\int_0^t U(-s)Q^{+}(f,f)(s)ds.
\]
 Since each term of above equation is non-negative, thus we have 
 \[
 \begin{split}
 \int_0^t U(-s)Q^{-}(f,f)(s)ds& \leq f_0+\int_0^t U(-s)Q^{+}(f,f)(s)ds\\
 &\leq f_0+\int_0^{\infty} U(-s)Q^{+}(f,f)(s)ds.
\end{split}
\]
By monotone convergence theorem, it holds that
\begin{equation}\label{Duhamel-p}
 \int_0^{\infty} U(-s)Q^{-}(f,f)(s)ds\leq f_0+\int_0^{\infty} U(-s)Q^{+}(f,f)(s)ds.
\end{equation} 
Then we are reduced to prove the right hand side of~\eqref{Duhamel-p} is bounded 
in $L^{3}_{x,v}$.
 
 Let $U^*(t)$ be the adjoint operator of $U(t)$, it is clearly that $U^*(t)=U(-t)$. 
Let $(\tmbq,\tmbr,\tmbp)$ be the KT-admissible triplet chosen in ~\eqref{solvable-triplets}
and recall that $1/\tmbp'+1/\tmbr'=2/3$. 
By duality, the homogeneous Strichartz estimate 
\[
\| \lr{v}^{\ell}  U(t)g\|_{L^{\tmbq}_tL^{\tmbr}_xL^{\tmbp}_v}\leq C\| \lr{v}^{\ell}  g\|_{L^{{3}/{2}}_{x,v}}
\]
implies
\[
\begin{split}
&\bigg\|\int_0^{\infty} \lr{v}^{\ell} U^*(t)Q^+(f,f)dt \bigg\|_{L^{3}_{x,v}} \\
&\leq C\|\lr{v}^{\ell} Q^+(f,f)\|_{L^{\tmbq'}_tL^{\tmbr'}_xL^{\tmbp'}_v}
\leq C\|\lr{v}^{\ell} f\|^{2}_{L^{\mbq}_tL^{\mbr}_xL^{\mbp}_v} 
\end{split}
\]
as before. Thus we conclude that $\lr{v}^{\ell} f$ scatters.

\section{Some tools }

One of the most powerful tools in estimating the oscillatory integral
 $$I_{\Lambda}(u,f)=\int_{{\mathbb R}^n}e^{i\Lambda f(y)}u(y)dy,$$ for large $\Lambda$ is the 
following lemma of stationary phase asymptotics. There are several 
versions used widely, here we only record one of these which is 
from Theorem 7.7.5 of H\"{o}rmander~\cite{Hor83}. We use the notation $D_j=-i\partial_j$.
\begin{thm}[Stationary phase asymptotics]\label{L:stationary}
Let $K\subset{\mathbb R}^n$ be a compact set, $X$ an open neighborhood of $K$ and 
$k$ a positive number. 
If $u\in C^{2k}_{c}({\mathbb R}^n), f\in C^{3k+1}(X)$ and 
${\rm Im}\;f\geq 0$ in $X, {\rm Im}\;f(y_0)=0,
f'(y_0)=0,{\rm det}\;f''(y_0)\neq 0,f'\neq 0$ in $K\backslash\{{y_0}\}$ then
\begin{equation}\label{F:stationary-formula}
\begin{aligned}
|I&-e^{i\Lambda f(y_0)}({\rm det}(\Lambda f''(y_0)/2\pi i))^{-1/2}\sum_{j<k}\Lambda^{-j}L_ju|\\
  &\leq C\Lambda^{-k}\sum_{|\alpha|\leq 2k}\sup|D^{\alpha}u|\;,\;\Lambda>0
\end{aligned}
\end{equation}
Here $C$ is bounded when $f$ stays in a bounded set in $C^{3k+1}(X)$ and 
$|y-y_0|/|f'(y)|$ has a uniform bound. 
With $$g_{y_0}(y)=f(y)-f(y_0)-\langle f''(y_0)(y-y_0),y-y_0 \rangle/2$$ which 
vanish of third order at $x_0$ we have 
\begin{equation}\label{E:L-operator}
L_ju=\sum_{\nu-\mu=j}\;\;\sum_{2\nu\geq3\mu}i^{-j}2^{-\nu}\langle f''(y_0)^{-1}D,D\rangle^{\nu}
(g_{y_0}^{\mu}u)(y_0)/\mu!\nu!
\end{equation}
which is a differential operator of order $2j$ acting on $u$ at $y_0$. The coefficients are rational 
homogeneous functions of degree $-j$ in $f''(y_0),\cdots, f^{2j+2}(y_0)$ with denominator
$({\rm det}\; f''(y_0))^{3j}$. In every term the total number of derivatives of $u$ and $f''$ is at most $2j$. 
\end{thm}

We need the following Schur’s Lemma which can be found in~\cite{Zhao15}.
\begin{lem}\label{Schur}
Let $\mu$ and $\nu$ be two positive measures on the space $X$. Let
 $K(x,y)$ be  a nonnegative  measurable functions on $X\times X$. 
 Let $T$ be the integral operator with kernel $K$, as   
\[
T(f)(x)=\int_Y K(x,y) f(y) d\nu(y)
\]
Suppose $1<p\leq q<\infty$ and $\frac{1}{p}+\frac{1}{p'}=1$. Let $\delta_{1}$ and $\delta_{2}$
be two real numbers such that $\delta_{1}+\delta_{2}=1$. If there exist positive functions 
$h_{1}$ and $h_{2}$ with positive constants $C_{1}$ and $C_{1}$ such that 
\[
\int_X K(x,y)^{\delta_{1}p'} h_{1}(y)^{p'} d\mu(y) \leq C_{1} h_{2}(x)^{p'}
\]
for almost all $x\in X$, and 
\[
\int_X K(x,y)^{\delta_{2} q} h_{2}(x)^{q} d\mu(y) \leq C_{2} h_{1}(y)^{q}
\]
Then $T$ maps $L^p(Y)$ to $L^q(X)$ with norm at most $C_{1}^{1/p'} C_{2}^{1/q}$.
\end{lem}

\noindent{\bf Acknowledgments.}
Ling-Bing He is supported by NSF of China under the grant 12141102. 
Jin-Cheng Jiang and Meng-Hao Liang were supported in part by  National Sci-Tech Grant 
MOST 109-2115-M-007-002-MY3 and NSTC 112-2115-M-007-005-MY3.
Hung-Wen Kuo was supported in part by  NSTC 112-2115-M-006-009.

\end{document}